\documentclass[11pt]{article}
\pdfoutput=1

\usepackage[dvipsnames]{xcolor}

\usepackage[framemethod=default]{mdframed}
\usepackage{caption}
\usepackage{indentfirst}
\usepackage{bm, mathrsfs, graphics,float,amssymb,amsmath,subeqnarray,setspace,graphicx,amsthm,epstopdf,subfigure, enumerate, color}
\usepackage[utf8]{inputenc}

\usepackage{fullpage}
\numberwithin{equation}{section}

\newtheorem{lemma}{Lemma}
\newtheorem{theorem}{Theorem}
\newtheorem{proposition}{Proposition}

\newtheorem{corollary}[theorem]{Corollary}
\newtheorem{remark}{Remark}

\ifx\assumption\undefined
\newtheorem{assumption}{Assumption}
\fi

\newcommand{\x}{\mathbf{x}}
\newcommand{\y}{\mathbf{y}}
\newcommand{\z}{\mathbf{z}}
\newcommand{\w}{\mathbf{w}}
\newcommand{\q}{\mathbf{q}}
\newcommand{\s}{\mathbf{s}}
\newcommand{\g}{\mathbf{g}}
\newcommand{\p}{\mathbf{p}}
\newcommand{\0}{\mathbf{0}}
\newcommand{\1}{\mathbf{1}}
\newcommand{\uu}{\mathbf{u}}
\renewcommand{\c}{\mathbf{c}}
\newcommand{\vv}{\mathbf{v}}
\renewcommand{\b}{\mathbf{b}}

\newcommand{\A}{\bm{\mathit{A}}}
\newcommand{\B}{\bm{\mathit{B}}}
\newcommand{\J}{\bm{\mathit{J}}}

\renewcommand{\S}{\bm{\mathit{S}}}

\newcommand{\E}{\bm{\mathbb{E}}}
\newcommand{\R}{\mathbb{R}}
\newcommand{\IS}{\mathbb{I}}

\newcommand{\DS}{\mathbb{D}}

\newcommand{\<}{\left\langle}
\renewcommand{\>}{\right\rangle}

\DeclareMathOperator*{\argmin}{argmin}

\usepackage{bbm}
\usepackage{multirow}
\usepackage{tablefootnote}
\usepackage{colortbl}
\usepackage{hhline}

\usepackage{algorithm}
\usepackage{algorithmic}
\usepackage{mathtools}

\begin{document}

\title{
On the Complexity Analysis of the Primal Solutions for the Accelerated Randomized Dual Coordinate Ascent
}

\author{
Huan Li
\thanks{
Nanjing University of Aeronautics and Astronautics;
email: lihuanss@nuaa.edu.cn;
This work was done when Huan Li was a Ph.D student at Peking University;
}
\qquad
Zhouchen Lin
\thanks{
Peking University;
email: zlin@pku.edu.cn;
}
}

\maketitle

\begin{abstract}
  Dual first-order methods are essential techniques for large-scale constrained convex optimization. However, when recovering the primal solutions, 
  we need $T(\epsilon^{-2})$ iterations to achieve an $\epsilon$-optimal primal solution when we apply an algorithm to the non-strongly convex dual problem with $T(\epsilon^{-1})$ iterations to achieve an $\epsilon$-optimal dual solution, where $T(x)$ can be $x$ or $\sqrt{x}$. In this paper, we prove that the iteration complexity of the primal solutions and dual solutions have the same $O\left(\frac{1}{\sqrt{\epsilon}}\right)$ order of magnitude for the accelerated randomized dual coordinate ascent. When the dual function further satisfies the quadratic functional growth condition, by restarting the algorithm at any period, we establish the linear iteration complexity for both the primal solutions and dual solutions even if the condition number is unknown. When applied to the regularized empirical risk minimization problem, we prove the iteration complexity of $O\left(n\log n+\sqrt{\frac{n}{\epsilon}}\right)$ in both primal space and dual space, where $n$ is the number of samples. Our result takes out the $\left(\log \frac{1}{\epsilon}\right)$ factor compared with the methods based on smoothing/regularization or Catalyst reduction. As far as we know, this is the first time that the optimal $O\left(\sqrt{\frac{n}{\epsilon}}\right)$ iteration complexity in the primal space is established for the dual coordinate ascent based stochastic algorithms. We also establish the accelerated linear complexity for some problems with nonsmooth loss, i.e., the least absolute deviation and SVM.
\end{abstract}

\section{Introduction}
In this paper, we study the following structured constrained convex optimization problem:
\begin{eqnarray}
\begin{aligned}\label{problem}
&\min_{\x\in \R^t}\quad F(\x)\equiv f(\x)+\frac{1}{n}\sum_{i=1}^n\phi_i(\A_i^T\x),\\
&s.t. \hspace*{0.62cm}\B\x+\b=\0,\\
&\hspace*{1.15cm} g_i(\x)\leq 0,\quad i=1,\cdots,m,
\end{aligned}
\end{eqnarray}
where $\A\in\R^{t\times n}$, $\B\in\R^{p\times t}$, each $\phi_i$ and $g_i$ is convex and $f$ is $\mu$-strongly convex. Both $f$ and $\phi_i$ can be non-differentiable. In machine learning, each column of $\A$ often represents a data point. $\phi_i$ is often the loss function, e.g., $\phi_i(y)=|y|$ for the absolute deviation and $\phi_i(y)=\max\{0,1-\l_iy\}$ for SVM, where $\l_i\in\{\pm1\}$ is the label for the $i$-th data. $f$ is often the regularizer, e.g., the $L_2$ regularization $f(\x)=\|\x\|_2^2$ and $L_1$-$L_2$ regularization $f(\x)=\|\x\|_2^2+\sigma\|\x\|_1$. Problem (\ref{problem}) is actually very general to incorporate many existing problems in machine learning. When dropping the constraints, problem (\ref{problem}) becomes the regularized empirical risk minimization (ERM) problem associated with linear predictors:
\begin{eqnarray}
\begin{aligned}\label{problem2}
&&\min_{\x\in \R^t}\quad F(\x)\equiv f(\x)+\frac{1}{n}\sum_{i=1}^n\phi_i(\A_i^T\x).
\end{aligned}
\end{eqnarray}
The ERM problem is widely used in machine learning. Please see \cite{zhang-2013-jmlr,zhang-2015-MP} for examples.

Due to the complicated constraints, people often do not solve problem (\ref{problem}) directly. Instead, they solve its dual problem by introducing the Lagrangian function. Many first-order methods can be used to solve the dual problem, e.g., the dual full gradient ascent (DFGA) \cite{Tseng-1990}, the accelerated DFGA (ADFGA) \cite{Beck-2014,ma-2013}, the randomized dual coordinate ascent (RDCA) \cite{Nesterov-2012,lu-2015-MP,richtarik-2014,zhang-2013-jmlr} and the accelerated RDCA (ARDCA) \cite{Nesterov-2012,fercoq-2015-siam,xiao-2015-siam,zhang-2015-MP}\footnote{Although the algorithm studied in this paper is a special case of APCG in \cite{xiao-2015-siam} and APPROX in \cite{fercoq-2015-siam}, we name it ARDCA to emphasize the application to the dual problem.}. They need $O\left(\frac{1}{\epsilon}\right)$, $O\left(\frac{1}{\sqrt{\epsilon}}\right)$, $O\left(\frac{\hat n}{\epsilon}\right)$ and $O\left(\frac{\hat n}{\sqrt{\epsilon}}\right)$ iterations to achieve an $\epsilon$-optimal dual solution, respectively, where $\hat n$ is the dimension in the dual space. At each iteration, RDCA and ARDCA choose one coordinate to sufficiently increase the dual objective value while keeping the others fixed. The cost at each iteration of RDCA and ARDCA may be much lower than that of DFGA and ADFGA. Since both $f$ and $\phi_i$ can be non-differentiable, the dual function is non-strongly convex. So only the sublinear complexity can be obtained.

It is not satisfactory to establish the iteration complexity only in the dual space. We should recover the primal solutions from the dual iterates and need to estimate how quickly the primal solutions converge. Unfortunately, \cite{Lu-2016-siam} established the algorithm independent result that the iteration complexity in the primal space is worse than that in the dual space. Specifically, \cite{Lu-2016-siam} studied the following problem, which is a special case of problem (\ref{problem}),
\begin{eqnarray}
\begin{aligned}\label{problem1}
&\min_{\x\in \R^t}\quad f(\x),\\
&s.t. \hspace*{0.62cm}\B\x+\b=\0,\\
&\hspace*{1.15cm} g_i(\x)\leq 0,\quad i=1,\cdots,m.
\end{aligned}
\end{eqnarray}
For a pair of approximate primal-dual solution $\{\x^*(\uu^K),\uu^K\}$\footnote{$\x^*(\uu)$ is recovered form $\uu$ and will be defined in (\ref{xu_solution}) later.}, the precision between $\x^*(\uu^K)$ and $\uu^K$ satisfies
\begin{eqnarray}
\begin{aligned}\notag
&|f(\x^*(\uu^K))-f(\x^*)|\leq O\left( \sqrt{D(\uu^K)-D(\uu^*)}+D(\uu^K)-D(\uu^*) \right),\\
&\left\| \left[
  \begin{array}{c}
    \B\x^*(\uu^K)+\b\\
    \max\left\{0,g(\x^*(\uu^K))\right\}
  \end{array}
\right]\right\|\leq O\left(\sqrt{D(\uu^K)-D(\uu^*)}\right),
\end{aligned}
\end{eqnarray}
where $D(\uu)$ is the negative of the dual function and $(\x^*,\uu^*)$ is a pair of optimal primal-dual solution. Thus if some algorithm achieves an $\epsilon$-optimal dual solution\footnote{We define an $\epsilon$-optimal dual solution as $D(\uu)-D(\uu^*)\leq\epsilon$.} of $\uu^K$ after $T(\epsilon^{-1})$ iterations, it only achieves an $\sqrt{\epsilon}$-optimal primal solution\footnote{We define an $\epsilon$-optimal primal solution as $|F(\x)-F(\x^*)|\leq\epsilon$ and $\left\| \left[\hspace*{-0.15cm}
  \begin{array}{c}
    \B\x+\b\\
    \max\left\{0,g(\x)\right\}
  \end{array}\hspace*{-0.15cm}
\right]\right\|\leq\epsilon$. $\|\cdot\|$ can be a general norm.} of $\x^*(\uu^K)$ after the same time. \cite{Lu-2016-siam} studied DFGA and ADFGA and established the $O\left(\frac{1}{\epsilon^2}\right)$ and $O\left(\frac{1}{\epsilon}\right)$ iteration complexity in the primal space to find an $\epsilon$-optimal primal solution. \cite{Dunner-2016-icml} proved the similar algorithm independent results for problem (\ref{problem2}). \cite{Kim-2016} proved the $O\left(\frac{1}{\epsilon^{2/3}}\right)$ iteration complexity to achieve an $\epsilon$-optimal primal solution for the deterministic accelerated full gradient methods for problem (\ref{problem2}).

Some researchers used regularization/smoothing to improve the iteration complexity of the primal solutions. They added a small regularization term $\epsilon\|\uu\|^2$ to the dual function to smooth the primal objective and solved a regularized dual problem by some algorithm with linear convergence rate. \cite{Devolder-2012-siam} applied ADFGA to a smoothed problem of (\ref{problem1}) and \cite{Necoara-2016} used ADFGA to solve a regularized dual problem of conic convex programming. However, they established the suboptimal iteration complexity of $O\left(\frac{1}{\sqrt{\epsilon}}\log \frac{1}{\epsilon}\right)$  to achieve an $\epsilon$-optimal primal solution recovered from the last dual iterate, which has an additional $\left(\log\frac{1}{\epsilon}\right)$ factor. The drawback of this strategy in practice is that it needs to choose the parameter $\epsilon$ in advance, which is related to the target accuracy. It is desirable to develop direct support for problems with non-smooth primal objective or non-strongly convex dual objective.

Other researchers improved the iteration complexities of DFGA and ADFGA in the primal space via averaging the primal solutions appropriately. \cite{Tseng-2008} studied the problem of $\min_{\x}\max_{\vv}\psi(\x,\vv)+P(\x)$ and established the $O\left(\frac{1}{\sqrt{\epsilon}}\right)$ iteration complexity measured by the duality gap for the accelerated full gradient method. \cite{Necoara-2014-ieee} and \cite{Patrinos-2013-ieee} used \cite{Tseng-2008}'s result for ADFGA to solve the embedded linear model predictive control problem, which is a special case of problem (\ref{problem1}). \cite{Necoara-2016} proved the $O\left(\frac{1}{\epsilon}\right)$ iteration complexity for DFGA and $O\left(\frac{1}{\sqrt{\epsilon}}\right)$ iteration complexity for ADFGA to achieve an $\epsilon$-optimal averaged primal solution for conic convex programming. None of them studied the general problem (\ref{problem}) and none of them studied the methods based on randomized dual coordinate ascent.

The randomized coordinate descent and its accelerated version have received extensive attention recently for solving large-scale optimization problems since it can break down the problem into smaller pieces. \cite{zhang-2013-jmlr} showed that the Stochastic Dual Coordinate Ascent (SDCA) needs $O\left(n\log n+\frac{1}{\epsilon}\right)$ iterations to reach an $\epsilon$-optimal solution in both the primal space and dual space for problem (\ref{problem2}). \cite{zhang-2015-MP} then developed an accelerated SDCA (ASDCA) and attained the suboptimal $O\left(\left(n+\sqrt{\frac{n}{\epsilon}}\right)\log\frac{1}{\epsilon}\right)$ iteration complexity to achieve an $\epsilon$-optimal primal solution via solving a regularized dual problem, which has the additional $\left(\log\frac{1}{\epsilon}\right)$ factor due to the smoothing/regularization technique. Catalyst \cite{lin-2015-nips}, a general scheme for accelerating first-order optimization methods, also yields the additional $\left(\log\frac{1}{\epsilon}\right)$ factor. The Accelerated randomized Proximal Coordinate Gradient (APCG) method \cite{xiao-2015-siam} is another famous method for problem (\ref{problem2}), which needs $O\left(\frac{n}{\sqrt{\epsilon}}\right)$ iterations to find a dual solution in $\epsilon$ accuracy. However, the sublinear complexity in the primal space is not established in \cite{xiao-2015-siam}. \cite{zhang-2017-jmlr} proposed a Stochastic Primal-Dual Coordinate method (SPDC) and \cite{lan-2017-MP} proposed a Randomized Primal-Dual Gradient method (RPDG) for problem (\ref{problem2}). They smoothed $\phi_i$ and achieved the $O\left(\left(n+\sqrt{\frac{n}{\epsilon}}\right)\log\frac{1}{\epsilon}\right)$ iteration complexity. When $\phi_i$ has $\frac{1}{\gamma}$-Lipschitz continuous gradient, ASDCA, APCG, SPDC and RPDG all have the accelerated linear complexity of $O\left(\left(n+\sqrt{\frac{n}{\gamma\mu}}\right)\log\frac{1}{\epsilon}\right)$.


\subsection{Contributions}
In this paper, we study the iteration complexity of the primal solutions when using ARDCA to solve the non-strongly convex dual problem. Specifically, we aim to prove that the complexity of the primal solutions has the same order of magnitude as that of the dual solutions.

For the general problem (\ref{problem}), when applying ARDCA to solve its dual problem, we prove the $O\left(\frac{\hat n}{\sqrt{\epsilon}}\right)$ iteration complexity of the primal solutions simply by averaging the last few primal iterates appropriately. This complexity has the same order of magnitude as that of the dual solutions and thus improves the theoretical results in \cite{Lu-2016-siam,Dunner-2016-icml}. As a comparison, literature \cite{Tseng-2008,Necoara-2014-ieee,Patrinos-2013-ieee,Necoara-2016} only studied ADFGA, which is much simpler than the analysis of ARDCA. Since we use ARDCA to solve the dual problem directly, rather than a regularized dual problem or a smoothed primal problem, our result takes out the $\left(\log\frac{1}{\epsilon}\right)$ factor compared with the smoothing/regularization based methods.

When the dual function satisfies the quadratic functional growth condition, by restarting ARDCA at any period, we prove the linear iteration complexity for both the primal solutions and dual solutions. Moreover, our analysis does not require the parameters of the algorithm depend on the condition number $\kappa$, which will be defined in Assumption \ref{assumption2} later and it is often difficult to estimate in practice. We show that ARDCA with restart outperforms RCDA for a wide range of inner iteration numbers and the optimal $O\left(\left(\hat n+\frac{\hat n}{\sqrt{\kappa}}\right)\log\frac{1}{\epsilon}\right)$ complexity can be attained when the inner iteration number is equal to $O\left(\hat n+\frac{\hat n}{\sqrt{\kappa}}\right)$. The difference with respect to \cite{zhengqu-2018} is that our analysis does not require the uniqueness of the optimal dual solution.

When applied to problem (\ref{problem2}), our work extends the theoretical results of \cite{xiao-2015-siam} and improves those of \cite{zhang-2015-MP}. We prove that ARDCA needs $O\left(n\log n+\sqrt{\frac{n}{\epsilon}}\right)$ iterations to find an $\epsilon$-optimal solution in both the primal space and dual space, while \cite{xiao-2015-siam} only proved the iteration complexity in the dual space. This complexity matches the theoretical lower bound \cite{Woodworth-2016} and state-of-the-art upper bound \cite{zhu-2017-stoc}. Our theory outperforms ASDCA \cite{zhang-2015-MP} and Catalyst \cite{lin-2015-nips} by the factor of $\left(\log \frac{1}{\epsilon}\right)$. As far as we know, we are the first to establish the optimal $O\left(\sqrt{\frac{n}{\epsilon}}\right)$ complexity in the primal space for the dual coordinate ascent based stochastic algorithms. When $\phi_i$ has $\frac{1}{\gamma}$-Lipschitz continuous gradient, ARDCA with restart has the optimal $O\left(\left(n+\sqrt{\frac{n}{\gamma\mu}}\right)\log\frac{1}{\epsilon}\right)$ complexity. Moreover, we establish the accelerated linear complexity of ARDCA with restart for some special problems with nonsmooth $\phi_i$, e.g., the least absolute deviation problem and support vector machine (SVM).

\subsection{Assumption, Notation and Problem Formulation}
\subsubsection{Assumption}
We study problem (\ref{problem}) under the following assumptions:
\begin{assumption}\label{assumption}
\item 1. $f$ is $\mu$-strongly convex over $\R^t$, i.e., $f(\y)\geq f(\x)+\<\s,\y-\x\>+\frac{\mu}{2}\|\y-\x\|^2,\forall \x,\y$, for every subgradient $\s\in\partial f(\x)$.

\item 2. $\phi_i$ is convex and $M$-Lipschitz continuous over $\R$, i.e., $|\phi_i(x)-\phi_i(y)|\leq M |x-y|,\forall x,y$.

\item 3. $g_i$ is convex and has bounded subgradients over $\R^t$, i.e., $\|\s\|\leq L_{g_i}$, $\forall \s\in\partial g_i(\x)$.

\item 4. There exists $\overline\x$ such that $g_i(\overline\x)<0$ and $\B\overline\x+\b=\0$.

\item 5. The optimal objective value of problem (\ref{problem}) is finite.
\end{assumption}

Assumption \ref{assumption}.4 is the Slater's condition. Assumption \ref{assumption}.4 and \ref{assumption}.5 ensure that the strong duality holds, i.e., the dual optimal value is equal to the primal optimal value \cite{Bertsekas-book}. Assumptions \ref{assumption}.1 and \ref{assumption}.3 will be used to establish the Lipschitz smoothness of part of the dual function in Lemma \ref{coordinate_LS}. Assumption \ref{assumption}.2 will be used to get the complexity for problem (\ref{problem}) from that of the reformulated problem (\ref{problem3}).

To make each iteration of the randomized dual coordinate ascent computationally efficient, we only consider the case that $g(\x)$ is a linear function for simplicity, i.e.,
\begin{eqnarray}
g(\x)=\J\x+\q,\label{function_g}
\end{eqnarray}
where $\J\in\R^{m\times t}$ and $\q\in\R^{m}$. However, the analysis in this paper suits for the general function $g(\x)$ satisfying Assumption \ref{assumption}.3.

In Section \ref{linar_sec}, we will prove the linear complexity of ARDCA with restart under the quadratic functional growth condition \cite{Necoara-2019,ma2016eb}. This condition is equivalent to the error bound condition in \cite{errorbound2,lewis2018} and is satisfied for broad applications in machine learning, e.g., the least absolute deviation and SVM \cite{lin2014}.
\begin{assumption}\label{assumption2}
$ D(\uu)$ satisfies the quadratic functional growth condition with respect to the norm $\|\cdot\|_L$, i.e., $\kappa\|\uu-\mbox{Proj}_{\DS^*}(\uu)\|_L^2\leq D(\uu)- D(\uu^*),\forall \uu$, where $\kappa>0$ is the condition number, $\uu^*$ is the optimal dual solution, $\mbox{Proj}_{\DS^*}(\uu)$ is the projection of $\uu$ onto the optimal dual solution set $\DS^*$ and $ D(\uu)$ is the negative of the dual function.
\end{assumption}

\subsubsection{Notation}
Lowercase bold letters $\uu,\vv,\z,\x,\y$ represent vectors, uppercase bold letters $\A,\B$ represent matrices and non-bold letters $\theta,\alpha$ represent scalars. Let $\R_+^m$ be the set of nonnegative vectors in $\R^m$ and $\hat n=n+p+m$ be the dimension of the dual variable. Denote $\uu_i$ and $\nabla_i d(\uu)$ as the $i$-th element of $\uu$ and $\nabla d(\uu)$, respectively. Let $\uu_{i:j}$ and $g_{1:m}(\x)$ be the vectors consisting of $\uu_i,\cdots,\uu_j$ and $g_1(\x),\cdots,g_m(\x)$, respectively. The value of $\uu$ at iteration $k$ is denoted by $\uu^k$. For scalars, e.g., $\theta$, $\theta_k$ represents its value at iteration $k$ and $\theta^2$ denotes its squares. $\A_i\in\R^t$ and $\A_{j,:}\in\R^n$ are the $i$-th column and $j$-th row of $\A$, respectively. We use $\|\cdot\|$ as the $l_2$ Euclidean norm and $\|\cdot\|_{\infty}$ as the infinite norm for a vector. Define the weighted norm $\|\x\|_L=\sqrt{\sum_iL_i\|\x_i\|^2}$ and its dual norm $\|\x\|_L^*=\sqrt{\sum_i\frac{1}{L_i}\|\x_i\|^2}$. For any matrix $\A$, $\|\A\|_2=\sigma_{\max}(\A)$ is the largest singular value of $\A$. $\lfloor x \rfloor$ ($\lceil x\rceil$) means the largest (smallest) integer less (larger) than or equal to $x$. For a function $\phi_i$, we use $\phi_i^*(u)=\sup_{v}\<u,v\>-\phi_i(v)$ to denote its conjugate and $\mbox{Prox}_{\phi_i}(v)=\argmin_{u} \phi_i(u)+\frac{1}{2}||u-v\|^2$ to denote its proximal mapping. Define $\phi(\y)=\sum_{i=1}^n\phi_i(\y_i)$. 

\subsubsection{Problem Formulation}
Reformulate problem (\ref{problem}) as
\begin{eqnarray}
\begin{aligned}\label{problem3}
&\min_{\x\in \R^t,\y\in\R^n}\quad f(\x)+\frac{1}{n}\sum_{i=1}^n\phi_i(\y_i),\\
&s.t. \hspace*{1.42cm}\frac{1}{n}(\A^T\x-\y)=\0,\\
&\hspace*{1.9cm} \B\x+\b=\0,\\
&\hspace*{1.93cm} g_i(\x)\leq 0,\quad i=1,\cdots,m,
\end{aligned}
\end{eqnarray}
and introduce the Lagrangian function as
\begin{eqnarray}\label{lagrangian}
L_F(\x,\y,\uu)\hspace*{-0.08cm}=\hspace*{-0.08cm}f(\x)\hspace*{-0.08cm}+\hspace*{-0.08cm}\frac{1}{n}\hspace*{-0.08cm}\sum_{i=1}^n\hspace*{-0.08cm}\phi_i(\y_i)\hspace*{-0.08cm}+\hspace*{-0.08cm}\frac{1}{n}\hspace*{-0.08cm}\<\uu_{1:n},\A^T\x\hspace*{-0.08cm}-\hspace*{-0.08cm}\y\>\hspace*{-0.08cm}+\hspace*{-0.08cm}\<\uu_{n+1:n+p},\B\x\hspace*{-0.08cm}+\hspace*{-0.08cm}\b\>\hspace*{-0.08cm}+\hspace*{-0.08cm}\sum_{i=1}^m \hspace*{-0.08cm}\uu_{n+p+i} g_i(\x),\hspace*{-0.2cm}
\end{eqnarray}
where $\uu\in\R^{\hat n}$ is the vector of the Lagrange multipliers. Define the dual feasible set as $\DS=\{\uu\in\R^{\hat n}:\uu_{n+p+1:n+p+m}\in \R_+^m\}$ and denote $\DS^*$ to be the optimal dual solution set. Then the Lagrange dual problem of (\ref{problem3}) can be expressed as
\begin{eqnarray}
\begin{aligned}
&\max_{\uu\in\DS}\min_{\x\in\R^t,\y\in\R^{n}} L_F(\x,\y,\uu)\notag\\
=&\max_{\uu\in\DS}\hspace*{-0.13cm}\left(\hspace*{-0.07cm}\min_{\x\in\R^t}\hspace*{-0.13cm}\left(\hspace*{-0.15cm}f\hspace*{-0.03cm}(\hspace*{-0.03cm}\x\hspace*{-0.03cm})\hspace*{-0.13cm}+\hspace*{-0.13cm}\<\hspace*{-0.03cm}\uu_{1:n},\hspace*{-0.06cm}\A^T\hspace*{-0.06cm}\x\hspace*{-0.03cm}/\hspace*{-0.03cm}n\hspace*{-0.03cm}\>\hspace*{-0.13cm}+\hspace*{-0.13cm}\<\hspace*{-0.03cm}\uu_{n\hspace*{-0.03cm}+\hspace*{-0.03cm}1:n\hspace*{-0.03cm}+\hspace*{-0.03cm}p},\hspace*{-0.06cm}\B\x\hspace*{-0.13cm}+\hspace*{-0.13cm}\b\hspace*{-0.03cm}\>\hspace*{-0.13cm}+\hspace*{-0.15cm}\sum_{i=1}^m \hspace*{-0.1cm} \uu_{n\hspace*{-0.03cm}+\hspace*{-0.03cm}p\hspace*{-0.03cm}+\hspace*{-0.03cm}i} g_i\hspace*{-0.03cm}(\hspace*{-0.03cm}\x\hspace*{-0.03cm})\hspace*{-0.15cm}\right)\hspace*{-0.13cm}+\hspace*{-0.13cm}\min_{\y\in\R^n}\hspace*{-0.07cm}\frac{1}{n}\hspace*{-0.13cm}\left(\hspace*{-0.07cm}\sum_{i=1}^n\hspace*{-0.1cm}\phi_i\hspace*{-0.03cm}(\hspace*{-0.03cm}\y_i\hspace*{-0.03cm})\hspace*{-0.13cm}-\hspace*{-0.13cm}\<\hspace*{-0.03cm}\uu_{1:n},\hspace*{-0.03cm}\y\hspace*{-0.03cm}\>\hspace*{-0.13cm}\right)\hspace*{-0.20cm}\right)\notag\\
=&\max_{\uu\in\DS} \left(L_f(\x^*(\uu),\uu)-\frac{1}{n}\sum_{i=1}^n\phi_i^*(\uu_{i})\right),\notag
\end{aligned}
\end{eqnarray}
where
\begin{eqnarray}
\begin{aligned}\label{xu_solution}
&L_f(\x,\uu)=f(\x)+\<\uu_{1:n},\A^T\x/n\>+\<\uu_{n+1:n+p},\B\x+\b\>+\sum_{i=1}^m \uu_{n+p+i} g_i(\x),\\
&\x^*(\uu)=\argmin_{\x\in\R^t} L_f(\x,\uu).
\end{aligned}
\end{eqnarray}
Define
\begin{eqnarray}
d(\uu)=-L_f(\x^*(\uu),\uu)\quad\mbox{and}\quad
h_i(\uu_i)=\left\{
  \begin{array}{lcl}
    \frac{1}{n}\phi_{i}^*(\uu_i),&& i=1,\cdots,n,\\
    0, && n+1\leq i\leq n+p,\\
    I_{u\geq 0}(\uu_i), && i> n+p,
  \end{array}
\right.\label{cont28}
\end{eqnarray}
where $I_{u\geq 0}(u)=\left\{
  \begin{array}{ll}
    0, & \mbox{if }u\geq 0,\\
    \infty, & \mbox{otherwise}.
  \end{array}
\right.$ Then we can rewrite the Lagrange dual problem as
\begin{eqnarray}\label{dual_problem}
\min_{\uu\in\R^{\hat n}}  D(\uu)=d(\uu)+h(\uu),
\end{eqnarray}
where we define $h(\uu)=\sum_{i=1}^{\hat n} h_i(\uu_{i})$ and $D(\uu)$ means the negative of the dual function. We study the negative of the dual function, rather than the dual function directly, since $D(\uu)$ is convex. Let $(\x^*,\y^*)$ and $\uu^*$ be the optimal primal solution and dual solution of problem (\ref{problem3}), respectively. Then they satisfy the KKT condition \cite{Bertsekas-book}. Since the strong duality holds, we have $f(\x^*)+\frac{1}{n}\phi(\y^*)=-D(\uu^*)$. Since $L_f(\x,\uu)$ is strongly convex over $\x$ for every $\uu\in \DS$, then $\x^*(\uu)$ is unique. Due to Danskin's theorem \cite{Bertsekas-book} we know that $d(\uu)$ is convex, differentiable and
\begin{eqnarray}
\nabla d(\uu)=-\left[\left(\A^T\x^*(\uu)/n\right)^T,(\B\x^*(\uu)+\b)^T,g_1(\x^*(\uu)),\cdots,g_m(\x^*(\uu))\right]^T.\label{gradient}
\end{eqnarray}
From Proposition 3.3 in \cite{Lu-2016-siam}, we have a Lipschitz smooth condition\footnote{\cite{Lu-2016-siam} studied the projected gradient method under the local Lipschitz smooth condition and the fast gradient method under the global Lipschitz smooth condition. The former condition is ensured by replacing Assumption \ref{assumption}.3 with $|g_i(\x)-g_i(\y)|\leq L_{g_i}\|\x-\y\|$ and a further assumption that $g_i$ is differentiable. In this paper, we use the global Lipschitz smooth condition over the dual feasible set $\DS$ for simplicity.} of $\|\nabla d(\uu)-\nabla d(\vv)\|\leq L\|\uu-\vv\|,\forall \uu,\vv\in \DS$, where
\begin{eqnarray}\label{L_constant}
L=\frac{\sqrt{m+1}\max\{\|[\A^T/n,\B]\|_2,\max_i L_{g_i}\}}{\mu}\sqrt{\|[\A^T/n,\B]\|_2^2+\sum_{i=1}^m L_{g_i}^2}.
\end{eqnarray}
Similarly, we can also prove a coordinatewise Lipschitz smooth condition in the following lemma, whose proof is given in Appendix B.
\begin{lemma}\label{coordinate_LS}
For any $\uu,\vv\in \DS$ and any $j$, assume that $\uu_i=\vv_i,\forall i\neq j$. Then $\|\nabla_jd(\uu)-\nabla_jd(\vv)\|\leq L_j\|\uu-\vv\|$, where
\begin{equation}\label{CL_constant}
L_j=\left\{
  \begin{array}{lcl}
    \frac{\|\A_j\|^2}{n^2\mu}, && j\leq n,\\
    \frac{\|\B_{j-n,:}\|^2}{\mu}, && n<j\leq n+p,\\
    \frac{L_{g_{j-n-p}}^2}{\mu}, && j>n+p.
  \end{array}
\right.
\end{equation}
\end{lemma}
An immediate consequence of Lemma \ref{coordinate_LS} is \cite[Lemma 1.2.3]{Nesterov-2004}
\begin{eqnarray}
|d(\uu)-d(\vv)-\<\nabla_j d(\vv),\uu_j-\vv_j\>|\leq \frac{L_j}{2}\|\uu_j-\vv_j\|^2\label{C_Lipschitz_smooth}
\end{eqnarray}
for all $\uu,\vv\in\DS$ satisfying $\uu_i=\vv_i,\forall i\neq j$.
\section{Accelerated Randomized Dual Coordinate Ascent}\label{sublinear_sec}

In this section, we use the standard accelerated randomized coordinate descent \cite{fercoq-2015-siam,xiao-2015-siam} to solve the dual problem (\ref{dual_problem}), which consists of the following steps at each iteration:
\begin{eqnarray}
&&\vv^k=\theta_k\z^k+(1-\theta_k)\uu^k,\notag\\
&&\mbox{select }i_k\mbox{ randomly with probability of }1/\hat n,\notag\\
&&\z_{i_k}^{k+1}=\argmin_{u} \frac{\hat n\theta_kL_{i_k}}{2}\|u-\z_{i_k}^k\|^2+\<\nabla_{i_k}d(\vv^k),u-\z_{i_k}^k\>+h_{i_k}(u),\notag\\
&&\z_j^{k+1}=\z_j^k,\forall j\neq i_k,\notag\\
&&\uu^{k+1}=\vv^k+\hat n\theta_k(\z^{k+1}-\z^k),\notag\\
&&\theta_{k+1}=\frac{\sqrt{\theta_k^4+4\theta_k^2}-\theta_k^2}{2}.\notag
\end{eqnarray}

At each iteration, accelerated randomized coordinate descent picks a random coordinate $i_k\in\{1,2,\cdots,\hat n\}$ and generates $\z^{k+1},\uu^{k+1}$ and $\vv^{k+1}$. Only the $i_k$-th coordinate of $\z^{k+1}$ is updated and the other coordinates remain unchanged. However, the above algorithm performs full-dimensional vector operations of $\vv^k$ and $\uu^k$, which makes the per-iteration cost higher than the simple non-accelerated coordinate descent. To avoid such operations, we can use a change of variables scheme proposed in \cite{lee2013,fercoq-2015-siam}. Specifically, introduce $\hat\uu^k$ initialized at $\hat\uu^0=0$ and the new algorithms consists of the following steps at each iteration:
\begin{eqnarray}
&&\mbox{select }i_k\mbox{ randomly with probability of }1/\hat n,\notag\\
&&\hat\z_{i_k}^{k+1}=\argmin_{u} \frac{\hat n\theta_kL_{i_k}}{2}\|u-\hat\z_{i_k}^k\|^2+\<\nabla_{i_k}d\left(\theta_k^2\hat\uu^k+\hat\z^k\right),u-\hat\z_{i_k}^k\>+h_{i_k}(u),\notag\\
&&\hat\z_j^{k+1}=\hat\z_j^k,\forall j\neq i_k,\notag\\
&&\hat\uu^{k+1}=\hat\uu^k-\frac{1-\hat n\theta_k}{\theta_k^2}(\hat\z^{k+1}-\hat\z^k),\notag\\
&&\theta_{k+1}=\frac{\sqrt{\theta_k^4+4\theta_k^2}-\theta_k^2}{2}.\notag
\end{eqnarray}
From Proposition 1 in \cite{fercoq-2015-siam}, we know that the above two algorithms are equivalent in the sense of $\z^k=\hat\z^k$, $\uu^{k+1}=\theta_k^2\hat\uu^{k+1}+\z^{k+1}$ and $\vv^k=\theta_k^2\hat\uu^k+\z^k$ given $\uu^0=\z^0=\hat\z^0$.

Next, we discuss the computation of $\nabla_{i_k} d(\vv^k)$, and we expect that it is best $\hat n$ times faster than the computation of $\nabla d(\vv^k)$. Consider the simple case of (\ref{function_g}). Define
\begin{eqnarray}
\S=\left[\A/n,\B^T,\J^T\right]\in\R^{ t\times\hat n}\qquad\mbox{and}\qquad\p=\left[\0^T,\b^T,\q^T\right]^T\in\R^{\hat n}.\label{def_Sq}
\end{eqnarray}
Then from the definitions in (\ref{cont28}) and (\ref{xu_solution}), we have $d(\uu)=-\min_{\x}(f(\x)+\<\S\uu,\x\>+\<\uu,\p\>)$. So we can prove $\x^*(\vv^k)=\nabla f^*(-\S\vv^k)$, $\nabla d(\vv^k)=-\S^T\x^*(\vv^k)-\p$ and $\nabla_i d(\vv^k)=-\S_i^T\x^*(\vv^k)-\p_i$. Thus if we keep a variable $\s_{\vv}^k\equiv\S\vv^k\in\R^t$ and update it without the full matrix-vector multiplication, $\x^*(\vv^k)$ and $\nabla_i d(\vv^k)$ can be efficiently computed. We describe the explicit update of $\s_{\vv}^k$ in Algorithm \ref{alg_arpcg}. At each iteration, only the $i_k$-th column of $\S$ is used, rather than the full matrix $\S$. So Algorithm \ref{alg_arpcg} only needs to deal with the $i_k$-th constraint, rather than all the constraints at each iteration.

The only thing left to do is to compute the gradient of $f^*$ and the proximal mapping of $\phi_i^*$. In machine learning, $f$ is often the regularizer and $\phi_i$ is the loss function. The most commonly used strongly convex regularizer is the $L_2$ regularization $f(\x)=\frac{\mu}{2}\|\x\|^2$. In this case, $f^*(\x)=\frac{1}{2\mu}\|\x\|^2$ and thus the computation time of $\nabla f^*(-\s_{\vv}^k)$ is $O(t)$. We can also use some non-smooth strongly convex regularizers in the form of $f(\x)=\frac{\mu}{2}\|\x\|^2+\sigma(\x)$. In this case, $\nabla f^*(\x)=\mbox{Prox}_{\sigma/\mu}(\x/\mu)$, which can be efficiently computed when the proximal mapping of $\sigma(\x)$ has closed form solution. On the other hand, when computing the proximal mapping of $\phi_i^*$ (or the proximal mapping of $\phi_i$ since $u=\mbox{Prox}_{\phi_i}(u)+\mbox{Prox}_{\phi_i^*}(u)$), we only need to solve an optimization problem with only one dimension, which can be efficiently done, e.g., by the nonsmooth Newton or quasi Newton method \cite{Lewis-2013-newton}. For many machine learning problems, the proximal mapping of $\phi_i^*$ can be computed efficiently (see, e.g., \cite{zhang-2013-jmlr,zhang-2015-MP}). Thus, Algorithm \ref{alg_arpcg} needs about $O(t)$ time at each iteration while DFGA and ADFGA need $O(t\hat n)$ time.

\begin{algorithm}
   \caption{ARDCA}
   \label{alg_arpcg}
\begin{algorithmic}
   \STATE Input $\uu^0\in \DS$, $K_0$, $K$
   \STATE Initialize $\z^0=\uu^0$, $\hat\uu^0=\0$, $\s_{\z}^0=\S\z^0$, $\s_{\hat\uu}^0=0$, $\theta_0=\frac{1}{\hat n}$.
   \FOR{$k=0, 1, 2, 3, \cdots, K$}
   \STATE $\s_{\vv}^k=\theta_k^2\s_{\hat\uu}^k+\s_{\z}^k$,
   \STATE $\x^*(\vv^k)=\nabla f^*(-\s_{\vv}^k)$,
   \STATE select $i_k$ randomly with probability of $1/\hat n$,
   \STATE $\nabla_{i_k} d(\vv^k)=-\S_{i_k}^T\x^*(\vv^k)-\p_{i_k}$,
   \STATE $\z_{i_k}^{k+1}=\argmin_{u} \hat n\theta_kL_{i_k}\|u-\z_{i_k}^k\|^2+\<\nabla_{i_k}d(\vv^k),u-\z_{i_k}^k\>+h_{i_k}(u)$,
   \STATE $\hat\uu_{i_k}^{k+1}=\hat\uu_{i_k}^k-\frac{1-\hat n\theta_k}{\theta_k^2}(\z_{i_k}^{k+1}-\z_{i_k}^k)$,
   \STATE $\z_j^{k+1}=\z_j^k$ and $\hat\uu_j^{k+1}=\hat\uu_j^k$ for $j\neq i_k$,
   \STATE $\s_{\z}^{k+1}=\s_{\z}^k+\S_{i_k}(\z_{i_k}^{k+1}-\z_{i_k}^{k})$,
   \STATE $\s_{\hat\uu}^{k+1}=\s_{\hat\uu}^k+\S_{i_k}(\hat\uu_{i_k}^{k+1}-\hat\uu_{i_k}^{k})$,
   \STATE $\theta_{k+1}=\frac{\sqrt{\theta_k^4+4\theta_k^2}-\theta_k^2}{2}$,
   \ENDFOR
   \STATE Output $\hat\x^K=\frac{\sum_{k=K_0}^K\frac{\x^*(\vv^k)}{\theta_k}}{\sum_{k=K_0}^K\frac{1}{\theta_k}}$ and $\uu^{K+1}=\theta_K^2\hat\uu^{K+1}+\z^{K+1}$.
\end{algorithmic}
\end{algorithm}
\begin{remark}
A main difference between Algorithm \ref{alg_arpcg} and the original accelerated randomized coordinate descent is that Algorithm \ref{alg_arpcg} uses the step-size of $\frac{1}{\hat n\theta_kL_{i_k}}$ when computing $\z_{i_k}^{k+1}$ while the original algorithm uses a larger one of $\frac{2}{\hat n\theta_kL_{i_k}}$, which makes the original algorithm faster than Algorithm \ref{alg_arpcg} in practice. The reason of the smaller step-size is to fit the proof. Specifically, it allows us to keep an additional term $\frac{\hat n^2\theta_k^2}{2}\|\z^{k+1}-\z^k\|_L^2$ in Lemma \ref{lemma3}, which is crucial in the proof of Lemma \ref{lemma2}. Otherwise, we may only bound $\|\E_{\xi_K}[\cdot]\|_L^*$ for the constraint functions, rather than $\E_{\xi_K}[\|\cdot\|_L^*]$ in Lemma \ref{lemma2}. The former is less interesting since the expectation is inside the norm.
\end{remark}

In Algorithm \ref{alg_arpcg}, we output the average of $\x^*(\vv^k)$ from some $K_0$ to $K$\footnote{We leave the efficient computation of the average in Appendix A.}, rather than $\x^*(\uu^{K+1})$. This little change allows us to give a faster convergence rate in the primal space. We average from $K_0$, rather than from the first iteration, since the first few iterations often produce poor solutions. We now state our main result on the convergence rate of the primal solutions for ARDCA. Let
\begin{eqnarray}
\xi_k=\{i_0,i_1,\cdots,i_k\}\notag
\end{eqnarray}
denote the random sequence, $\E_{\xi_k}$ be the expectation with respect to $\xi_k$ and $\E_{i_k|\xi_{k-1}}$ be the conditional expectation with respect to $i_k$ conditioned on $\xi_{k-1}$, then we have the following theorem.
\begin{theorem}\label{main_theorem}
Suppose Assumption \ref{assumption} holds. Let $K_0\leq\left\lfloor\frac{K}{\upsilon(1+1/\hat n)}+1\right\rfloor$ with any $\upsilon>1$. Then for Algorithm \ref{alg_arpcg}, we have
\begin{eqnarray}
\begin{aligned}
&\left|\E_{\xi_K}\hspace*{-0.05cm}[F(\hat\x^K)]\hspace*{-0.09cm}-\hspace*{-0.09cm}F(\x^*)\right|\hspace*{-0.09cm}\leq\hspace*{-0.09cm}\frac{9\hat n^2\hspace*{-0.09cm}\left( (1\hspace*{-0.09cm}-\hspace*{-0.09cm}\theta_0)\hspace*{-0.09cm}\left(D(\uu^0)\hspace*{-0.09cm}-\hspace*{-0.09cm} D(\uu^*)\right)\hspace*{-0.09cm}+\hspace*{-0.09cm}\|\uu^0\hspace*{-0.09cm}-\hspace*{-0.09cm}\uu^*\|_L^2\hspace*{-0.09cm}+\hspace*{-0.09cm}\|\uu^*\|_L^2\hspace*{-0.09cm}+\hspace*{-0.09cm}M^2\sum_{i=1}^nL_i\right)}{(K^2/4+\hat nK)\left(1-1/\upsilon\right)},\notag\\
&\E_{\xi_K}\hspace*{-0.09cm}\left[\left\| \left[\hspace*{-0.12cm}
  \begin{array}{c}
    \B\hat\x^K+\b\\
    \max\left\{0,g(\hat\x^K)\right\}
  \end{array}
\right]\right\|_L^*\right] \hspace*{-0.09cm}\leq\hspace*{-0.09cm} \frac{7\hat n^2}{(K^2\hspace*{-0.05cm}/\hspace*{-0.05cm}4\hspace*{-0.09cm}+\hspace*{-0.09cm}\hat nK)\hspace*{-0.09cm}\left(1\hspace*{-0.09cm}-\hspace*{-0.09cm}1/\upsilon\right)}\sqrt{ \hspace*{-0.05cm}(1\hspace*{-0.09cm}-\hspace*{-0.09cm}\theta_0)\hspace*{-0.09cm}\left(D(\uu^0)\hspace*{-0.09cm}-\hspace*{-0.09cm} D(\uu^*)\right)\hspace*{-0.09cm}+\hspace*{-0.09cm}\|\uu^0\hspace*{-0.09cm}-\hspace*{-0.09cm}\uu^*\|_L^2}.\notag
\end{aligned}
\end{eqnarray}
\end{theorem}

Now we compare the convergence rate of the primal solutions with that of the dual solutions. For the dual problem (\ref{dual_problem}), \cite{xiao-2015-siam} proved the $O\left(\frac{\hat n^2}{K^2}\right)$ convergence rate, which is described in the following proposition.
\begin{proposition}\cite{xiao-2015-siam}
Suppose Assumption \ref{assumption} holds. Then for Algorithm \ref{alg_arpcg}, we have
\begin{eqnarray}
\E_{\xi_K}\hspace*{-0.07cm}[ D(\uu^{K+1})]\hspace*{-0.07cm}-\hspace*{-0.07cm} D(\uu^*)\hspace*{-0.07cm}\leq\hspace*{-0.07cm} \left(\hspace*{-0.07cm}\frac{2\hat n}{2\hat n\hspace*{-0.07cm}+\hspace*{-0.07cm}K\hat n/\hspace*{-0.07cm}\sqrt{\hat n^2\hspace*{-0.07cm}-\hspace*{-0.07cm}1}}\hspace*{-0.07cm}\right)^2\hspace*{-0.07cm}\left(\hspace*{-0.07cm} D(\uu^0)\hspace*{-0.07cm}-\hspace*{-0.07cm} D(\uu^*)\hspace*{-0.07cm}+\hspace*{-0.07cm}\frac{\hat n^2}{2(\hat n^2\hspace*{-0.07cm}-\hspace*{-0.07cm}1)}\|\uu^0\hspace*{-0.07cm}-\hspace*{-0.07cm}\uu^*\|_L^2\hspace*{-0.07cm}\right).\label{dual_rate}
\end{eqnarray}
\end{proposition}
Thus we can see that Algorithm \ref{alg_arpcg} needs $O\left(\frac{\hat n}{\sqrt{\epsilon}}\right)$ iterations to achieve an $\epsilon$-optimal primal solution and dual solution, i.e., the iteration complexity of the primal solutions has the same order of magnitude as that of the dual solutions for ARDCA.

To make a better comparison with the existing result, we describe the relation of the primal objective and constraint functions with the dual objective without averaging the primal solutions in the following proposition, which only applies to problem (\ref{problem1}). Combing with (\ref{dual_rate}), we can immediately get the $O\left( \frac{\hat n}{K+\hat n} \right)$ convergence rate in the primal space, which verifies that averaging the primal solutions helps to improve the convergence rate in the primal space.
\begin{proposition}\cite{Lu-2016-siam}
Suppose Assumption \ref{assumption} holds for problem (\ref{problem1}). Then for Algorithm \ref{alg_arpcg}, we have
\begin{eqnarray}
\begin{aligned}
&\E_{\xi_K}\hspace*{-0.04cm}[f\hspace*{-0.04cm}(\hspace*{-0.04cm}\x^*(\hspace*{-0.04cm}\uu^{K\hspace*{-0.04cm}+\hspace*{-0.04cm}1}\hspace*{-0.04cm})\hspace*{-0.04cm})]\hspace*{-0.11cm}-\hspace*{-0.11cm}f(\hspace*{-0.04cm}\x^*\hspace*{-0.04cm})\hspace*{-0.11cm}\leq\hspace*{-0.11cm} \left(\hspace*{-0.13cm} \|\hspace*{-0.04cm}\uu^{K\hspace*{-0.04cm}+\hspace*{-0.04cm}1}\hspace*{-0.04cm}\|_{\infty}\hspace*{-0.07cm}\sqrt{\hspace*{-0.04cm}2L\hat n}\hspace*{-0.11cm}+\hspace*{-0.11cm} \sqrt{\E_{\xi_K}\hspace*{-0.04cm}[D(\hspace*{-0.04cm}\uu^{K\hspace*{-0.04cm}+\hspace*{-0.04cm}1}\hspace*{-0.04cm})]\hspace*{-0.11cm}-\hspace*{-0.11cm} D(\hspace*{-0.04cm}\uu^*\hspace*{-0.04cm})}\hspace*{-0.04cm} \right)\hspace*{-0.13cm}\sqrt{\E_{\xi_K}\hspace*{-0.04cm}[D(\hspace*{-0.04cm}\uu^{K\hspace*{-0.04cm}+\hspace*{-0.04cm}1}\hspace*{-0.04cm})]\hspace*{-0.11cm}-\hspace*{-0.11cm} D(\hspace*{-0.04cm}\uu^*\hspace*{-0.04cm})},\notag\\
&\E_{\xi_K}[f(\x^*(\uu^{K+1}))]-f(\x^*)\geq -\|\uu^*\|\sqrt{2L\left(\E_{\xi_K}[D(\uu^{K+1})]- D(\uu^*)\right)},\\
&\E_{\xi_K}\left[\left\| \left[
  \begin{array}{c}
    \B\x^*(\uu^{K+1})+\b\\
    \max\left\{0,g(\x^*(\uu^{K+1}))\right\}
  \end{array}
\right]\right\|_L^*\right] \leq \sqrt{2L\left(\E_{\xi_K}[D(\uu^{K+1})]- D(\uu^*)\right)}.\notag
\end{aligned}
\end{eqnarray}
\end{proposition}

\section{Convergence Rate Analysis of the Primal Solutions}\label{sec_proof}
In this section, we prove Theorem \ref{main_theorem}. First we study a simple case in Section \ref{sec:simple_case} to intuitively show how to relate the primal objective with the dual objective and why average helps to improve the convergence rate. Then we give the detailed analysis for the general case in Section \ref{sec:general_case}.
\subsection{Intuition: A Case Study}\label{sec:simple_case}
In this section, we study a simple case of problem (\ref{problem}):
\begin{eqnarray}
\begin{aligned}\notag
&\min_{\x\in\R^t} f(\x),\quad s.t.\quad \B\x=\b.
\end{aligned}
\end{eqnarray}
We only consider the gradient descent to solve the dual problem for simplicity, which has the recursion of
\begin{eqnarray}
\begin{aligned}\label{GD}
&\uu^{k+1}=\uu^k-\frac{1}{L}\nabla d(\uu^k).
\end{aligned}
\end{eqnarray}
Then $d(\uu)$ and $\nabla d(\uu)$ reduce to
\begin{eqnarray}
&&d(\uu)=-f(\x^*(\uu))-\<\uu,\B\x^*(\uu)-\b\>,\notag\\
&&\nabla d(\uu)=-(\B\x^*(\uu)-\b),\label{cont51}
\end{eqnarray}
which further leads to
\begin{eqnarray}
\begin{aligned}
&d(\uu)-\<\uu,\nabla d(\uu)\>=-f(\x^*(\uu)).\label{cont52}
\end{aligned}
\end{eqnarray}
(\ref{cont52}) is a crucial property to relate the primal objective and dual objective. From the $L$-smoothness of $d(\uu)$, we have
\begin{eqnarray}
\begin{aligned}\notag
d(\uu^{k+1})\leq d(\uu^k)+\<\nabla d(\uu^k),\uu^{k+1}-\uu^k\>+\frac{L}{2}\|\uu^{k+1}-\uu^k\|^2.
\end{aligned}
\end{eqnarray}
From $-f(\x^*)=d(\uu^*)\leq d(\uu^{k+1})$, (\ref{cont52}), (\ref{cont51}) and (\ref{GD}), we have
\begin{eqnarray}
\begin{aligned}\notag
-f(\x^*)\leq&-f(\x^*(\uu^k))+\<\nabla d(\uu^k),\uu^{k+1}\>+\frac{L}{2}\|\uu^{k+1}-\uu^k\|^2\\
=&-f(\x^*(\uu^k))-\<\B\x^*(\uu^k)-\b,\uu\>+\<\nabla d(\uu^k),\uu^{k+1}-\uu\>+\frac{L}{2}\|\uu^{k+1}-\uu^k\|^2\\
=&-f(\x^*(\uu^k))-\<\B\x^*(\uu^k)-\b,\uu\>-L\<\uu^{k+1}-\uu^k,\uu^{k+1}-\uu\>+\frac{L}{2}\|\uu^{k+1}-\uu^k\|^2\\
=&-f(\x^*(\uu^k))-\<\B\x^*(\uu^k)-\b,\uu\>+\frac{L}{2}\|\uu^k-\uu\|^2-\frac{L}{2}\|\uu^{k+1}-\uu\|^2.
\end{aligned}
\end{eqnarray}
Define $\hat\x^K=\frac{\sum_{k=0}^K\x^*(\uu^k)}{K+1}$, letting $\uu=\uu^*$ and summing over $k=0,1,\cdots,K$, we have
\begin{eqnarray}
\begin{aligned}\notag
f(\hat\x^K)+\<\B\hat\x^K-\b,\uu^*\>-f(\x^*)\overset{a}\leq& \frac{\sum_{k=0}^K\left(f(\x^*(\uu^k))+\<\B\x^*(\uu^k)-\b,\uu^*\>\right)}{K+1}-f(\x^*)\\
\leq& \frac{L}{2(K+1)}\|\uu^0-\uu^*\|^2,
\end{aligned}
\end{eqnarray}
where we use Jensen's inequality for $f(\x)$ in $\overset{a}\leq$. On the other hand, from (\ref{cont51}) and (\ref{GD}), we have
\begin{eqnarray}
\begin{aligned}\label{cont53}
&\|\B\hat\x^K\hspace*{-0.09cm}-\hspace*{-0.09cm}\b\|\hspace*{-0.09cm}=\hspace*{-0.09cm}\frac{1}{K\hspace*{-0.09cm}+\hspace*{-0.09cm}1}\hspace*{-0.09cm}\left\|\sum_{k=0}^K(\B\x^*(\uu^k)\hspace*{-0.09cm}-\hspace*{-0.09cm}\b)\right\|\hspace*{-0.09cm}=\hspace*{-0.09cm}\frac{1}{K+1}\hspace*{-0.09cm}\left\|\sum_{k=0}^K\nabla d(\uu^k)\right\|\hspace*{-0.09cm}=\hspace*{-0.09cm}\frac{L}{K+1}\hspace*{-0.09cm}\left\|\sum_{k=0}^K(\uu^{k+1}\hspace*{-0.09cm}-\hspace*{-0.09cm}\uu^k)\right\|\\
&=\frac{L}{K+1}\|\uu^{K+1}-\uu^0\|\leq \frac{L}{K+1}\left(\|\uu^0-\uu^*\|+\|\uu^{K+1}-\uu^*\|\right)\overset{b}\leq \frac{2L}{K+1}\|\uu^0-\uu^*\|,
\end{aligned}
\end{eqnarray}
where we use the non-increasing of $\|\uu^k-\uu^*\|$ for gradient descent in $\overset{b}\leq$. Thus, from the above two inequalities, we have
\begin{eqnarray}
\begin{aligned}\notag
f(\hat\x^K)-f(\x^*)\leq O\left(\frac{1}{K}\right)\quad\mbox{and}\quad\|\B\hat\x^K-\b\|\leq O\left(\frac{1}{K}\right),
\end{aligned}
\end{eqnarray}
which gives the $O\left(\frac{1}{K}\right)$ convergence rate of the primal solutions. When replacing gradient descent with accelerated gradient descent, the convergence rate can be improved to $O\left(\frac{1}{K^2}\right)$. However, more efforts are required for the analysis in the dual space and the averaging weights should be designed carefully. When solving the general problem (\ref{problem}), we should consider the separable part $\phi_i(\A_i^T\x)$ and the inequality constraints more carefully. When using the accelerated randomized dual coordinate ascent, we should pay more efforts to deal with the expectation, especially that the deduction in (\ref{cont53}) is not enough to deal with the constraint functions and it requires more skillful analysis. We give the detailed analysis in the following section.

From the above analysis, we can see that (\ref{cont52}) is a critical trick in our analysis. To show the importance of averaging the primal solutions, we give a simple convergence rate analysis measured at the non-averaged primal solution, which is adapted from \cite{Lu-2016-siam}. From (\ref{cont52}) and $-f(\x^*)=d(\uu^*)$, we have
\begin{eqnarray}
\begin{aligned}\notag
f(\x^*(\uu))-f(\x^*)=d(\uu^*)-d(\uu)+\<\uu,\nabla d(\uu)\>\leq \<\uu,\nabla d(\uu)\>\leq \sqrt{\hat n}\|\uu\|_{\infty}\|\nabla d(\uu)\|.
\end{aligned}
\end{eqnarray}
So we only need to bound $\|\nabla d(\uu)\|$. From the smoothness of $d(\uu)$, we have
\begin{eqnarray}
\begin{aligned}\notag
&d(\uu)-d(\uu^*)\geq d(\uu)-d(\uu+\nabla d(\uu)/L)\\
&\geq \<\nabla d(\uu),\uu+\nabla d(\uu)/L - \uu\>-\frac{L}{2}\|\uu+\nabla d(\uu)/L - \uu\|^2=\frac{1}{2L}\|\nabla d(\uu)\|^2.
\end{aligned}
\end{eqnarray}
Thus, we have $\|\nabla d(\uu)\|\leq\sqrt{2L(d(\uu)-d(\uu^*))}$. Since $d(\uu^K)-d(\uu^*)\leq O\left(\frac{1}{K}\right)$ for gradient descent and (\ref{cont51}), we have
\begin{eqnarray}
\begin{aligned}\notag
f(\x^*(\uu^K))-f(\x^*)\leq O\left(\frac{1}{\sqrt{K}}\right)\quad\mbox{and}\quad \|\B\x^*(\uu^K)-\b\|\leq O\left(\frac{1}{\sqrt{K}}\right).
\end{aligned}
\end{eqnarray}

Generally speaking, average helps to improve the convergence rate for many algorithms. Typical examples include the Douglas-Rachford splitting and ADMM \cite{yin-DR2017}. \cite{yin-DR2017} proved that for the two algorithms, the $O\left(\frac{1}{\sqrt{K}}\right)$ and $O\left(\frac{1}{K}\right)$ rates are tight for the non-averaged and averaged solutions, respectively.

From the above analysis for the non-averaged solutions, we can see that $\|\nabla d(\uu)\|$ serves as a measure of the optimality and feasibility of the primal solutions. \cite{nesterov-2012-gradient} studied how to make gradient small. Specifically, to find a point $\uu$ with $\|\nabla d(\uu)\|\leq\epsilon$, accelerated dual ascent needs $O\left( \frac{1}{\epsilon^{2/3}} \right)$ iterations. When using some  regularization technique, the $O\left( \frac{1}{\sqrt{\epsilon}}\log\frac{1}{\epsilon} \right)$ complexity can be attained. We can see that none of them reaches the optimal $O\left( \frac{1}{\sqrt{\epsilon}} \right)$ complexity. On the other hand, when averaging the primal solutions, we do not need to make the gradient small. Instead, we only need to make the average of gradients small, i.e., $\frac{1}{K+1}\hspace*{-0.09cm}\left\|\sum_{k=0}^K\nabla d(\uu^k)\right\|$ in (\ref{cont53}). This is the reason why average helps to improve the convergence rate.
\subsection{ARPCA for the General Problem}\label{sec:general_case}

To better analyze the method, we give an equivalent algorithm of ARDCA and describe it in Algorithm \ref{alg_arpcg_equ}. In Algorithm \ref{alg_arpcg_equ}, variables $\widetilde\z_{j}^k,\forall j\neq i_k,$ are only used for analysis. In practice, we do not need to compute $\widetilde\z_{j}^k,\forall j\neq i_k$.
\begin{algorithm}
   \caption{Equivalent ARDCA only for analysis}
   \label{alg_arpcg_equ}
\begin{algorithmic}
   \STATE Initialize $\z^0=\uu^0\in \DS$, $\theta_0=\frac{1}{\hat n}$.
   \FOR{$k=0, 1, 2, 3, \cdots$}
   \STATE $\vv^k=\theta_k\z^k+(1-\theta_k)\uu^k$,
   \FOR{$i=1, 2, \cdots, \hat n$}
   \STATE $\widetilde\z_i^k=\argmin_{u} \hat n\theta_kL_i\|u-\z_{i}^k\|^2+\<\nabla_{i}d(\vv^k),u-\z_{i}^k\>+h_{i}(u)$,
   \ENDFOR
   \STATE select $i_k$ randomly with probability $\frac{1}{\hat n}$,
   \STATE $\z_{i_k}^{k+1}=\widetilde\z_{i_k}^k$,
   \STATE $\z_j^{k+1}=\z_j^k$ for $j\neq i_k$,
   \STATE $\uu^{k+1}=\vv^k+\hat n\theta_k(\z^{k+1}-\z^k)$,
   \STATE $\theta_{k+1}=\frac{\sqrt{\theta_k^4+4\theta_k^2}-\theta_k^2}{2}$.
   \ENDFOR
   \STATE Output $\hat\x^K=\frac{\sum_{k=K_0}^K\frac{\x^*(\vv^k)}{\theta_k}}{\sum_{k=K_0}^K\frac{1}{\theta_k}}$ and $\uu^{K+1}$.
\end{algorithmic}
\end{algorithm}

The definition of $\{\theta_0,\theta_1,\cdots,\theta_K\}$ satisfies $\frac{1-\theta_{k}}{\theta_{k}^2}=\frac{1}{\theta_{k-1}^2}$. Define $\theta_{-1}=1/\sqrt{\hat n^2-\hat n}$, which also satisfies $\frac{1-\theta_{k}}{\theta_{k}^2}=\frac{1}{\theta_{k-1}^2}$ for $k=0$. For the sequence $\{\theta_0,\theta_1,\cdots,\theta_K\}$, we can simply prove the following properties.
\begin{lemma}\label{theta_lemma} For the sequence $\{\theta_0,\theta_1,\cdots,\theta_K\}$ satisfying $\theta_0=\frac{1}{\hat n}$ and $\frac{1-\theta_{k}}{\theta_{k}^2}=\frac{1}{\theta_{k-1}^2},\forall k\geq 0$, we have
\begin{enumerate}
\item $0\leq\theta_k\leq\theta_{k-1}\leq\cdots\leq\theta_1\leq\theta_0=\frac{1}{\hat n}$.
\item $\sum_{k=K_0}^K\frac{1}{\theta_k}=\frac{1}{\theta_K^2}-\frac{1}{\theta_{K_0-1}^2}$.
\item $\frac{k}{2}+\frac{k}{2\hat n}+\hat n\geq\frac{1}{\theta_k}\geq\frac{k}{2}+\hat n$.
\item Letting $K_0\leq\left\lfloor\frac{K}{\upsilon(1+1/\hat n)}+1\right\rfloor$, we have $\frac{1}{\theta_K^2}-\frac{1}{\theta_{K_0-1}^2}\geq \left(\frac{K^2}{4}+\hat nK\right)\left(1-\frac{1}{\upsilon}\right)$ for any $\upsilon>1$.
\end{enumerate}
\end{lemma}

We follow \cite{fercoq-2015-siam} to define the sequence $\{\alpha_{k,t}:0\leq t\leq k,k=0,1,\cdots\}$ satisfying
\begin{eqnarray}
\alpha_{0,0}=1,
\alpha_{1,t}=
\left \{
  \begin{array}{ll}
     1-\hat n\theta_0, & t=0,\\
     \hat n\theta_0, & t=1,\\
  \end{array}
\right.
\alpha_{k+1,t}=
\left \{
  \begin{array}{ll}
     (1-\theta_k)\alpha_{k,t}, & t\leq k-1,\\
     (1-\theta_k)\alpha_{k,k}-(\hat n-1)\theta_k, & t=k,\\
     \hat n\theta_k, & t=k+1.\\
  \end{array}
\right.\notag
\end{eqnarray}
From Lemma 2 in \cite{fercoq-2015-siam}, we have $0\leq \alpha_{k,t}\leq 1,\forall t=0,\cdots,k$, $\sum_{t=0}^k\alpha_{k,t}=1$ and $\uu^{k+1}=\sum_{t=0}^{k+1} \alpha_{k+1,t}\z^t$. Define $H^{k+1}=\sum_{t=0}^{k+1} \alpha_{k+1,t} h(\z^t)$, then $h(\uu^{k+1})\leq H^{k+1}$ due to Jensen's inequality for $h(\x)$. We can easily verify that variables $\z^k$, $\uu^{k}$ and $\vv^{k}$ remain in $\DS$ at all times by induction, described in the following lemma. Then we can use the Lipschitz smooth condition described in Lemma \ref{coordinate_LS}.
\begin{lemma}
For Algorithm \ref{alg_arpcg_equ}, we have $\z^k\in\DS$, $\uu^k\in\DS$ and $\vv^k\in\DS,\forall k\geq 0$.
\end{lemma}
For Algorithm \ref{alg_arpcg_equ}, let
\begin{eqnarray}
\y_i^k=-2\hat n\theta_k L_i(\widetilde\z_i^k-\z_{i}^k)-\nabla_{i} d(\vv^k),\quad 1\leq i\leq n. \label{def_y}
\end{eqnarray}
From the optimality condition of $\widetilde\z_i^k$ in Algorithm \ref{alg_arpcg_equ}, we have $\y_i^k\in\partial h_i(\widetilde\z_i^k), 1\leq i\leq n$. Define
\begin{eqnarray}
\sigma_1(\uu_i,\widetilde\z_i^k)=\left\{
  \begin{array}{lcl}
    h_{i}(\widetilde\z_i^k)+\<\y_i^k,\uu_{i}-\widetilde\z_i^k\>-h_{i}(\uu_{i}), && \mbox{ if }i\leq n,\\
    0, && \mbox{ if }n< i\leq n+p+m,\\
  \end{array}
\right.\label{cont49}
\end{eqnarray}
and
\begin{eqnarray}
\sigma_2(\uu,\vv^k)=d(\vv^k)+\<\nabla d(\vv^k),\uu-\vv^k\>-d(\uu).\label{cont50}
\end{eqnarray}
From the convexity of $h_i$ and $d$, we have $\sigma_1(\uu_i,\widetilde\z_i^k)\leq 0$ and $\sigma_2(\uu,\vv^k)\leq 0$. We use $\sigma_1(\uu_i,\widetilde\z_i^k)$ and $\sigma_2(\uu,\vv^k)$ to relate the primal objective function, primal constraint functions and dual objective function in the following lemma.
\begin{lemma}\label{pd_relation}
Suppose Assumption \ref{assumption} holds. For any $\uu\in\DS$, we have
\begin{eqnarray}
\begin{aligned}\label{pd_relation1}
-\sum_{i=1}^n\sigma_1(\uu_i,\widetilde\z_i^k)-\sigma_2(\uu,\vv^k)=\<\triangle(\x^*(\vv^k),n\y^k),\uu\>+ D(\uu)+f(\x^*(\vv^k))+\frac{1}{n}\phi(n\y^k),
\end{aligned}
\end{eqnarray}
where
\begin{eqnarray}
\triangle(\x,\y)=\left[(\A^T\x-\y)^T/n,(\B\x+\b)^T,g_1(\x),\cdots,g_m(\x)\right]^T.\notag
\end{eqnarray}
\end{lemma}
\begin{proof}
From (\ref{gradient}), (\ref{xu_solution}) and the definition of $d(\uu)$ in (\ref{cont28}), we have
\begin{eqnarray}
f(\x^*(\uu))=-d(\uu)+\nabla d(\uu)^T\uu.\notag
\end{eqnarray}
Thus, by the definition of $\sigma_2(\uu,\vv^k)$ in (\ref{cont50}), we have
\begin{eqnarray}
\begin{aligned}
\sigma_2(\uu,\vv^k)=&\<\nabla d(\vv^k),\uu\>-d(\uu)-\left[\<\nabla d(\vv^k),\vv^k\>-d(\vv^k)\right]\notag\\
=&\<\nabla d(\vv^k),\uu\>-d(\uu)-f(\x^*(\vv^k)).\notag
\end{aligned}
\end{eqnarray}
From the definition of $\sigma_1(\uu_i,\widetilde\z_i^k)$ in (\ref{cont49}) and the fact that $\y_i^k\in\partial h_i(\widetilde\z_i^k),1\leq i\leq n$, we can also have
\begin{eqnarray}
\begin{aligned}
\sum_{i=1}^n\sigma_1(\uu_i,\widetilde\z_i^k)=&\sum_{i=1}^{n}\left(\<\y_i^k,\uu_i\>-h_i(\uu_i)-\left[\<\y_i^k,\widetilde\z_i^k\>-h_i(\widetilde\z_i^k)\right]\right)\notag\\
\overset{a}=&\sum_{i=1}^{n}\left(\<\y_i^k,\uu_i\>-h_i(\uu_i)-h_i^*(\y_i^k)\right)\notag\\
\overset{b}=&\sum_{i=1}^{n}\left(\<\y_i^k,\uu_i\>-h_i(\uu_i)-\frac{1}{n}\phi_{i}(n\y_i^k)\right),\notag
\end{aligned}
\end{eqnarray}
where we use the definition of conjugate in $\overset{a}=$ and the fact that $\varphi(x)=\alpha\psi(x)\Rightarrow \varphi^*(y)=\alpha \psi^*(y/\alpha)$ and $h_i(\uu_i)=\frac{1}{n}\phi_i^*(\uu_i),\forall i\leq n$ in $\overset{b}=$. So the result of (\ref{pd_relation1}) immediately follows by adding the above two equations and using (\ref{gradient}).
\end{proof}

\subsubsection{Primal Objective}

In the following lemma, we use the relation (\ref{pd_relation1}) to bound the Lagrangian function $L_F(\hat\x^K,\hat\y^K,\uu)$. We also bound $\sum_{k=K_0}^K\E_{\xi_K}[\|\z^{k+1}-\z^k\|_L^2]$ and $\E_{\xi_K}[\|\z^{K+1}-\uu^*\|_L^2]$, which will be used to bound the constraint functions later. As a by-product, we also give the convergence rate of the dual solutions.
\begin{lemma}\label{lemma0}
Suppose Assumption \ref{assumption} holds. Define $\hat\x^K=\frac{\sum_{k=K_0}^K\frac{\x^*(\vv^k)}{\theta_k}}{\sum_{k=K_0}^K\frac{1}{\theta_k}}$ and $\hat\y^K=\frac{\sum_{k=K_0}^K\frac{n\y^k}{\theta_k}}{\sum_{k=K_0}^K\frac{1}{\theta_k}}$. Then we have
\begin{eqnarray}
\begin{aligned}\label{cont14}
&\left(\frac{1}{\theta_K^2}-\frac{1}{\theta_{K_0-1}^2}\right)\E_{\xi_K}\left[\<\triangle(\hat\x^K,\hat\y^K),\uu\>+ D(\uu^*)+f(\hat\x^K)+\frac{1}{n}\phi(\hat\y^K)\right]\\
\leq&2(\hat n^2-\hat n)\left( D(\uu^0)- D(\uu^*)\right)+2\hat n^2\|\uu^0-\uu^*\|_L^2+2\hat n^2\|\uu-\uu^*\|_L^2
\end{aligned}
\end{eqnarray}
for any $\uu\in\DS$ independent on $\xi_K$. We also have
\begin{eqnarray}
&&\hspace*{-1.2cm}\frac{1}{2}\sum_{k=K_0}^K\E_{\xi_K}[\|\z^{k+1}-\z^k\|_L^2]\leq  (1-\theta_0)\left(D(\uu^0)- D(\uu^*)\right)+\|\uu^0-\uu^*\|_L^2,\label{cont1}\\
&&\hspace*{-1.2cm}\E_{\xi_K}[\|\z^{K+1}-\uu^*\|_L^2 ]\leq  (1-\theta_0)\left(D(\uu^0)- D(\uu^*)\right)+\|\uu^0-\uu^*\|_L^2,\label{cont3}\\
&&\hspace*{-1.2cm}\frac{\E_{\xi_K}\hspace*{-0.07cm}[ D(\uu^{K+1})]\hspace*{-0.07cm}-\hspace*{-0.07cm} D(\uu^*)}{\theta_K^2}\hspace*{-0.07cm}+\hspace*{-0.07cm}\hat n^2\|\z^{K+1}\hspace*{-0.07cm}-\hspace*{-0.07cm}\uu^*\|_L^2\hspace*{-0.07cm}\leq\hspace*{-0.07cm}\hat n^2\hspace*{-0.07cm}\left( (1\hspace*{-0.07cm}-\hspace*{-0.07cm}\theta_0)\hspace*{-0.07cm}\left(D(\uu^0)\hspace*{-0.07cm}-\hspace*{-0.07cm} D(\uu^*)\right)\hspace*{-0.07cm}+\hspace*{-0.07cm}\|\uu^0\hspace*{-0.07cm}-\hspace*{-0.07cm}\uu^*\|_L^2\right).\label{cont34}
\end{eqnarray}
\end{lemma}
The proof of Lemma \ref{lemma0} is based on the following lemma, which gives the progress in one iteration in the dual space. Lemma \ref{lemma3} can be proved by the techniques in the proof of Theorem 3 in \cite{fercoq-2015-siam}, except that we keep the additional terms $\sum_{i=1}^n\sigma_1(\uu_i,\widetilde\z_i^k)+\sigma_2(\uu,\vv^k)$ and $\|\z^{k+1}-\z^k\|_L^2$. We leave the proof of Lemma \ref{lemma3} in Appendix D.
\begin{lemma}\label{lemma3}
Suppose Assumption \ref{assumption} holds. Then we have
\begin{eqnarray}
\begin{aligned}\label{cont4}
&\E_{\xi_K}\left[d(\uu^{k+1})+H^{k+1}- D(\uu)+\hat n^2\theta_k^2\|\z^{k+1}-\uu\|_L^2+\frac{\hat n^2\theta_k^2}{2}\|\z^{k+1}-\z^k\|_L^2\right]\\
\leq&(1\hspace*{-0.07cm}-\hspace*{-0.07cm}\theta_k)\E_{\xi_K}\hspace*{-0.07cm}[d(\uu^k)\hspace*{-0.07cm}+\hspace*{-0.07cm}H^k\hspace*{-0.07cm}-\hspace*{-0.07cm} D(\uu)]\hspace*{-0.07cm}+\hspace*{-0.07cm}\hat n^2\theta_k^2\E_{\xi_K}\hspace*{-0.07cm}[\|\z^k\hspace*{-0.07cm}-\hspace*{-0.07cm}\uu\|_L^2]\hspace*{-0.07cm}+\hspace*{-0.07cm}\theta_k\E_{\xi_K}\hspace*{-0.07cm}\left[\sum_{i=1}^n\hspace*{-0.07cm}\sigma_1(\uu_i,\widetilde\z_i^k)\hspace*{-0.07cm}+\hspace*{-0.07cm}\sigma_2(\uu,\vv^k)\hspace*{-0.07cm}\right]
\end{aligned}
\end{eqnarray}
for any $\uu\in\DS$ independent on $\xi_K$.
\end{lemma}
Based on Lemma \ref{lemma3}, we are ready to prove Lemma \ref{lemma0}.
\begin{proof}
Dividing both sides of (\ref{cont4}) by $\theta_k^2$ and using $\frac{1-\theta_k}{\theta_k^2}=\frac{1}{\theta_{k-1}^2}$, we have
\begin{eqnarray}
\begin{aligned}\label{cont29}
&\frac{\E_{\xi_K}[d(\uu^{k+1})+H^{k+1}- D(\uu)]}{\theta_k^2}+\hat n^2\E_{\xi_K}[\|\z^{k+1}-\uu\|_L^2]+\frac{\hat n^2}{2}\E_{\xi_K}[\|\z^{k+1}-\z^k\|_L^2]\\
\leq&\frac{\E_{\xi_K}[d(\uu^k)+H^k- D(\uu)]}{\theta_{k-1}^2}+\hat n^2\E_{\xi_K}[\|\z^k-\uu\|_L^2]+\frac{\E_{\xi_K}\left[\sum_{i=1}^n\sigma_1(\uu_i,\widetilde\z_i^k)+\sigma_2(\uu,\vv^k)\right]}{\theta_k}.
\end{aligned}
\end{eqnarray}
Letting $\uu=\uu^*$, we know $\frac{\E_{\xi_K}[d(\uu^k)+H^k]- D(\uu^*)}{\theta_{k-1}^2}+\hat n^2\E_{\xi_K}[\|\z^k-\uu^*\|_L^2]$ is decreasing since $\sigma_1(\uu_i,\z_{i}^{k+1})\leq 0$ and $\sigma_2(\uu,\vv_{i}^{k})\leq 0,\forall \uu$. Thus, we have
\begin{eqnarray}
\E_{\xi_K}\hspace*{-0.07cm}\left[\frac{d(\uu^{K_0})\hspace*{-0.07cm}+\hspace*{-0.07cm}H^{K_0}\hspace*{-0.07cm}-\hspace*{-0.07cm} D(\uu^*)}{\theta_{K_0-1}^2}\hspace*{-0.07cm}+\hspace*{-0.07cm}\hat n^2\|\z^{K_0}\hspace*{-0.07cm}-\hspace*{-0.07cm}\uu^*\|_L^2\right]\leq(\hat n^2-\hat n)\left( D(\uu^0)\hspace*{-0.07cm}-\hspace*{-0.07cm} D(\uu^*)\right)\hspace*{-0.07cm}+\hspace*{-0.07cm}\hat n^2\|\z^0\hspace*{-0.07cm}-\hspace*{-0.07cm}\uu^*\|_L^2,\label{cont30}
\end{eqnarray}
where we use $H^0=h(\uu^0)$ and $\frac{1}{\theta^2_{-1}}=\hat n^2-\hat n$. Summing (\ref{cont29}) over $k=K_0,K_0+1,\cdots,K$ and using $h(\uu^{k+1})\leq H^{K+1}$, we have
\begin{eqnarray}
\begin{aligned}
&\frac{\E_{\xi_K}[ D(\uu^{K+1})- D(\uu)]}{\theta_K^2}+\hat n^2\E_{\xi_K}[\|\z^{K+1}-\uu\|_L^2]\\
\leq&\frac{\E_{\xi_K}[d(\uu^{K_0})+H^{K_0}- D(\uu)]}{\theta_{K_0-1}^2}+\hat n^2\E_{\xi_K}[\|\z^{K_0}-\uu\|_L^2]\notag\\
&+\sum_{k=K_0}^K\frac{\E_{\xi_K}\left[\sum_i\sigma_1(\uu_i,\widetilde\z_i^k)+\sigma_2(\uu,\vv^k)\right]}{\theta_k}-\frac{\hat n^2}{2}\sum_{k=K_0}^K\E_{\xi_K}[\|\z^{k+1}-\z^k\|_L^2].\notag
\end{aligned}
\end{eqnarray}
Letting $\uu=\uu^*$, from (\ref{cont30}), $\sigma_1(\uu_i,\widetilde\z_{i}^k)\leq 0$ and $\sigma_2(\uu,\vv_{i}^{k})\leq 0$, we can immediately have (\ref{cont1}), (\ref{cont3}) and (\ref{cont34}).
On the other hand, we also have
\begin{eqnarray}
\begin{aligned}\label{cont2}
&-\sum_{k=K_0}^K\frac{\E_{\xi_K}\left[\sum_i\sigma_1(\uu_i,\widetilde\z_i^k)+\sigma_2(\uu,\vv^k)\right]}{\theta_k}\\
\leq&\frac{\E_{\xi_K}[d(\uu^{K_0})+H^{K_0}- D(\uu)]}{\theta_{K_0-1}^2}+\hat n^2\E_{\xi_K}[\|\z^{K_0}-\uu\|_L^2]-\frac{\E_{\xi_K}[ D(\uu^{K+1})]- D(\uu)}{\theta_K^2}.
\end{aligned}
\end{eqnarray}
From (\ref{pd_relation1}), for any $\uu\in\DS$ we have
\begin{eqnarray}
\begin{aligned}
&-\sum_{k=K_0}^K\frac{\E_{\xi_K}\left[\sum_{i=1}^n\sigma_1(\uu_i,\widetilde\z_i^k)+\sigma_2(\uu,\vv^k)\right]}{\theta_k}\notag\\
=&\<\sum_{k=K_0}^K\frac{\E_{\xi_K}[\A^T\x^*(\vv^k)-n\y^k]}{n\theta_k},\uu_{1:n}\>+\<\sum_{k=K_0}^K\frac{\E_{\xi_K}[\B\x^*(\vv^k)+\b]}{\theta_k},\uu_{n+1:n+p}\>\\
&+\<\left[\sum_{k=K_0}^K\frac{\E_{\xi_K}[g_1(\x^*(\vv^k))]}{\theta_k},\cdots,\sum_{k=K_0}^K\frac{\E_{\xi_K}[g_m(\x^*(\vv^k))]}{\theta_k}\right]^T,\uu_{n+p+1:n+p+m}\>\notag\\
&+\sum_{k=K_0}^K\frac{ D(\uu)}{\theta_k}+\sum_{k=K_0}^K\frac{\E_{\xi_K}[f(\x^*(\vv^k))]}{\theta_k}+\frac{1}{n}\sum_{k=K_0}^K\frac{\E_{\xi_K}[\phi(n\y^k)]}{\theta_k}\notag
\end{aligned}
\end{eqnarray}
\begin{eqnarray}
\begin{aligned}
=&\sum_{k=K_0}^K\frac{1}{\theta_k}\left[\<\frac{\sum_{k=K_0}^K\frac{\E_{\xi_K}[\A^T\x^*(\vv^k)-n\y^k]}{n\theta_k}}{\sum_{k=K_0}^K\frac{1}{\theta_k}},\uu_{1:n}\>+\<\frac{\sum_{k=K_0}^K\frac{\E_{\xi_K}[\B\x^*(\vv^k)+\b]}{\theta_k}}{\sum_{k=K_0}^K\frac{1}{\theta_k}},\uu_{n+1:n+p}\>\right]\notag\\
&+\sum_{k=K_0}^K\frac{1}{\theta_k}\<\left[\frac{\sum_{k=K_0}^K\frac{\E_{\xi_K}[g_1(\x^*(\vv^k))]}{\theta_k}}{\sum_{k=K_0}^K\frac{1}{\theta_k}},\cdots,\frac{\sum_{k=K_0}^K\frac{\E_{\xi_K}[g_m(\x^*(\vv^k))]}{\theta_k}}{\sum_{k=K_0}^K\frac{1}{\theta_k}}\right]^T,\uu_{n+p+1:n+p+m}\>\notag\\
&+\sum_{k=K_0}^K\frac{ D(\uu)}{\theta_k}+\sum_{k=K_0}^K\frac{1}{\theta_k}\frac{\sum_{k=K_0}^K\frac{\E_{\xi_K}[f(\x^*(\vv^k))]}{\theta_k}}{\sum_{k=K_0}^K\frac{1}{\theta_k}} +\frac{1}{n}\sum_{k=K_0}^K\frac{1}{\theta_k}\frac{\sum_{k=K_0}^K\frac{\E_{\xi_K}[\phi(n\y^k)]}{\theta_k}}{\sum_{k=K_0}^K\frac{1}{\theta_k}}\notag\\
\overset{a}\geq&\sum_{k=K_0}^K\frac{1}{\theta_k}\<\frac{1}{n}\E_{\xi_K}[\A^T\hat\x^K-\hat\y^K],\uu_{1:n}\>+\sum_{k=K_0}^K\frac{1}{\theta_k}\<\E_{\xi_K}[\B\hat\x^K+\b],\uu_{n+1:n+p}\>\notag\\
&+\sum_{k=K_0}^K\frac{1}{\theta_k}\<\left[\E_{\xi_K}[g_1(\hat\x^K)],\cdots,\E_{\xi_K}[g_m(\hat\x^K)]\right]^T,\uu_{n+p+1:n+p+m}\>\notag\\
&+\sum_{k=K_0}^K\frac{ D(\uu)}{\theta_k}+\sum_{k=K_0}^K\frac{\E_{\xi_K}[f(\hat\x^K)]}{\theta_k}+\frac{1}{n}\sum_{k=K_0}^K\frac{\E_{\xi_K}[\phi(\hat\y^K)]}{\theta_k}\notag\\
=&\sum_{k=K_0}^K\frac{1}{\theta_k}\left[\<\E_{\xi_K}[\triangle(\hat\x^K,\hat\y^K)],\uu\>+ D(\uu)+\E_{\xi_K}[f(\hat\x^K)]+\frac{1}{n}\E_{\xi_K}[\phi(\hat\y^K)]\right],\notag
\end{aligned}
\end{eqnarray}
where we use the definition of $\hat\x^K$ and $\hat\y^K$, $\uu_{n+p+1:n+p+m}\geq 0$ and Jensen's inequality for $g_i$, $f$ and $\phi_i$ in $\overset{a}\geq$. Thus, from (\ref{cont2}) and the second property in Lemma \ref{theta_lemma}, we have
\begin{eqnarray}
\begin{aligned}
&\left(\frac{1}{\theta_K^2}-\frac{1}{\theta_{K_0-1}^2}\right)\E_{\xi_K}\left[\<\triangle(\hat\x^K,\hat\y^K),\uu\>+f(\hat\x^K)+\frac{1}{n}\phi(\hat\y^K)\right]\notag\\
\leq& \frac{\E_{\xi_K}[d(\uu^{K_0})+H^{K_0}]}{\theta_{K_0-1}^2}+\hat n^2\E_{\xi_K}[\|\z^{K_0}-\uu\|_L^2]-\frac{\E_{\xi_K}[ D(\uu^{K+1})]}{\theta_K^2}.\notag
\end{aligned}
\end{eqnarray}
where we eliminate $D(\uu)$ from both sides. Adding $\left(\frac{1}{\theta_K^2}-\frac{1}{\theta_{K_0-1}^2}\right) D(\uu^*)$ to both sides and using $\|\z^{K_0}-\uu\|_L^2\leq 2\|\z^{K_0}-\uu^*\|_L^2+2\|\uu-\uu^*\|_L^2$, we have
\begin{eqnarray}
\begin{aligned}
&\left(\frac{1}{\theta_K^2}-\frac{1}{\theta_{K_0-1}^2}\right)\E_{\xi_K}\left[\<\triangle(\hat\x^K,\hat\y^K),\uu\>+ D(\uu^*)+f(\hat\x^K)+\frac{1}{n}\phi(\hat\y^K)\right]\notag\\
\leq& \E_{\xi_K}\hspace*{-0.05cm}\left[\frac{d(\uu^{K_0})\hspace*{-0.05cm}+\hspace*{-0.05cm}H^{K_0}\hspace*{-0.05cm}-\hspace*{-0.05cm} D(\uu^*)}{\theta_{K_0-1}^2}\hspace*{-0.05cm}+\hspace*{-0.05cm}2\hat n^2\|\z^{K_0}\hspace*{-0.05cm}-\hspace*{-0.05cm}\uu^*\|_L^2\hspace*{-0.05cm}+\hspace*{-0.05cm}2\hat n^2\|\uu\hspace*{-0.05cm}-\hspace*{-0.05cm}\uu^*\|_L^2\right]\hspace*{-0.05cm}-\hspace*{-0.05cm}\frac{\E_{\xi_K}[ D(\uu^{K+1})]\hspace*{-0.05cm}-\hspace*{-0.05cm} D(\uu^*)}{\theta_K^2}\notag\\
\overset{b}\leq&2(\hat n^2-\hat n)\left( D(\uu^0)- D(\uu^*)\right)+2\hat n^2\|\z^0-\uu^*\|_L^2+2\hat n^2\|\uu-\uu^*\|_L^2,\notag
\end{aligned}
\end{eqnarray}
where we use $d(\uu^{K_0})+H^{K_0}\geq  D(\uu^{K_0})\geq D(\uu^*)$ and (\ref{cont30}) in $\overset{b}\leq$.
\end{proof}

Now we are ready to prove the convergence rate of the primal solution. We first consider the primal objective function of problem (\ref{problem3}) in the following lemma.
\begin{lemma}\label{lemma1}
Suppose Assumption \ref{assumption} holds. Define $\hat\x^K=\frac{\sum_{k=K_0}^K\frac{\x^*(\vv^k)}{\theta_k}}{\sum_{k=K_0}^K\frac{1}{\theta_k}}$ and $\hat\y^K=\frac{\sum_{k=K_0}^K\frac{n\y^k}{\theta_k}}{\sum_{k=K_0}^K\frac{1}{\theta_k}}$. Then we have
\begin{eqnarray}
\begin{aligned}
&\left|\E_{\xi_K}[f(\hat\x^K)]+\frac{1}{n}\E_{\xi_K}[\phi(\hat\y^K)]-f(\x^*)-\frac{1}{n}\phi(\y^*)\right|\notag\\
\leq&\frac{2\hat n^2\hspace*{-0.05cm}\left((1\hspace*{-0.05cm}-\hspace*{-0.05cm}\theta_0)\hspace*{-0.05cm}\left(D(\uu^0)\hspace*{-0.05cm}-\hspace*{-0.05cm} D(\uu^*)\right)\hspace*{-0.05cm}+\hspace*{-0.05cm}\|\uu^0\hspace*{-0.05cm}-\hspace*{-0.05cm}\uu^*\|_L^2\hspace*{-0.05cm}+\hspace*{-0.05cm}\|\uu^*\|_L^2\right)}{\frac{1}{\theta_K^2}-\frac{1}{\theta_{K_0-1}^2}}\hspace*{-0.05cm}+\hspace*{-0.05cm}\sqrt{\E_{\xi_K}\hspace*{-0.05cm}\left[\left(\|\widetilde\triangle(\hat\x^K,\hat\y^K)\|_L^*\right)^2\right]}\|\uu^*\|_L,\notag
\end{aligned}
\end{eqnarray}
where
\begin{eqnarray}
\widetilde\triangle(\x,\y)=\left[(\A^T\x-\y)^T/n,(\B\x+\b)^T,\max\{0,g_1(\x)\},\cdots,\max\{0,g_m(\x)\}\right]^T.\notag
\end{eqnarray}
\end{lemma}
\begin{proof}
Define $\hat\uu^*_i= \left\{
  \begin{array}{ll}
    \uu_i^*, & \mbox{if } i\leq n+p,\\
    \uu_i^*, & \mbox{if } i>n+p\mbox{ and }\E_{\xi_K}[g_i(\hat\x^K)]\geq 0,\\
    0,      & \mbox{if }  i>n+p\mbox{ and }\E_{\xi_K}[g_i(\hat\x^K)]< 0.
  \end{array}
\right.$ Since $\uu_{n+p+1:n+p+m}^*\geq \0$, then $\hat\uu_{n+p+1:n+p+m}^*\geq \0$. We also have $\|\hat\uu^*\|_L\leq\|\uu^*\|_L$ and $\|\hat\uu^*-\uu^*\|_L\leq\|\uu^*\|_L$. Moreover, $\hat\uu^*$ is independent on $\xi_K$ since we use $\E_{\xi_K}[g_i(\hat\x^K)]$ in the definition, rather than $g_i(\hat\x^K)$. So we can let $\uu=\hat\uu^*$ in (\ref{cont14}). Define
$\triangle_E(\x,\y)=\left[
\begin{array}{c}
\E_{\xi_K}[(\A^T\x-\y)/n]\\
\E_{\xi_K}[\B\x+\b]\\
\max\{\0,\E_{\xi_K}[g_{1:m}(\x)]\}
\end{array}\right]$, then we have
\begin{eqnarray}
\begin{aligned}\label{cont33}
&\<\E_{\xi_K}[\triangle(\hat\x^K,\hat\y^K)],\hat\uu^*\>\overset{a}=\<\triangle_E(\hat\x^K,\hat\y^K),\hat\uu^*\>\geq-\|\triangle_E(\hat\x^K,\hat\y^K)\|_L^*\|\hat\uu^*\|_L\\
&\overset{b}\geq-\|\E_{\xi_K}[\widetilde\triangle(\hat\x^K,\hat\y^K)]\|_L^*\|\hat\uu^*\|_L\overset{c}\geq -\sqrt{\E_{\xi_K}\left[\left(\|\widetilde\triangle(\hat\x^K,\hat\y^K)\|_L^*\right)^2\right]}\|\hat\uu^*\|_L,
\end{aligned}
\end{eqnarray}
where we use $\hat\uu_{n+p+i}^*=0$ if $\E_{\xi_K}[g_i(\hat\x^K)]<0$ in $\overset{a}=$, $\max\{0,\E_{\xi_K}[a]\}\leq \E_{\xi_K}[\max\{0,a\}]$ in $\overset{b}\geq$ and $(\E[a])^2\leq \E [a^2]$ in $\overset{c}\geq$. Thus, letting $\uu=\hat\uu^*$ in (\ref{cont14}), we have
\begin{eqnarray}
\begin{aligned}
&\E_{\xi_K}[f(\hat\x^K)]+\frac{1}{n}\E_{\xi_K}[\phi(\hat\y^K)]-f(\x^*)-\frac{1}{n}\phi(\y^*)\notag\\
=&\E_{\xi_K}[f(\hat\x^K)]+\frac{1}{n}\E_{\xi_K}[\phi(\hat\y^K)]+ D(\uu^*)\notag\\
\leq&\frac{2\hat n^2\hspace*{-0.05cm}\left( (1\hspace*{-0.05cm}-\hspace*{-0.05cm}\theta_0)\hspace*{-0.05cm}\left(D(\uu^0)\hspace*{-0.05cm}-\hspace*{-0.05cm} D(\uu^*)\right)\hspace*{-0.05cm}+\hspace*{-0.05cm}\|\uu^0\hspace*{-0.05cm}-\hspace*{-0.05cm}\uu^*\|_L^2\hspace*{-0.05cm}+\hspace*{-0.05cm}\|\uu^*\|_L^2\right)}{\frac{1}{\theta_K^2}-\frac{1}{\theta_{K_0-1}^2}}\hspace*{-0.05cm}+\hspace*{-0.05cm}\sqrt{\E_{\xi_K}\hspace*{-0.05cm}\left[\left(\|\widetilde\triangle(\hat\x^K,\hat\y^K)\|_L^*\right)^2\right]}\|\uu^*\|_L.\notag
\end{aligned}
\end{eqnarray}
On the other hand, since $(\x^*,\y^*,\uu^*)$ is a KKT point, we have
\begin{eqnarray}
L_F(\hat\x^K,\hat\y^K,\uu^*)\geq L_F(\x^*,\y^*,\uu^*)=f(\x^*)+\frac{1}{n}\phi(\y^*).\notag
\end{eqnarray}
From the definition in (\ref{lagrangian}), we have
\begin{eqnarray}
\begin{aligned}
&f(\hat\x^K)+\frac{1}{n}\phi(\hat\y^K)-f(\x^*)-\frac{1}{n}\phi(\y^*)\notag\\
\geq& -\frac{1}{n}\<\uu_{1:n}^*,\A^T\hat\x^K-\hat\y^K\>-\<\uu_{n+1:n+p}^*,\B\hat\x^K+\b\>-\sum_{i=1}^m \uu_{n+p+i}^* g_i(\hat\x^K)\notag\\
\overset{d}\geq& -\frac{1}{n}\<\uu_{1:n}^*,\A^T\hat\x^K-\hat\y^K\>-\<\uu_{n+1:n+p}^*,\B\hat\x^K+\b\>-\sum_{i=1}^m \uu_{n+p+i}^* \max\{0,g_i(\hat\x^K)\}\notag\\
\geq&-\|\widetilde\triangle(\hat\x^K,\hat\y^K)\|_L^*\|\uu^*\|_L,\notag
\end{aligned}
\end{eqnarray}
where we use $\uu^*_{n+p+1:n+p+m}\geq 0$ in $\overset{d}\geq$. So we have
\begin{eqnarray}
\begin{aligned}\label{cont48}
&\E_{\xi_K}[f(\hat\x^K)]+\frac{1}{n}\E_{\xi_K}[\phi(\hat\y^K)]-f(\x^*)-\frac{1}{n}\phi(\y^*)\\
\geq&-\E_{\xi_K}\left[\|\widetilde\triangle(\hat\x^K,\hat\y^K)\|_L^*\right]\|\uu^*\|_L\geq-\sqrt{\E_{\xi_K}\left[\left(\|\widetilde\triangle(\hat\x^K,\hat\y^K)\|_L^*\right)^2\right]}\|\uu^*\|_L,
\end{aligned}
\end{eqnarray}
which completes the proof.
\end{proof}

\subsubsection{Constraint Functions}

Lemma \ref{lemma2} establishes the convergence rate for the constraint functions of problem (\ref{problem3}). A straightforward extension of (\ref{cont53}) only leads to $\|\E_{\xi_K}[\widetilde\triangle(\hat\x^K,\hat\y^K)]\|_L^*\leq O\left(\frac{1}{K^2}\right)$, which is less interesting since the expectation is inside the norm. We should take the expectation outside the norm, i.e., $\E_{\xi_K}[\|\widetilde\triangle(\hat\x^K,\hat\y^K)\|_L^*]\leq O\left(\frac{1}{K^2}\right)$. The later one cannot be attained by the simple techniques in (\ref{cont53}) and requires more skillful tricks. The proof sketch of Lemma \ref{lemma2} is that we first establish a recursion (\ref{cont18}) and then using (\ref{cont1}) and the definition of $\s^k$ in (\ref{cont11}) to bound the constraint functions.
\begin{lemma}\label{lemma2}
Suppose Assumption \ref{assumption} holds. Define $\hat\x^K=\frac{\sum_{k=K_0}^K\frac{\x^*(\vv^k)}{\theta_k}}{\sum_{k=K_0}^K\frac{1}{\theta_k}}$ and $\hat\y^K=\frac{\sum_{k=K_0}^K\frac{n\y^k}{\theta_k}}{\sum_{k=K_0}^K\frac{1}{\theta_k}}$. Then we have
\begin{eqnarray}
\begin{aligned}\label{cont35}
\sqrt{\E_{\xi_K}\left[\left(\|\widetilde\triangle(\hat\x^K,\hat\y^K)\|_L^*\right)^2\right]}\leq \frac{\sqrt{48}\hat n^2\sqrt{(1-\theta_0)\left( D(\uu^0)- D(\uu^*)\right)+\|\uu^0-\uu^*\|_L^2}}{\frac{1}{\theta_K^2}-\frac{1}{\theta_{K_0-1}^2}}.
\end{aligned}
\end{eqnarray}
\end{lemma}
\begin{proof}
From the update of $\widetilde\z^k$ and the definitions of $\y^k$ and $\nabla d(\vv^k)$ in (\ref{def_y}) and (\ref{gradient}), we have
\begin{eqnarray}
&&\widetilde\z^k_{i}=\z_{i}^k-\frac{\nabla_i d(\vv^k)+\y_i^k}{2\hat n\theta_k L_i}=\z_{i}^k-\frac{(-\A_i^T\x^*(\vv^k)+n\y_i^k)/n}{2\hat n\theta_k L_i},i\leq n,\label{cont41}\\
&&\widetilde\z^k_{i}=\z^k_{i}-\frac{\nabla d_{i}(\vv^k)}{2\hat n\theta_k L_i}=\z^k_{i}+\frac{\B_{i,:}^T\x^*(\vv^k)+\b_i}{2\hat n\theta_k L_i},n<i\leq n+p,\notag\\
&&\widetilde\z^k_{i}=\left[\z^k_{i}-\frac{\nabla d_{i}(\vv^k)}{2\hat n\theta_k L_i}\right]_+=\z_i^k+\frac{\max\left\{g_i(\x^*(\vv^k)),-2\hat n\theta_k L_i\z_i^k\right\}}{2\hat n\theta_k L_i},i>n+p.\notag
\end{eqnarray}
Define $\pi^k\in\R^{\hat n}$ such that $\pi_i^k=\left\{
  \begin{array}{ll}
    \frac{\A_i^T\x^*(\vv^k)-n\y_i^k}{n}, & i\leq n,\\
    \B_{i,:}^T\x^*(\vv^k)+\b_i, & n<i\leq n+p,\\
    \max\left\{g_i(\x^*(\vv^k)),-2\hat n\theta_k L_i\z_i^k\right\}, & i>n+p,
  \end{array}
\right.$ then we have
\begin{eqnarray}
&\widetilde\z_i^k-\z_i^k= \frac{\pi_i^k}{2\hat n\theta_k L_i}\label{cont10}
\end{eqnarray}
and
\begin{eqnarray}
&\E_{i_k|\xi_{k-1}}[\z_i^{k+1}]=\frac{1}{\hat n}\widetilde\z_i^k+(1-\frac{1}{\hat n})\z_i^k=\frac{1}{\hat n}\left(\z_i^k+\frac{\pi_i^k}{2\hat n\theta_k L_i}\right)+(1-\frac{1}{\hat n})\z_i^k=\z_i^k+\frac{\pi_i^k}{2\hat n^2\theta_k L_i}.\notag
\end{eqnarray}
Define $\g^k\in\R^{\hat n}$ and $\s^k\in\R^{\hat n}$ such that
\begin{eqnarray}
\g_i^k=\frac{\pi_i^k}{2\hat n^2\theta_kL_i}+\z_i^k-\z_i^{k+1} \quad\mbox{ and }\quad \s_i^k=\sum_{t=K_0}^k\g_i^t\mbox{ (specially, } \s_i^k=0,k<K_0),\label{cont11}
\end{eqnarray}
then we get $\E_{i_k|\xi_{k-1}}[\g_i^k]= 0$ and $\E_{i_k|\xi_{k-1}}[\s_i^k]= \s_i^{k-1}$. Moreover, for $k\geq K_0$, we have
\begin{eqnarray}
\begin{aligned}
\E_{i_k|\xi_{k-1}}[\|\g^k\|_L^2]\overset{a}=&\E_{i_k|\xi_{k-1}}\left[\left\|\frac{1}{\hat n}(\widetilde\z^k-\z^k)+\z^k-\z^{k+1}\right\|_L^2\right]\notag\\
=&\frac{1}{\hat n}\sum_{i=1}^{\hat n}\left[ L_{i}\left\|\frac{1}{\hat n}(\widetilde\z_{i}^k-\z_{i}^k)+\z_{i}^k-\widetilde\z_{i}^k\right\|^2+\sum_{j\neq i}L_j\left\|\frac{1}{\hat n}(\widetilde\z_j^k-\z_j^k)+\z_j^k-\z_j^k\right\|^2 \right]\notag\\
=&\left(\frac{1}{\hat n}\left(1-\frac{1}{\hat n}\right)^2+\frac{\hat n-1}{\hat n^3}\right)\sum_{i=1}^{\hat n}L_i\|\widetilde\z_i^k-\z_i^k\|^2\notag\\
\leq&\frac{1}{\hat n}\sum_{i=1}^{\hat n}L_i\|\widetilde\z_i^k-\z_i^k\|^2\overset{b}=\E_{i_k|\xi_{k-1}}[\|\z^{k+1}-\z^k\|_L^2],\notag
\end{aligned}
\end{eqnarray}
where we use (\ref{cont10}) and (\ref{cont11}) in $\overset{a}=$ and (\ref{cont61}) in $\overset{b}=$. Then for any $k\geq K_0$, we have
\begin{eqnarray}
\begin{aligned}\label{cont18}
\E_{\xi_{k}}[\|\s^k\|_L^2]=&\E_{\xi_{k-1}}\left[\E_{i_k|\xi_{k-1}}[\|\s^k\|_L^2]\right]\\
=&\E_{\xi_{k-1}}\left[ \E_{i_k|\xi_{k-1}}\left[\|\s^k-\E_{i_k|\xi_{k-1}}[\s^k]+\E_{i_k|\xi_{k-1}}[\s^k]\|_L^2\right] \right]\\
=&\E_{\xi_{k-1}}\left[ \E_{i_k|\xi_{k-1}}\left[\|\s^k-\E_{i_k|\xi_{k-1}}[\s^k]\|_L^2\right]+\|\E_{i_k|\xi_{k-1}}[\s^k]\|_L^2 \right]\\
=&\E_{\xi_{k-1}}\left[ \E_{i_k|\xi_{k-1}}[\|\g^k\|_L^2]+\|\s^{k-1}\|_L^2 \right]\\
\leq&\E_{\xi_k}[\|\z^{k+1}-\z^k\|_L^2]+\E_{\xi_{k-1}}[\|\s^{k-1}\|_L^2].
\end{aligned}
\end{eqnarray}
Summing over $k=K_0,K_0+1,\cdots,K$ and using $\s^{K_0-1}=\0$, we have
\begin{eqnarray}
\begin{aligned}\label{cont19}
\E_{\xi_{K}}\left[\|\s^K\|_L^2\right]\leq&\sum_{k=K_0}^K\E_{\xi_k}[\|\z^{k+1}-\z^k\|_L^2]=\sum_{k=K_0}^K\E_{\xi_{K}}[\|\z^{k+1}-\z^k\|_L^2]\\
\overset{a}\leq&2(1-\theta_0)\left( D(\uu^0)- D(\uu^*)\right)+2\|\uu^0-\uu^*\|_L^2,
\end{aligned}
\end{eqnarray}
where we use (\ref{cont1}) in $\overset{a}\leq$. On the other hand, from the definition of $\s^K$ and $\g^k$, we have
\begin{eqnarray}
\begin{aligned}
&\|\s^K\|_L^2=\left\|\sum_{k=K_0}^K\g^k\right\|_L^2=\sum_iL_i\left\|\sum_{k=K_0}^K\left(\frac{\pi_i^k}{2\hat n^2\theta_k L_i}+\z_i^k-\z_i^{k+1}\right)\right\|^2\notag\\
&=\sum_i\hspace*{-0.08cm}L_i\hspace*{-0.08cm}\left\|\sum_{k=K_0}^K\hspace*{-0.08cm}\frac{\pi_i^k}{2\hat n^2\theta_k L_i}\hspace*{-0.08cm}+\hspace*{-0.08cm}\z_i^{K_0}\hspace*{-0.08cm}-\hspace*{-0.08cm}\z_i^{K+1}\right\|^2\hspace*{-0.08cm}\geq\hspace*{-0.08cm} \frac{1}{3}\hspace*{-0.08cm}\left(\left\|\sum_{k=K_0}^K\frac{\pi^k}{2\hat n^2\theta_k}\right\|_L^*\right)^2\hspace*{-0.08cm}-\hspace*{-0.08cm}\|\z^{K_0}\hspace*{-0.08cm}-\hspace*{-0.08cm}\uu^*\|_L^2\hspace*{-0.08cm}-\hspace*{-0.08cm}\|\z^{K+1}\hspace*{-0.08cm}-\hspace*{-0.08cm}\uu^*\|_L^2. \notag
\end{aligned}
\end{eqnarray}
So from (\ref{cont19}), (\ref{cont3}) and (\ref{cont30}), we have
\begin{eqnarray}\
\begin{aligned}\label{cont32}
&\E_{\xi_K}\hspace*{-0.09cm}\left[\hspace*{-0.09cm}\left(\left\|\sum_{k=K_0}^K\frac{\pi^k}{2\hat n^2\theta_k}\right\|_L^*\right)^2\right]\leq 3\E_{\xi_K}\hspace*{-0.09cm}\left[\|\s^K\|_L^2\right]\hspace*{-0.09cm}+\hspace*{-0.09cm}3\E_{\xi_K}\hspace*{-0.09cm}\left[\|\z^{K+1}\hspace*{-0.09cm}-\hspace*{-0.09cm}\uu^*\|_L^2\right]\hspace*{-0.09cm}+\hspace*{-0.09cm}3\E_{\xi_K}\hspace*{-0.09cm}\left[\|\z^{K_0}\hspace*{-0.09cm}-\hspace*{-0.09cm}\uu^*\|_L^2\right]\\
&\leq 12(1-\theta_0)\left( D(\uu^0)\hspace*{-0.09cm}-\hspace*{-0.09cm} D(\uu^*)\right)\hspace*{-0.09cm}+\hspace*{-0.09cm}12\|\uu^0\hspace*{-0.09cm}-\hspace*{-0.09cm}\uu^*\|_L^2.
\end{aligned}
\end{eqnarray}
For $i\leq n$, we have
\begin{eqnarray}
\begin{aligned}\label{cont20}
&\sum_{k=K_0}^K\frac{\pi_i^k}{2\hat n^2\theta_k}=\frac{1}{2\hat n^2}\sum_{k=K_0}^K\frac{(\A_i^T\x^*(\vv^k)-n\y_i^k)/n}{\theta_k}\\
&=\frac{\sum_{k=K_0}^K\frac{1}{\theta_k}}{2\hat n^2}\frac{\sum_{k=K_0}^K\frac{\A_i^T\x^*(\vv^k)-n\y_i^k}{n\theta_k}}{\sum_{k=K_0}^K\frac{1}{\theta_k}}=\frac{\frac{1}{\theta_K^2}-\frac{1}{\theta_{K_0-1}^2}}{2\hat n^2}(\A_i^T\hat\x^K-\hat\y_i^K)/n.
\end{aligned}
\end{eqnarray}
Similarly, for $n<i\leq n+p$, we have
\begin{eqnarray}
&&\sum_{k=K_0}^K\frac{\pi_i^k}{2\hat n^2\theta_k}=\frac{\frac{1}{\theta_K^2}-\frac{1}{\theta_{K_0-1}^2}}{2\hat n^2}(\B_{i,:}^T\hat\x^K+\b_i),\notag
\end{eqnarray}
and for $n+p< i\leq n+p+m$, we have
\begin{eqnarray}
\begin{aligned}
&\sum_{k=K_0}^K\frac{\pi_i^k}{2\hat n^2\theta_k}\geq\frac{1}{2\hat n^2}\sum_{k=K_0}^K\frac{g_i(\x^*(\vv^k))}{\theta_k}\geq\frac{\sum_{k=K_0}^K\frac{1}{\theta_k}}{2\hat n^2}g_i(\hat\x^K)=\frac{\frac{1}{\theta_K^2}-\frac{1}{\theta_{K_0-1}^2}}{2\hat n^2}g_i(\hat\x^K)\notag\\
\Rightarrow&\left(\sum_{k=K_0}^K\frac{\pi_i^k}{2\hat n^2\theta_k}\right)^2\geq \left(\frac{1}{\theta_K^2}-\frac{1}{\theta_{K_0-1}^2}\right)^2\frac{\left( \max\{0,g_i(\hat\x^K)\}\right)^2}{4\hat n^4}.\notag
\end{aligned}
\end{eqnarray}
Then we have
\begin{eqnarray}
\frac{1}{4\hat n^4}\left(\frac{1}{\theta_K^2}-\frac{1}{\theta_{K_0-1}^2}\right)^2\left(\|\widetilde\triangle(\hat\x^K,\hat\y^K)\|_L^*\right)^2 \leq \left(\left\|\sum_{k=K_0}^K\frac{\pi^k}{2\hat n^2\theta_k}\right\|_L^*\right)^2.\notag
\end{eqnarray}
Taking expectation with respect to $\xi_K$ and from (\ref{cont32}), we can immediately have the conclusion.
\end{proof}
\subsubsection{Proof of Theorem \ref{main_theorem}}
From Lemma \ref{lemma1}, Lemma \ref{lemma2}, $\frac{1}{\theta_K^2}-\frac{1}{\theta_{K_0-1}^2}\geq \left(\frac{K^2}{4}+\hat nK\right)\left(1-\frac{1}{\upsilon}\right)$, $(\E[a])^2\leq \E[a^2]$ and
\begin{eqnarray}
\begin{aligned}\label{cont7}
&\left|\frac{1}{n}\phi(\A^T\hat\x^{K})-\frac{1}{n}\phi(\hat\y^K)\right|\leq\frac{1}{n}\sum_{i=1}^n M|\A_i^T\hat\x^K-\hat\y_i^K|\\
&\leq\sqrt{\sum_{i=1}^nM^2L_i}\left\|\frac{1}{n}(\A^T\hat\x^K-\hat\y^K)\right\|_L^*\leq \sqrt{\sum_{i=1}^nM^2L_i}\|\widetilde\triangle(\hat\x^K,\hat\y^K)\|_L^*,
\end{aligned}
\end{eqnarray}
we have
\begin{eqnarray}
\begin{aligned}
&\left|\E_{\xi_K}\left[F(\hat\x^K)\right]-F(\x^*)\right|\leq\frac{2\hat n^2\left((1-\theta_0)\left(D(\uu^0)- D(\uu^*)\right)+\|\uu^0-\uu^*\|_L^2+\|\uu^*\|_L^2\right)}{\left(\frac{K^2}{4}+\hat nK\right)\left(1-\frac{1}{\upsilon}\right)}\notag\\
&\hspace*{1.5cm}+\frac{6\sqrt{2}\hat n^2(\|\uu^*\|_L+M\sqrt{\sum_{i=1}^nL_i})}{\left(\frac{K^2}{4}+\hat nK\right)\left(1-\frac{1}{\upsilon}\right)}\sqrt{ (1-\theta_0)\left(D(\uu^0)- D(\uu^*)\right)+\|\uu^0-\uu^*\|_L^2}.\notag
\end{aligned}
\end{eqnarray}
From Cauchy-Schwartz inequality, we have the desired result.

\begin{remark}
In Assumption \ref{assumption}, we assume that $\phi_i$ is $M$-Lipschitz continuous. From the above analysis, we can see that it is only used in (\ref{cont7}). Lemmas \ref{lemma1} and \ref{lemma2} do not need this assumption. Moreover, it only affects the convergence rate in the primal space and the analysis in the dual space does not require it.
\end{remark}

\section{Extension under the Quadratic Growth Condition}\label{linar_sec}
In this section, we give the linear complexity under stronger assumptions. Specifically, we use both Assumptions \ref{assumption} and \ref{assumption2} in this section. The quadratic functional growth condition in Assumption \ref{assumption2} is equivalent to the global error bound condition \cite{lewis2018} and is satisfied for broad applications. We give a simple example satisfying Assumption \ref{assumption2} and refer the reader to \cite{Bolte2017,guoyin2013,yang2009,tianbao-2017} for more examples.\newline
\textbf{Example}. Consider problem (\ref{problem}) with strongly convex and smooth $f$ and the simple form (\ref{function_g}) of $g(\x)$. Furthermore, we require that $\sum_{i=1}^n\phi_i^*(\uu_i)$ has the form of $\<\c,\uu\>+P(\uu)$, where $P(\uu)$ is a polyhedral function or an indicator function of a polyhedral set. In this case, $d(\uu)=f^*(-\S\uu)-\<\p,\uu\>$, where $\S$ and $\p$ are defined in (\ref{def_Sq}). It may not be strongly convex since $\S$ may not be full column rank. However, $ D(\uu)$ satisfies the error bound condition \cite{errorbound2,lin2014} and thus satisfies Assumption \ref{assumption2}. The least absolute deviation, SVM and multiclass SVM \cite{zhang-2015-MP} have the required form.

To exploit Assumption \ref{assumption2}, we use Algorithm \ref{alg_arpcg} with restart \cite{Donoghue-2015-NesRestart,zhengqu-2018} and establish the faster convergence rate. Namely, at each iteration, Algorithm \ref{alg_arpcg} is called with fixed and finite iterations with the output of the previous iteration being the initializer of current iteration. We describe the method in Algorithm \ref{alg_arpcg_linear}. We use the inner-outer iteration procedure, rather than the one loop accelerated algorithms with direct support to strongly convex dual functions, e.g., APCG \cite{xiao-2015-siam}, since the quadratic functional growth condition is generally not enough to prove the accelerated linear convergence rate for the one loop accelerated algorithms \cite{Necoara-2019}.

\begin{algorithm}
   \caption{ARDCA with restart}
   \label{alg_arpcg_linear}
\begin{algorithmic}
   \STATE Input $\uu^{-1,K+1}=\uu^{0,0}\in \DS$.
   \FOR{$t=0, 1, 2, \cdots, N$}
   \STATE Run ARDCA($\uu^{t-1,K+1}$,$K_0$,$K$) and output $\uu^{t,K+1}$ and $\hat\x^{t,K}$.
   \ENDFOR
   \STATE Output $\hat\x^{N,K}$.
\end{algorithmic}
\end{algorithm}

Define $\uu^{t,0,*}=\mbox{Proj}_{\DS^*}(\uu^{t,0})=\argmin_{\uu\in\DS^*}\|\uu^{t,0}-\uu\|_L$ to be the nearest optimal solution to $\uu^{t,0}$. Denote $\xi_{t,K}=\{i_{t,0},i_{t,1},\cdots,i_{t,K}\}$ and $\zeta_t=\cup_{r=0}^t\xi_{r,K}$ to be the random sequence, where $i_{t,s}$ is the random index chosen at the $t$-th outer iteration and $s$-th inner iteration of Algorithm \ref{alg_arpcg_linear}. We give the linear convergence of Algorithm \ref{alg_arpcg_linear} for both primal solutions and dual solutions in Theorem \ref{theorem_linear}.

\begin{theorem}\label{theorem_linear}
Suppose Assumptions \ref{assumption} and \ref{assumption2} hold. For Algorithm \ref{alg_arpcg_linear}, we have
\begin{eqnarray}
\begin{aligned}\label{cont47}
(1-\theta_0)\left(\E_{\zeta_N}[D(\uu^{N,0})]-D(\uu^*)\right)+\E_{\zeta_N}[\|\uu^{N,0}-\uu^{N,0,*}\|_L^2]\leq \left(\frac{1+(1-\theta_0)\kappa}{1+\frac{\kappa}{2}\left(\frac{K}{2\hat n}+1\right)^2}\right)^NT_{0.0},
\end{aligned}
\end{eqnarray}
where $T_{0,0}= (1-\theta_0)\left(D(\uu^{0,0})- D(\uu^*)\right)+\|\uu^{0,0}-\uu^{0,0,*}\|_L^2$.

Suppose Assumptions \ref{assumption} and \ref{assumption2} hold. Assume that $\DS^*$ is bounded, i.e., $\|\uu^*\|_L\leq C_{\DS^*},\forall \uu^*\in \DS^*$. Let $K_0\leq\left\lfloor\frac{K}{\upsilon(1+1/\hat n)}+1\right\rfloor$ with any $\upsilon>1$ and $K\geq \hat n$. Then for Algorithm \ref{alg_arpcg_linear}, we have
\begin{eqnarray}
\begin{aligned}\label{cont54}
&\left|\E_{\zeta_N}\left[F(\hat\x^{N,K})\right]-F(\x^*)\right|\leq C_1\left(\frac{1+(1-\theta_0)\kappa}{1+\frac{\kappa}{2}\left(\frac{K}{2\hat n}+1\right)^2}\right)^{N/2}+C_2\left(\frac{1+(1-\theta_0)\kappa}{1+\frac{\kappa}{2}\left(\frac{K}{2\hat n}+1\right)^2}\right)^N,\\
&\E_{\zeta_N}\left[\left\| \left[
  \begin{array}{c}
    \B\hat\x^{N,K}+\b\\
    \max\left\{0,g(\hat\x^{N,K})\right\}
  \end{array}
\right]\right\|_L^*\right] \leq C_3\left(\frac{1+(1-\theta_0)\kappa}{1+\frac{\kappa}{2}\left(\frac{K}{2\hat n}+1\right)^2}\right)^{N/2},
\end{aligned}
\end{eqnarray}
where $C_1=\frac{6C_{\DS^*}\sqrt{T_{0,0}} + 6M\sqrt{\sum_{i=1}^nL_i}\sqrt{T_{0,0}}}{1-1/\nu} + \frac{2\sqrt{m}C_{\DS^*}\sqrt{T_{0,0}}}{\sqrt{1-1/\nu}}$, $C_2=\frac{2T_{0,0}}{1-1/\nu}$ and $C_3=\frac{6\sqrt{T_{0,0}}}{1-1/\nu}$.
\end{theorem}

In the traditional analysis for the restart scheme, the inner iteration number heavily depends on the condition number $\kappa$ \cite{zhengqu-2017,zhengqu-2018}. Specifically, letting $\uu^*=\uu^{t,0,*}$ in (\ref{cont34}) and from Assumption \ref{assumption2}, we have
\begin{eqnarray}
\begin{aligned}\notag
&&\E_{\zeta_N}[ D(\uu^{t+1,0})]- D(\uu^*)\leq\hat n^2\theta_K^2\left( 1-\theta_0+\frac{1}{\kappa}\right)\left(\E_{\zeta_N}[D(\uu^{t,0})]- D(\uu^*)\right).
\end{aligned}
\end{eqnarray}
To make $\hat n^2\theta_K^2\left( 1-\theta_0+\frac{1}{\kappa}\right)<1$, we should require $K=O\left(\hat n+\frac{\hat n}{\sqrt{\kappa}}\right)$, otherwise, the traditional analysis cannot prove the decrement of the objective. However, in practice, $\kappa$ is often difficult to estimate. Different from the traditional analysis, in Theorem \ref{theorem_linear}, we show the linear convergence when the algorithm restarts at any period. In other words, $\frac{1+(1-\theta_0)\kappa}{1+\frac{\kappa}{2}\left(\frac{K}{2\hat n}+1\right)^2}<1$ for any $K\geq\hat n$. Our analysis applies to the problems where $\kappa$ is unknown.

\begin{remark}
Our result (\ref{cont47}) is motivated by \cite{zhengqu-2018}. However, when applied to the dual problem (\ref{dual_problem}), \cite{zhengqu-2018} requires that $D(\uu)$ has the unique optimal dual solution $\uu^*$ and needs a stronger quadratic functional growth condition of
\begin{equation}
\kappa\|\uu-\uu^*\|_L^2\leq D(\uu)-D(\uu^*).\label{cont55}
\end{equation}
As a comparison, in Assumption \ref{assumption2}, we do not need the uniqueness of the optimal dual solution and only assume
\begin{equation}
\kappa\|\uu-\mbox{Proj}_{\DS^*}(\uu)\|_L^2\leq D(\uu)-D(\uu^*).\label{cont56}
\end{equation}
Comparing (\ref{cont56}) with (\ref{cont55}), we can see that (\ref{cont55}) can deduce (\ref{cont56}) and not vice versa. The analysis in \cite{zhengqu-2018} cannot be applied under our assumptions since a critical property in \cite[Equ. (28)]{zhengqu-2018} does not hold any more. To deal with the more general assumption (\ref{cont56}), we develop a totally different proof framework and it is much simpler and more general.
\end{remark}

Let us compare the complexity of Algorithm \ref{alg_arpcg_linear} with that of randomized dual coordinate ascent, i.e., $O\left(\left(\hat n+\frac{\hat n}{\kappa}\right)\log\frac{1}{\epsilon}\right)$ \cite{lu-2015-MP}. If $\kappa<1$, Algorithm \ref{alg_arpcg_linear} attains the optimal complexity of $O\left(\left(\hat n+\frac{\hat n}{\sqrt{\kappa}}\right)\log\frac{1}{\epsilon}\right)$ for both the primal solutions and dual solutions when $K=O\left(\hat n+\frac{\hat n}{\sqrt{\kappa}}\right)$. When $\hat n\leq K\leq \hat n+\frac{\hat n}{\kappa}$, Algorithm \ref{alg_arpcg_linear} outperforms the randomized dual coordinate ascent\footnote{Please see the details in Appendix F.}. When $K$ is larger than $\hat n+\frac{\hat n}{\kappa}$, Algorithm \ref{alg_arpcg_linear} performs worse than randomized dual coordinate ascent. On the other hand, if $\kappa>1$, the complexity of Algorithm \ref{alg_arpcg_linear} \textcolor{red}{has the same order of magnitude as} that of randomized dual coordinate ascent when $\hat n\leq K\leq \hat n+\frac{\hat n}{\kappa}$. When $K>\hat n+\frac{\hat n}{\kappa}$, Algorithm \ref{alg_arpcg_linear} performs worse. For most practical problems, we have $\kappa<1$ and thus Algorithm \ref{alg_arpcg_linear} is a better and safe choice for a wide range of $K$.
\subsection{Convergence Rate Analysis of the Dual Solutions}

In this section, we prove (\ref{cont47}). We first consider one outer iteration of Algorithm \ref{alg_arpcg_linear} and use the symbols in Algorithm \ref{alg_arpcg_equ} for simplicity. In the following lemma, we show that $\uu^{k+1}$ is the convex combination of $\uu^1,\cdots,\uu^k$ and $\z^{k+1}$.
\begin{lemma}\label{lemma_convex_bi}
For Algorithm \ref{alg_arpcg_equ}, we have
\begin{eqnarray}
\uu^{k+1}=\frac{\theta_k}{\theta_0}\z^{k+1}+\theta_k\sum_{t=1}^k\left( \frac{\theta_t(1-\theta_0)^{k-t}}{\theta_{t-1}^2}-\frac{(1-\theta_0)^{k+1-t}}{\theta_{t-1}} \right)\uu^t,\label{convex_u}
\end{eqnarray}
where $\frac{\theta_t(1-\theta_0)^{k-t}}{\theta_{t-1}^2}-\frac{(1-\theta_0)^{k+1-t}}{\theta_{t-1}}>0$ and
\begin{eqnarray}
\frac{\theta_k}{\theta_0}+\theta_k\sum_{t=1}^k\left( \frac{\theta_t(1-\theta_0)^{k-t}}{\theta_{t-1}^2}-\frac{(1-\theta_0)^{k+1-t}}{\theta_{t-1}} \right)=1.\label{convex_u2}
\end{eqnarray}
\end{lemma}
\begin{proof}
For Algorithm \ref{alg_arpcg_equ}, we have
\begin{eqnarray}
\uu^{k+1}=(1-\theta_k)\uu^k+\frac{\theta_k}{\theta_0}\z^{k+1}-\frac{\theta_k(1-\theta_0)}{\theta_0}\z^k.\label{rec_u}
\end{eqnarray}
Decomposing term $(1-\theta_k)\uu^k$ into $(1-\theta_k)\left( 1-\frac{\theta_{k-1}(1-\theta_0)}{\theta_k} \right)\uu^k$ and $(1-\theta_k)\frac{\theta_{k-1}(1-\theta_0)}{\theta_k}\uu^k$ and using (\ref{rec_u}) for the later one, we have
\begin{eqnarray}
\begin{aligned}\notag
\uu^{k+1}
=&(1-\theta_k)\left( 1-\frac{\theta_{k-1}(1-\theta_0)}{\theta_k} \right)\uu^k+\frac{\theta_k}{\theta_0}\z^{k+1}-\frac{\theta_k(1-\theta_0)}{\theta_0}\z^k\\
&+(1-\theta_k)\frac{\theta_{k-1}(1-\theta_0)}{\theta_k}\left[(1-\theta_{k-1})\uu^{k-1}+\frac{\theta_{k-1}}{\theta_0}\z^k-\frac{\theta_{k-1}(1-\theta_0)}{\theta_0}\z^{k-1}\right]\\
=&(1\hspace*{-0.01cm}-\hspace*{-0.01cm}\theta_k)\left( 1\hspace*{-0.01cm}-\hspace*{-0.01cm}\frac{\theta_{k-1}(1\hspace*{-0.01cm}-\hspace*{-0.01cm}\theta_0)}{\theta_k} \right)\uu^k\hspace*{-0.01cm}+\hspace*{-0.01cm}\frac{\theta_k}{\theta_0}\z^{k+1}\hspace*{-0.01cm}+\hspace*{-0.01cm}\frac{\theta_k\theta_{k-1}(1\hspace*{-0.01cm}-\hspace*{-0.01cm}\theta_0)}{\theta_{k-2}^2}\uu^{k-1}\hspace*{-0.01cm}-\hspace*{-0.01cm}\frac{\theta_k(1\hspace*{-0.01cm}-\hspace*{-0.01cm}\theta_0)^2}{\theta_0}\z^{k-1}.
\end{aligned}
\end{eqnarray}
Decomposing $\frac{\theta_k\theta_{k-1}(1-\theta_0)}{\theta_{k-2}^2}\uu^{k-1}$ into $\frac{\theta_k\theta_{k-1}(1-\theta_0)}{\theta_{k-2}^2}\left(1-\frac{\theta_{k-2}(1-\theta_0)}{\theta_{k-1}}\right)\uu^{k-1}$ and $\frac{\theta_k\theta_{k-1}(1-\theta_0)}{\theta_{k-2}^2}\frac{\theta_{k-2}(1-\theta_0)}{\theta_{k-1}}\uu^{k-1}$ and using (\ref{rec_u}) for the later one again, we have
\begin{eqnarray}
\begin{aligned}\notag
\uu^{k+1}=&(1\hspace*{-0.01cm}-\hspace*{-0.01cm}\theta_k)\hspace*{-0.01cm}\left(\hspace*{-0.01cm} 1\hspace*{-0.01cm}-\hspace*{-0.01cm}\frac{\theta_{k-1}(1\hspace*{-0.01cm}-\hspace*{-0.01cm}\theta_0)}{\theta_k} \hspace*{-0.01cm}\right)\hspace*{-0.01cm}\uu^k\hspace*{-0.01cm}+\hspace*{-0.01cm}\frac{\theta_k}{\theta_0}\z^{k+1}\hspace*{-0.01cm}+\hspace*{-0.01cm}\frac{\theta_k\theta_{k-1}(1\hspace*{-0.01cm}-\hspace*{-0.01cm}\theta_0)}{\theta_{k-2}^2}\hspace*{-0.01cm}\left(\hspace*{-0.01cm}1\hspace*{-0.01cm}-\hspace*{-0.01cm}\frac{\theta_{k-2}(1\hspace*{-0.01cm}-\hspace*{-0.01cm}\theta_0)}{\theta_{k-1}}\hspace*{-0.01cm}\right)\hspace*{-0.01cm}\uu^{k-1}\\
&-\hspace*{-0.11cm}\frac{\theta_k\hspace*{-0.04cm}(\hspace*{-0.04cm}1\hspace*{-0.1cm}-\hspace*{-0.1cm}\theta_0\hspace*{-0.04cm})^2}{\theta_0}\z^{k\hspace*{-0.01cm}-\hspace*{-0.01cm}1}\hspace*{-0.1cm}+\hspace*{-0.1cm}\frac{\theta_k\theta_{k-1}\hspace*{-0.04cm}(\hspace*{-0.04cm}1\hspace*{-0.1cm}-\hspace*{-0.1cm}\theta_0\hspace*{-0.04cm})}{\theta_{k-2}^2}\frac{\theta_{k-2}\hspace*{-0.04cm}(\hspace*{-0.04cm}1\hspace*{-0.1cm}-\hspace*{-0.1cm}\theta_0\hspace*{-0.04cm})}{\theta_{k-1}}\hspace*{-0.1cm}\left(\hspace*{-0.1cm}\hspace*{-0.04cm}(\hspace*{-0.04cm}1\hspace*{-0.1cm}-\hspace*{-0.1cm}\theta_{k-2}\hspace*{-0.04cm})\uu^{k\hspace*{-0.01cm}-\hspace*{-0.01cm}2}\hspace*{-0.1cm}+\hspace*{-0.1cm}\frac{\theta_{k-2}}{\theta_0}\z^{k\hspace*{-0.01cm}-\hspace*{-0.01cm}1}\hspace*{-0.1cm}-\hspace*{-0.1cm}\frac{\theta_{k-2}\hspace*{-0.04cm}(\hspace*{-0.04cm}1\hspace*{-0.1cm}-\hspace*{-0.1cm}\theta_0\hspace*{-0.04cm})}{\theta_0}\z^{k\hspace*{-0.01cm}-\hspace*{-0.01cm}2}\hspace*{-0.1cm}\right)\\
=&(1\hspace*{-0.01cm}-\hspace*{-0.01cm}\theta_k)\left( 1\hspace*{-0.01cm}-\hspace*{-0.01cm}\frac{\theta_{k-1}(1\hspace*{-0.01cm}-\hspace*{-0.01cm}\theta_0)}{\theta_k} \right)\uu^k\hspace*{-0.01cm}+\hspace*{-0.01cm}\frac{\theta_k}{\theta_0}\z^{k+1}\hspace*{-0.01cm}+\hspace*{-0.01cm}\frac{\theta_k\theta_{k-1}(1\hspace*{-0.01cm}-\hspace*{-0.01cm}\theta_0)}{\theta_{k-2}^2}\left(1\hspace*{-0.01cm}-\hspace*{-0.01cm}\frac{\theta_{k-2}(1\hspace*{-0.01cm}-\hspace*{-0.01cm}\theta_0)}{\theta_{k-1}}\right)\uu^{k-1}\\
&+\frac{\theta_k\theta_{k-2}(1-\theta_0)^2}{\theta_{k-3}^2}\uu^{k-2}-\frac{\theta_k(1-\theta_0)^3}{\theta_0}\z^{k-2}.
\end{aligned}
\end{eqnarray}
Do the above operations recursively, we have
\begin{eqnarray}
\begin{aligned}\notag
\uu^{k+1}=&\frac{\theta_k}{\theta_0}\z^{k+1}\hspace*{-0.07cm}+\hspace*{-0.07cm}\sum_{t=1}^k\hspace*{-0.07cm} \frac{\theta_k\theta_t(1\hspace*{-0.07cm}-\hspace*{-0.07cm}\theta_0)^{k-t}}{\theta_{t-1}^2}\hspace*{-0.07cm}\left( \hspace*{-0.07cm} 1\hspace*{-0.07cm}-\hspace*{-0.07cm}\frac{\theta_{t-1}(1\hspace*{-0.07cm}-\hspace*{-0.07cm}\theta_0)}{\theta_t} \hspace*{-0.07cm}\right)\hspace*{-0.07cm}\uu^t\hspace*{-0.07cm}+\hspace*{-0.07cm}\frac{\theta_k\theta_0(1\hspace*{-0.07cm}-\hspace*{-0.07cm}\theta_0)^k}{\theta_{-1}^2}\uu^0\hspace*{-0.07cm}-\hspace*{-0.07cm}\frac{\theta_k(1\hspace*{-0.07cm}-\hspace*{-0.07cm}\theta_0)^{k+1}}{\theta_0}\z^0\\
\overset{a}=&\frac{\theta_k}{\theta_0}\z^{k+1}+\theta_k\sum_{t=1}^k\left( \frac{\theta_t(1-\theta_0)^{k-t}}{\theta_{t-1}^2}-\frac{(1-\theta_0)^{k+1-t}}{\theta_{t-1}} \right)\uu^t,
\end{aligned}
\end{eqnarray}
where we use $\frac{1}{\theta_{-1}^2}=\frac{1-\theta_0}{\theta_0^2}$ and $\uu^0=\z^0$ in $\overset{a}=$. On the other hand, we can easily check that
\begin{eqnarray}
\begin{aligned}\notag
&\sum_{t=1}^k\left( \frac{\theta_t(1-\theta_0)^{k-t}}{\theta_{t-1}^2}-\frac{(1-\theta_0)^{k+1-t}}{\theta_{t-1}} \right)=\sum_{t=1}^k\left( \frac{(1-\theta_t)(1-\theta_0)^{k-t}}{\theta_t}-\frac{(1-\theta_0)^{k-(t-1)}}{\theta_{t-1}} \right)\\
&=\frac{1}{\theta_k}-\sum_{t=1}^k(1-\theta_0)^{k-t}-\frac{(1-\theta_0)^k}{\theta_0}=\frac{1}{\theta_k}-\frac{1-(1-\theta_0)^k}{\theta_0}-\frac{(1-\theta_0)^k}{\theta_0}=\frac{1}{\theta_k}-\frac{1}{\theta_0},
\end{aligned}
\end{eqnarray}
which leads to (\ref{convex_u2}). Next, we prove $\frac{\theta_t(1-\theta_0)^{k-t}}{\theta_{t-1}^2}-\frac{(1-\theta_0)^{k+1-t}}{\theta_{t-1}}>0$, which is equivalent to $\frac{\theta_t}{\theta_{t-1}}>1-\theta_0$, which is true since $\frac{\theta_t}{\theta_{t-1}}=\sqrt{1-\theta_t}\geq 1-\theta_t>1-\theta_0$.
\end{proof}

Based on Lemma \ref{lemma_convex_bi}, we can establish the decreasing of $(1-\theta_0)\left(\E_{\xi_K}[D(\uu^{k+1})]-D(\uu^*)\right)+\E_{\xi_K}[\|\uu^{k+1}-\uu^{k+1,*}\|_L^2]$ for any $\theta_k$ in the following lemma. The proof sketch of Lemma \ref{linar_lemma} is that we first establish a recursion (\ref{cont44}) based on Lemma \ref{lemma_convex_bi} and (\ref{cont34}) and then prove the result in Lemma \ref{linar_lemma} by induction.
\begin{lemma}\label{linar_lemma}
Suppose Assumptions \ref{assumption} and \ref{assumption2} hold. For Algorithm \ref{alg_arpcg_equ}, we have
\begin{eqnarray}
\begin{aligned}\notag
&(1-\theta_0)\left(\E_{\xi_K}[D(\uu^{k+1})]-D(\uu^*)\right)+\E_{\xi_K}[\|\uu^{k+1}-\uu^{k+1,*}\|_L^2]\\
\leq& \frac{1+(1-\theta_0)\kappa}{1+\frac{\kappa\theta_0^2}{2\theta_k^2}}\left((1-\theta_0)\left(D(\uu^0)-D(\uu^*)\right)+\|\uu^0-\uu^{0,*}\|_L^2\right),
\end{aligned}
\end{eqnarray}
where we define $\uu^{k,*}=\mbox{Proj}_{\DS^*}(\uu^k)=\argmin_{\uu\in\DS^*}\|\uu^{k}-\uu\|_L$.
\end{lemma}
\begin{proof}
We consider $(1-\theta_0)\left( D(\uu^{k+1})-D(\uu^*) \right)+\|\uu^{k+1}-\uu^{k+1,*}\|_F^2$. Decomposing the second term into $\sigma_k\|\uu^{k+1}-\uu^{k+1,*}\|_F^2$ and $(1-\sigma_k)\|\uu^{k+1}-\uu^{k+1,*}\|_F^2$ and using Assumption \ref{assumption2} for the first term, we have
\begin{eqnarray}
\begin{aligned}\label{cont42}
&(1-\theta_0)\left(\E_{\xi_K}[D(\uu^{k+1})]-D(\uu^*)\right)+\E_{\xi_K}[\|\uu^{k+1}-\uu^{k+1,*}\|_L^2]\\
\leq& \left(1-\theta_0+\frac{\sigma_k}{\kappa}\right)\left( \E_{\xi_K}[D(\uu^{k+1})]-D(\uu^*)\right)+(1-\sigma_k)\E_{\xi_K}[\|\uu^{k+1}-\uu^{k+1,*}\|_L^2].
\end{aligned}
\end{eqnarray}
From the definition of $\uu^{k+1,*}$, (\ref{convex_u2}) and (\ref{convex_u}), we have
\begin{eqnarray}
\begin{aligned}\notag
&\|\uu^{k+1}-\uu^{k+1,*}\|_L^2\\
\leq& \left\| \uu^{k+1}-\left(\frac{\theta_k}{\theta_0}\uu^{0,*}+\theta_k\sum_{t=1}^k\left( \frac{\theta_t(1-\theta_0)^{k-t}}{\theta_{t-1}^2}-\frac{(1-\theta_0)^{k+1-t}}{\theta_{t-1}} \right)\uu^{t,*}\right) \right\|_L^2\\
\leq&\frac{\theta_k}{\theta_0}\|\z^{k+1}-\uu^{0,*}\|_L^2+\theta_k\sum_{t=1}^k\left( \frac{\theta_t(1-\theta_0)^{k-t}}{\theta_{t-1}^2}-\frac{(1-\theta_0)^{k+1-t}}{\theta_{t-1}} \right)\|\uu^t-\uu^{t,*}\|_L^2.
\end{aligned}
\end{eqnarray}
Plugging it into (\ref{cont42}) and letting $\sigma_k=\frac{\frac{\kappa\theta_0}{\theta_k}-(1-\theta_0)\kappa}{1+\frac{\kappa\theta_0}{\theta_k}}$, we have
\begin{eqnarray}
\begin{aligned}\label{cont43}
&(1-\theta_0)\left(\E_{\xi_K}[D(\uu^{k+1})]-D(\uu^*)\right)+\E_{\xi_K}[\|\uu^{k+1}-\uu^{k+1,*}\|_L^2]\\
\leq& \frac{1+(1-\theta_0)\kappa}{1+\frac{\kappa\theta_0}{\theta_k}}\left(\frac{\theta_0}{\theta_k}\left( \E_{\xi_K}[D(\uu^{k+1})]-D(\uu^*)\right)+\frac{\theta_k}{\theta_0}\E_{\xi_K}[\|\z^{k+1}-\uu^{0,*}\|_L^2]\right.\\
&\hspace*{2.5cm}\left.+\theta_k\sum_{t=1}^k\left( \frac{\theta_t(1-\theta_0)^{k-t}}{\theta_{t-1}^2}-\frac{(1-\theta_0)^{k+1-t}}{\theta_{t-1}} \right)\E_{\xi_K}[\|\uu^t-\uu^{t,*}\|_L^2]\right)\\
\overset{a}\leq& \frac{1+(1-\theta_0)\kappa}{1+\frac{\kappa\theta_0}{\theta_k}}\left(\frac{\theta_k}{\theta_0}T_0+\theta_k\sum_{t=1}^k\left( \frac{\theta_t(1-\theta_0)^{k-t}}{\theta_{t-1}^2}-\frac{(1-\theta_0)^{k+1-t}}{\theta_{t-1}} \right)\E_{\xi_K}[\|\uu^t-\uu^{t,*}\|_L^2]\right),
\end{aligned}
\end{eqnarray}
where we denote $T_k=(1-\theta_0)\left(\E_{\xi_K}[D(\uu^k)]-D(\uu^*)\right)+\E_{\xi_K}[\|\uu^k-\uu^{k,*}\|_L^2]$ for simplicity and use (\ref{cont34}) with $\uu^*=\uu^{0,*}$ in $\overset{a}\leq$. On the other hand, decomposing $\|\uu^t-\uu^{t,*}\|_L^2$ into $\sigma\|\uu^t-\uu^{t,*}\|_L^2$ and $(1-\sigma)\|\uu^t-\uu^{t,*}\|_L^2$, using Assumption \ref{assumption2} for the first term and letting $\sigma=\frac{(1-\theta_0)\kappa}{1+(1-\theta_0)\kappa}$, we have
\begin{eqnarray}
\begin{aligned}\notag
\E_{\xi_K}[\|\uu^t-\uu^{t,*}\|_L^2]\leq \frac{\sigma}{\kappa}\left(\E_{\xi_K}[D(\uu^t)]-D(\uu^*)\right)+(1-\sigma)\E_{\xi_K}[\|\uu^t-\uu^{t,*}\|_L^2]=\frac{T_t}{1+(1-\theta_0)\kappa}.
\end{aligned}
\end{eqnarray}
Plugging it into (\ref{cont43}), we have
\begin{eqnarray}
\begin{aligned}\label{cont44}
&T_{k+1}\leq\frac{1\hspace*{-0.04cm}+\hspace*{-0.04cm}(1\hspace*{-0.04cm}-\hspace*{-0.04cm}\theta_0)\kappa}{1\hspace*{-0.04cm}+\hspace*{-0.04cm}\frac{\kappa\theta_0}{\theta_k}}\hspace*{-0.04cm}\left(\hspace*{-0.04cm}\frac{\theta_k}{\theta_0}T_0\hspace*{-0.04cm}+\hspace*{-0.04cm}\frac{\theta_k}{1\hspace*{-0.04cm}+\hspace*{-0.04cm}(1\hspace*{-0.04cm}-\hspace*{-0.04cm}\theta_0)\kappa}\hspace*{-0.04cm}\sum_{t=1}^k\hspace*{-0.04cm}\left(\hspace*{-0.04cm} \frac{\theta_t(1\hspace*{-0.04cm}-\hspace*{-0.04cm}\theta_0)^{k-t}}{\theta_{t-1}^2}\hspace*{-0.04cm}-\hspace*{-0.04cm}\frac{(1\hspace*{-0.04cm}-\hspace*{-0.04cm}\theta_0)^{k+1-t}}{\theta_{t-1}} \right)\hspace*{-0.04cm}T_t\hspace*{-0.04cm}\right)\hspace*{-0.04cm}.
\end{aligned}
\end{eqnarray}
Next, we prove $T_{k+1}\leq \frac{1+(1-\theta_0)\kappa}{1+\kappa\theta_0^2/(2\theta_k^2)}T_0$ by induction. From (\ref{cont44}), we have $T_1\leq \frac{1+(1-\theta_0)\kappa}{1+\kappa}T_0\leq \frac{1+(1-\theta_0)\kappa}{1+\kappa\theta_0^2/(2\theta_0^2)}T_0$. Assume that $T_t\leq \frac{1+(1-\theta_0)\kappa}{1+\kappa\theta_0^2/(2\theta_{t-1}^2)}T_0$ holds for $t\leq k$. Now, we consider $t=k+1$. From (\ref{cont44}), we have
\begin{eqnarray}
\begin{aligned}\label{cont45}
T_{k+1}\leq& \frac{1+(1-\theta_0)\kappa}{1+\frac{\kappa\theta_0}{\theta_k}}\left(\frac{\theta_k}{\theta_0}+\theta_k\sum_{t=1}^k\left( \frac{\theta_t(1-\theta_0)^{k-t}}{\theta_{t-1}^2}-\frac{(1-\theta_0)^{k-(t-1)}}{\theta_{t-1}} \right)\frac{1}{1+\frac{\kappa\theta_0^2}{2\theta_{t-1}^2}}\right)T_0.
\end{aligned}
\end{eqnarray}
We can easily check
\begin{eqnarray}
\begin{aligned}\notag
&\sum_{t=1}^k\left( \frac{\theta_t(1-\theta_0)^{k-t}}{\theta_{t-1}^2}-\frac{(1-\theta_0)^{k-(t-1)}}{\theta_{t-1}} \right)\frac{1}{1+\frac{\kappa\theta_0^2}{2\theta_{t-1}^2}}\\
=&\sum_{t=1}^k\left( \frac{\theta_t}{\theta_{t-1}^2+\frac{\kappa\theta_0^2}{2}}(1-\theta_0)^{k-t} - \frac{\theta_{t-1}}{\theta_{t-1}^2+\frac{\kappa\theta_0^2}{2}}(1-\theta_0)^{k-(t-1)} \right)\\
=&\sum_{t=1}^k\hspace*{-0.07cm}\left(\hspace*{-0.07cm} \frac{\theta_t}{\theta_t^2\hspace*{-0.07cm}+\hspace*{-0.07cm}\frac{\kappa\theta_0^2}{2}}(1\hspace*{-0.07cm}-\hspace*{-0.07cm}\theta_0)^{k-t} \hspace*{-0.07cm}-\hspace*{-0.07cm} \frac{\theta_{t-1}}{\theta_{t-1}^2\hspace*{-0.07cm}+\hspace*{-0.07cm}\frac{\kappa\theta_0^2}{2}}(1\hspace*{-0.07cm}-\hspace*{-0.07cm}\theta_0)^{k-(t-1)} \hspace*{-0.07cm}\right)\hspace*{-0.07cm}+\hspace*{-0.07cm}\sum_{t=1}^k\hspace*{-0.07cm}\left(\hspace*{-0.07cm}\frac{\theta_t}{\theta_{t-1}^2\hspace*{-0.07cm}+\hspace*{-0.07cm}\frac{\kappa\theta_0^2}{2}}\hspace*{-0.07cm}-\hspace*{-0.07cm}\frac{\theta_t}{\theta_t^2\hspace*{-0.07cm}+\hspace*{-0.07cm}\frac{\kappa\theta_0^2}{2}}\hspace*{-0.07cm}\right)\hspace*{-0.07cm}(1\hspace*{-0.07cm}-\hspace*{-0.07cm}\theta_0)^{k-t}\\
=& \frac{\theta_k}{\theta_k^2+\frac{\kappa\theta_0^2}{2}}-\frac{(1-\theta_0)^k}{\theta_0+\frac{\kappa\theta_0}{2}}+\sum_{t=1}^k\left(\frac{\theta_t}{\theta_{t-1}^2+\frac{\kappa\theta_0^2}{2}}-\frac{\theta_t}{\theta_t^2+\frac{\kappa\theta_0^2}{2}}\right)(1-\theta_0)^{k-t}
\end{aligned}
\end{eqnarray}
and
\begin{eqnarray}
\begin{aligned}\notag
&\frac{\theta_t}{\theta_{t-1}^2\hspace*{-0.04cm}+\hspace*{-0.04cm}\frac{\kappa\theta_0^2}{2}}\hspace*{-0.04cm}-\hspace*{-0.04cm}\frac{\theta_t}{\theta_t^2\hspace*{-0.04cm}+\hspace*{-0.04cm}\frac{\kappa\theta_0^2}{2}}\hspace*{-0.04cm}\overset{a}=\hspace*{-0.04cm}-\frac{\theta_t^2\theta_{t-1}^2}{\left(\hspace*{-0.04cm}\theta_{t-1}^2\hspace*{-0.04cm}+\hspace*{-0.04cm}\frac{\kappa\theta_0^2}{2}\hspace*{-0.04cm}\right)\hspace*{-0.04cm}\left(\hspace*{-0.04cm}\theta_t^2\hspace*{-0.04cm}+\hspace*{-0.04cm}\frac{\kappa\theta_0^2}{2}\hspace*{-0.04cm}\right)}\hspace*{-0.04cm}=\hspace*{-0.04cm}-\frac{1}{\left(\hspace*{-0.04cm}1\hspace*{-0.04cm}+\hspace*{-0.04cm}\frac{\kappa\theta_0^2}{2\theta_{t-1}^2}\hspace*{-0.04cm}\right)\hspace*{-0.04cm}\left(\hspace*{-0.04cm}1\hspace*{-0.04cm}+\hspace*{-0.04cm}\frac{\kappa\theta_0^2}{2\theta_t^2}\hspace*{-0.04cm}\right)}\hspace*{-0.04cm}\leq\hspace*{-0.04cm}-\frac{1}{\left(\hspace*{-0.04cm}1\hspace*{-0.04cm}+\hspace*{-0.04cm}\frac{\kappa\theta_0^2}{2\theta_k^2}\hspace*{-0.04cm}\right)^2},
\end{aligned}
\end{eqnarray}
where we use $\theta_{t-1}^2-\theta_t^2=\theta_{t-1}^2\theta_t^2\left( \frac{1}{\theta_t^2}-\frac{1}{\theta_{t-1}^2} \right)=\theta_{t-1}^2\theta_t$ in $\overset{a}=$. Plugging them into (\ref{cont45}), we only need to prove
\begin{eqnarray}
\begin{aligned}\label{cont46}
\left( 1+\frac{\kappa\theta_0^2}{2\theta_k^2} \right) \left(\frac{\theta_k}{\theta_0}+ \frac{\theta_k^2}{\theta_k^2+\frac{\kappa\theta_0^2}{2}}-\frac{\theta_k(1-\theta_0)^k}{\theta_0+\frac{\kappa\theta_0}{2}}-\frac{\theta_k\sum_{t=1}^k(1-\theta_0)^{k-t}}{\left(1+\frac{\kappa\theta_0^2}{2\theta_k^2}\right)^2}\right)\leq 1+\frac{\kappa\theta_0}{\theta_k}.
\end{aligned}
\end{eqnarray}
After some simple calculations, (\ref{cont46}) is equivalent to
\begin{eqnarray}
\begin{aligned}\notag
1\leq&\frac{\theta_0\sum_{t=1}^k(1-\theta_0)^{k-t}}{\left(1+\frac{\kappa\theta_0^2}{2\theta_k^2}\right)^2}+\frac{(1-\theta_0)^k}{1+\frac{\kappa}{2}}+\frac{\frac{\kappa\theta_0^2}{\theta_k^2}}{1+\frac{\kappa\theta_0^2}{2\theta_k^2}}.
\end{aligned}
\end{eqnarray}
Since $\theta_0\sum_{t=1}^k(1-\theta_0)^{k-t}=1-(1-\theta_0)^k$, we only need to ensure $\frac{(1-\theta_0)^k}{1+\frac{\kappa}{2}}-\frac{(1-\theta_0)^k}{\left(1+\kappa\theta_0^2/(2\theta_k^2)\right)^2}\geq 0$ and $\frac{1}{\left(1+\kappa\theta_0^2/(2\theta_k^2)\right)^2}+\frac{\kappa\theta_0^2/\theta_k^2}{1+\kappa\theta_0^2/(2\theta_k^2)}\geq 1$. Both inequalities hold for any $\theta_k$ and any $\kappa$.
\end{proof}

Now we consider the outer iterations of Algorithm \ref{alg_arpcg_linear}. Replace $\uu^{k+1}$, $\uu^0$, $\uu^{k+1,*}$ and $\uu^{0,*}$ in Lemma \ref{linar_lemma} by $\uu^{t,K+1}$, $\uu^{t,0}$, $\uu^{t,K+1,*}$ and $\uu^{t,0,*}$, respectively. From $\frac{1}{\theta_K^2}\geq \left(\frac{K}{2}+\hat n\right)^2$ and $\uu^{t,K+1}=\uu^{t+1,0}$, we can immediately have (\ref{cont47}).

\subsection{Convergence Rate Analysis of the Primal Solutions}\label{sec_linear_proof}

In this section, we prove (\ref{cont54}). We first establish the relation between the primal objective and dual objective in the following lemma and then (\ref{cont54}) can be attained by Lemmas \ref{linar_lemma} and \ref{linear_lemma_p} immediately.
\begin{lemma}\label{linear_lemma_p}
Suppose Assumptions \ref{assumption} and \ref{assumption2} hold. Assume that $\DS^*$ is bounded, i.e., $\|\uu^*\|_L\leq C_{\DS^*},\forall \uu^*\in \DS^*$. Let $K_0\leq\left\lfloor\frac{K}{\upsilon(1+1/\hat n)}+1\right\rfloor$ with any $\upsilon>1$ and $K\geq \hat n$. Then for Algorithm \ref{alg_arpcg_linear} we have
\begin{eqnarray}
\begin{aligned}
&\left|\E_{\zeta_t}\left[F(\hat\x^{t,K})\right]-F(\x^*)\right|\leq \left( \frac{2T_{t,0} + 6C_{\DS^*}\sqrt{T_{t,0}} + 6M\sqrt{\sum_{i=1}^nL_i}\sqrt{T_{t,0}}}{1-1/\nu} + \frac{2\sqrt{m}C_{\DS^*}\sqrt{T_{t,0}}}{\sqrt{1-1/\nu}} \right),\notag\\
&\E_{\zeta_t}\hspace*{-0.06cm}\left[\left\|\hspace*{-0.06cm} \left[\hspace*{-0.1cm}
  \begin{array}{c}
    \B\hat\x^{t,K}+\b\\
    \max\left\{0,g(\hat\x^{t,K})\right\}
  \end{array}\hspace*{-0.1cm}
\right]\hspace*{-0.06cm}\right\|_L^*\right] \hspace*{-0.06cm}\leq\frac{6\sqrt{T_{t,0}}}{1-1/\nu},\notag
\end{aligned}
\end{eqnarray}
where $T_{t,0}=(1-\theta_0)\left(\E_{\zeta_{t-1}}[D(\uu^{t,0})]-D(\uu^*)\right)+\E_{\zeta_{t-1}}[\|\uu^{t,0}-\uu^{t,0,*}\|_L^2]$.
\end{lemma}
\begin{proof}
We denote $\uu^{t,k},\z^{t,k},\vv^{t,k},i_{t,k},\xi_{t,k},\hat\x^{t,K},\hat\y^{t,K}$ to be the variables at the $t$-th outer iteration of Algorithm \ref{alg_arpcg_linear}, which are the counterparts of $\uu^{k},\z^k,\vv^k,i_k,\xi_k,\hat\x^K$ and $\hat\y^K$ in Algorithm \ref{alg_arpcg}.
Choose $\uu^*=\uu^{t,0,*}$ and let $\uu=\uu^{t,0,*}$ in (\ref{cont14}), which is independent on $\xi_{t,K}$ conditioned on $\zeta_{t-1}$. For the $t$-th outer iteration of Algorithm \ref{alg_arpcg_linear}, we have
\begin{eqnarray}
\begin{aligned}
&\left(\frac{1}{\theta_{K}^2}-\frac{1}{\theta_{K_0-1}^2}\right)\E_{\xi_{t,K}|\zeta_{t-1}}\left[\<\triangle(\hat\x^{t,K},\hat\y^{t,K}),\uu^{t,0,*}\>+ D(\uu^*)+f(\hat\x^{t,K})+\frac{1}{n}\phi(\hat\y^{t,K})\right]\notag\\
\leq& 2\hat n^2\left((1-\theta_0)\left(D(\uu^{t,0})-D(\uu^*)\right)+\|\uu^{t,0}-\uu^{t,0,*}\|_L^2\right).\notag
\end{aligned}
\end{eqnarray}
Taking expectation with respect to $\zeta_{t-1}$ and using the forth property of Lemma \ref{theta_lemma}, we have
\begin{eqnarray}
\begin{aligned}\label{cont36}
\E_{\zeta_t}\hspace*{-0.04cm}\left[\<\triangle(\hat\x^{t,K},\hat\y^{t,K}),\uu^{t,0,*}\>\hspace*{-0.04cm}+\hspace*{-0.04cm} D(\uu^*)\hspace*{-0.04cm}+\hspace*{-0.04cm}f(\hat\x^{t,K})\hspace*{-0.04cm}+\hspace*{-0.04cm}\frac{1}{n}\phi(\hat\y^{t,K})\right]\hspace*{-0.04cm}\leq\hspace*{-0.04cm} \frac{2\hat n^2}{(K^2/4\hspace*{-0.04cm}+\hspace*{-0.04cm}\hat n K)(1\hspace*{-0.04cm}-\hspace*{-0.04cm}1/\nu)}T_{t,0}.
\end{aligned}
\end{eqnarray}
Choosing $\uu^*=\uu^{t,0,*}$ in (\ref{cont35}) and using a similar induction, we have
\begin{eqnarray}\label{cont37}
\E_{\zeta_t}\hspace*{-0.04cm}[\|\widetilde\triangle(\hat\x^{t,K}\hspace*{-0.04cm},\hspace*{-0.04cm}\hat\y^{t,K})\|_L^*]\hspace*{-0.1cm}\leq\hspace*{-0.1cm} \E_{\zeta_{t-1}}\hspace*{-0.08cm} \sqrt{\hspace*{-0.04cm}\E_{\xi_{t,K}\hspace*{-0.02cm}|\hspace*{-0.02cm}\zeta_{t\hspace*{-0.02cm}-\hspace*{-0.02cm}1}}\hspace*{-0.1cm}\left[\hspace*{-0.08cm}\left(\hspace*{-0.08cm}\|\widetilde\triangle(\hat\x^K\hspace*{-0.04cm},\hspace*{-0.04cm}\hat\y^K)\|_L^*\hspace*{-0.08cm}\right)^2\right]}\hspace*{-0.1cm}\leq \hspace*{-0.1cm} \frac{7\hat n^2}{(K^2\hspace*{-0.04cm}/\hspace*{-0.04cm}4\hspace*{-0.08cm}+\hspace*{-0.08cm}\hat n K)(1\hspace*{-0.08cm}-\hspace*{-0.08cm}1\hspace*{-0.04cm}/\hspace*{-0.04cm}\nu)}\hspace*{-0.08cm}\sqrt{ T_{t,0} }.
\end{eqnarray}
Now, we consider the objective function. From (\ref{cont36}) and the definition in (\ref{lagrangian}), we have
\begin{eqnarray}
\E_{\zeta_t}\left[ L_F(\hat\x^{t,K},\hat\y^{t,K},\uu^{t,0,*})-L_F(\x^*,\y^*,\uu^{t,0,*}) \right]\leq \frac{2\hat n^2}{(K^2/4+\hat n K)(1-1/\nu)}T_{t,0}.\notag
\end{eqnarray}
Since $(\x^*,\y^*,\uu^{t,0,*})$ satisfies the KKT condition, we have $\0\in\partial_{\x,\y}L_F(\x^*,\y^*,\uu^{t,0,*})$. $L_F(\x,\y,\uu^{t,0,*})$ is convex with respect to $(\x,\y)$ and $\mu$-strongly convex with respect to $\x$, so we have
\begin{eqnarray}
&&L_F(\hat\x^{t,K},\hat\y^{t,K},\uu^{t,0,*})-L_F(\x^*,\y^*,\uu^{t,0,*})\geq \frac{\mu}{2}\|\hat\x^{t,K}-\x^*\|^2,\notag
\end{eqnarray}
which leads to $\E_{\zeta_t}\left[\|\hat\x^{t,K}-\x^*\|^2\right]\leq \frac{4\hat n^2}{\mu(K^2/4+\hat n K)(1-1/\nu)}T_{t,0}$ and
\begin{eqnarray}
\E_{\zeta_t}[\|\hat\x^{t,K}-\x^*\|]\leq \frac{2\hat n}{\sqrt{\mu(K^2/4+\hat n K)(1-1/\nu)}}\sqrt{ T_{t,0} }.\label{cont39}
\end{eqnarray}
Let $\IS$ be the index set such that for any $i\in\IS$, we have $\uu^{t,0,*}_{n+p+i}>0$ and $g_i(\hat\x^{t,K})<0$. So we have
\begin{eqnarray}
\begin{aligned}
&\E_{\zeta_t}\left[\<\triangle(\hat\x^{t,K},\hat\y^{t,K}),\uu^{t,0,*}\>\right]\notag\\
=&\E_{\zeta_t}\left[\<\widetilde\triangle(\hat\x^{t,K},\hat\y^{t,K}),\uu^{t,0,*}\>\right]+\sum_{i\in\IS}\E_{\zeta_t}\left[\uu^{t,0,*}_{n+p+i}g_i(\hat\x^{t,K})\right]\notag\\
\overset{a}=&\E_{\zeta_t}\left[\<\widetilde\triangle(\hat\x^{t,K},\hat\y^{t,K}),\uu^{t,0,*}\>\right]+\sum_{i\in\IS}\E_{\zeta_t}\left[\uu^{t,0,*}_{n+p+i}(g_i(\hat\x^{t,K})-g_i(\x^*))\right]\notag\\
\geq& -\E_{\zeta_t}\left[\|\uu^{t,0,*}\|_L\|\widetilde\triangle(\hat\x^{t,K},\hat\y^{t,K})\|_L^*\right]-\E_{\zeta_t}\left[\|\uu^{t,0,*}\|_L\sqrt{\sum_{i\in\IS}\frac{1}{L_{n+p+i}}\left(g_i(\hat\x^{t,K})-g_i(\x^*)\right)^2}\right]\notag\\
\overset{b}\geq& -\E_{\zeta_t}\left[\|\uu^{t,0,*}\|_L\|\widetilde\triangle(\hat\x^{t,K},\hat\y^{t,K})\|_L^*\right]-\sqrt{m\mu}\E_{\zeta_t}\left[\|\uu^{t,0,*}\|_L\|\hat\x^{t,K}-\x^*\| \right]\notag\\
\overset{c}\geq& -\left( \frac{7\hat n^2}{(K^2/4+\hat n K)(1-1/\nu)} + \frac{2\hat n\sqrt{m}}{\sqrt{(K^2/4+\hat n K)(1-1/\nu)}} \right)C_{\DS^*}\sqrt{T_{t,0}},\notag
\end{aligned}
\end{eqnarray}
where in $\overset{a}=$ we use $\uu^{t,0,*}_{n+p+i}g_i(\x^*)=0$ from the complementary slackness in the KKT condition. From Assumption \ref{assumption}.3 and (\ref{CL_constant}), we have $\sum_{i\in\IS}\frac{1}{L_{n+p+i}}\left(g_i(\hat\x^{t,K})-g_i(\x^*)\right)^2\leq \sum_{i\in\IS}\frac{L_{g_i}^2}{L_{n+p+i}}\|\hat\x^{t,K}-\x^*\|^2\leq m\mu\|\hat\x^{t,K}-\x^*\|^2$, which leads to $\overset{b}\geq$. In $\overset{c}\geq$, we use $\|\uu^{t,0,*}\|_L\leq C_{\DS^*}$, (\ref{cont37}) and (\ref{cont39}). From (\ref{cont36}), we have
\begin{eqnarray}
\begin{aligned}\notag
&\E_{\zeta_t}\left[ D(\uu^*)+f(\hat\x^{t,K})+\frac{1}{n}\phi(\hat\y^{t,K})\right]\leq \frac{2\hat n^2T_{t,0} + 7\hat n^2C_{\DS^*}\sqrt{T_{t,0}}}{(K^2/4+\hat n K)(1-1/\nu)} + \frac{2\hat n\sqrt{m}C_{\DS^*}\sqrt{T_{t,0}}}{\sqrt{(K^2/4+\hat n K)(1-1/\nu)}}.
\end{aligned}
\end{eqnarray}
From $D(\uu^*)=-F(\x^*)$, (\ref{cont7}), (\ref{cont37}), (\ref{cont48}) and $K\geq\hat n$, we have the conclusion.
\end{proof}

\section{Application to the Regularized ERM}\label{sec:erm}
The regularized empirical risk minimization problem (\ref{problem2}) has broad applications in machine learning. For the special problem (\ref{problem2}), its dual problem (\ref{dual_problem}) becomes
\begin{eqnarray}
\min_{\uu\in\R^n} D(\uu)=d(\uu)+h(\uu)\equiv f^*\left(-\frac{\A\uu}{n}\right)+\frac{1}{n}\sum_{i=1}^n\phi_i^*(\uu_i).\label{dual_problem2}
\end{eqnarray}
We follow \cite{zhang-2013-jmlr,zhang-2015-MP} to assume $\|\A_i\|\leq 1,\forall i$, which can be guaranteed by normalizing the data. Then we have $L=\frac{\|\A\|_2^2}{n^2\mu}$ and $L_j=\frac{\|\A_j\|^2}{n^2\mu}\leq \frac{1}{n^2\mu}$ from (\ref{L_constant}) and (\ref{CL_constant}). From Lemmas 21 and 22 in \cite{zhang-2013-jmlr}, we have $|\z_i^k|\leq M,|\widetilde\z_i^k|\leq M,k=0,1,\cdots,K$, and $|\uu_i^*|\leq M$, which leads to $\|\uu^0-\uu^*\|_L^2\leq \frac{4M^2}{n\mu}$ and $L\|\uu^0-\uu^*\|^2\leq \frac{4M^2}{n\mu}\|\A\|_2^2$. We will discuss the iteration complexity of ARDCA in three scenarios.

\subsection{Strongly Convex and Nonsmooth $f$, Convex and Nonsmooth $\phi_i$}\label{sec:erm1}
From Theorem \ref{main_theorem}, we know that the convergence rate of ARDCA for problem (\ref{problem2}) is:
\begin{eqnarray}
\begin{aligned}\notag
\E_{\xi_K}[F(\hat\x^K)]-F(\x^*)\leq\frac{9nM^2\left(6+\frac{n\mu}{M^2}\left( D(\uu^0)- D(\uu^*)\right)\right)}{\mu(K^2/4+nK)(1-1/\upsilon)}.
\end{aligned}
\end{eqnarray}

In order to have the $O\left(\frac{nM^2}{\mu K^2}\right)$ convergence rate for ARDCA, we should find an initializer good enough such that $ D(\uu^0)- D(\uu^*)\leq O\left(\frac{M^2}{n\mu}\right)$. We use ARDCA with fixed $\theta_k=\frac{1}{n}$, i.e., non-accelerated RDCA, to find such initializer. Specifically, we describe the method in Algorithm \ref{alg_arpcg2}. Lemma \ref{erm_lemma} establishes the convergence rates of both the primal solutions and dual solutions for the first step of Algorithm \ref{alg_arpcg2}, whose proof is given in Appendix E.
\begin{algorithm}
   \caption{ARDCA for ERM}
   \label{alg_arpcg2}
\begin{algorithmic}
   \STATE Input $\uu^0\in \DS$, $K'$, $K_0$, $K$.
   \STATE Run ARDCA($\uu^0$,$0$,$K'$) with fixed $\theta_k=\frac{1}{n}$ and output $\uu^{K'+1}$ and $\x^*(\vv^{K'})$.
   \STATE Run ARDCA($\uu^{K'+1}$,$K_0$,$K$) with decreasing $\theta_k$ and output $\uu^{K+1}$ and $\hat\x^K$.
\end{algorithmic}
\end{algorithm}

\begin{lemma}\label{erm_lemma}
Let $K'= \left\lceil n\log\left(\min\left\{\frac{1}{\epsilon},\frac{n\mu}{M^2}\right\}\left( D(\uu^0)+F(\x^*(\uu^0))\right)\right)-1\right \rceil$. Suppose Assumptions \ref{assumption}.1, \ref{assumption}.2 and \ref{assumption}.5 hold. Then for step 1 of Algorithm \ref{alg_arpcg2}, we have.
\begin{eqnarray}
&&\E_{\xi_{K'}}[ D(\uu^{{K'+1}})]- D(\uu^*)\leq 9\max\left\{\epsilon,\frac{M^2}{n\mu}\right\},\label{cont57}\\
&&\E_{\xi_{K'}}[ F(\x^*(\vv^{K'})]-F(\x^*) \leq 17\max\left\{\epsilon,\frac{M^2}{n\mu}\right\}.\label{cont58}
\end{eqnarray}
\end{lemma}

An immediate consequence of Lemma \ref{erm_lemma} is that if $\epsilon\geq O\left(\frac{M^2}{n\mu}\right)$, we only need to run step 1 of Algorithm \ref{alg_arpcg2} with linear complexity to achieve an $\epsilon$-optimal solution. We describe the results in Corollary \ref{erm_corollary2}. However, in statistical learning, $\mu$ is usually on the order of $\frac{1}{\sqrt{n}}$ or $\frac{1}{n}$ \cite{Bousquet2002,zhang-2017-jmlr}, thus $\frac{M^2}{n\mu}$ is often not too small. In the following discussions, we only consider the case of $\epsilon< O\left(\frac{M^2}{n\mu}\right)$.
\begin{corollary}\label{erm_corollary2}
If $\epsilon\geq O\left(\frac{M^2}{n\mu}\right)$, we only need to run ARDCA($\uu^0$,0,$K'$) with fixed $\theta_k=\frac{1}{n}$, i.e., non-accelerated RDCA, for $K'=\left\lceil n\log\left(\frac{ D(\uu^0)+F(\x^*(\uu^0))}{\epsilon}\right)\right\rceil$ iterations to find an $\epsilon$-optimal solution such that
\begin{eqnarray}
&&\E_{\xi_{K'}}[F(\x^*(\vv^{K'}))]-F(\x^*)\leq \epsilon,\quad\E_{\xi_{K'}}[ D(\uu^{K'+1})]- D(\uu^*)\leq \epsilon.\notag
\end{eqnarray}
If $\epsilon< O\left(\frac{M^2}{n\mu}\right)$, Algorithm \ref{alg_arpcg2} needs $K'+K=O\left(n\log\left(\frac{n\mu}{M^2}\left( D(\uu^0)+F(\x^*(\uu^0))\right)\right)+M\sqrt{\frac{n}{\mu\epsilon}}\right)$ iterations to find an $\epsilon$-optimal solution such that
\begin{eqnarray}
&&\E_{\xi_K\cup\xi_{K'}}[F(\hat\x^K)]-F(\x^*)\leq \epsilon,\quad\E_{\xi_K\cup\xi_{K'}}[ D(\uu^{K+1})]- D(\uu^*)\leq \epsilon.\notag
\end{eqnarray}
\end{corollary}


Algorithm \ref{alg_arpcg} is a special case of APCG \cite{xiao-2015-siam}. \cite{xiao-2015-siam} only established the $O\left(n\sqrt{\frac{C}{\epsilon}}\right)$ iteration complexity in the dual space to achieve an $\epsilon$-optimal dual solution\footnote{When $\phi_i$ has Lipchitz continuous gradient, \cite{xiao-2015-siam} proved the linear convergence rate in the primal space. However, when $\phi_i$ is only Lipchitz continuous, the convergence rate in the primal space is not established in \cite{xiao-2015-siam}.},
where $C= D(\uu^0)- D(\uu^*)+\frac{1}{2}\|\z^0-\uu^*\|_L^2\leq D(\uu^0)- D(\uu^*)+\frac{2M^2}{n\mu}$. \cite{zhang-2015-MP} developed an accelerated SDCA with an inner-outer iteration procedure, where the outer loop is a full-dimensional accelerated proximal point method. At each iteration of the outer loop, SDCA is called to solve a subproblem inexactly. ASDCA is mainly used for the problems with smooth $\phi_i$. When $\phi_i$ is nonsmooth, \cite{zhang-2015-MP} used ASDCA to  solve a smoothed problem of (\ref{problem2}), i.e., a regularized problem of (\ref{dual_problem2}), and achieved a slightly worse iteration complexity of $O\left(\left(n+M\sqrt{\frac{n}{\mu\epsilon}}\right)\log\frac{1}{\epsilon}\right)$ to find an $\epsilon$-optimal primal solution. 
We can also use Catalyst \cite{lin-2015-nips} to solve the problems with nonsmooth $\phi_i$ without using smoothing. However, Catalyst also yields the additional $\left(\log\frac{1}{\epsilon}\right)$ factor. To make ASDCA faster than SDCA, which has the $O\left(n\log\frac{n\mu}{M^2}+\frac{M^2}{\mu\epsilon}\right)$ complexity \cite{zhang-2013-jmlr}, \cite{zhang-2015-MP} required $\epsilon\leq\frac{M^2}{n\mu}$. Katyusha \cite{zhu-2017-stoc}, a primal-only algorithm, obtains the state-of-the-art iteration complexity of $O\left(n\log\frac{F(\x^0)-F(\x^*)}{\epsilon}+M\sqrt{\frac{n}{\mu\epsilon}}\right)$, which is worse than our result when $\epsilon\leq\frac{M^2}{n\mu}$. Our result matches the theoretical lower bound of $\left(n+M\sqrt{\frac{n}{\mu\epsilon}}\right)$ \cite{Woodworth-2016} when ignoring the constant term of $n\log n$. All the compared methods need $O(t)$ runtime at each iteration.

\subsection{Strongly Convex and Nonsmooth $f$, Convex and Smooth $\phi_i$}\label{sec_erm2}

When each $\phi_i$ is $1/\gamma$-smooth, which is defined as $\phi_i(u)\leq \phi_i(v)+\<\nabla \phi_i(v),u-v\>+\frac{1}{2\gamma}\|u-v\|^2$, then $\phi_i^*$ is $\gamma$-strongly convex and $ D(\uu)$ is $\frac{\gamma}{n}$-strongly convex. In this case, Assumption \ref{assumption2} is satisfied with $\kappa=n\gamma\mu$. From the discussion at the end of Section \ref{linar_sec}, we know that Algorithm \ref{alg_arpcg_linear} needs $O\left(\left(n+\sqrt{\frac{n}{\gamma\mu}}\right)\log\frac{1}{\epsilon}\right)$ iterations in total to achieve an $\epsilon$-optimal primal solution and dual solution in the best case scenario, which matches or outperforms the complexities of the dual based stochastic algorithms established in \cite{xiao-2015-siam,zhang-2015-MP}. Specifically, their complexities are $O\left(\left(n+\sqrt{\frac{n}{\gamma\mu}}\right)\log\frac{1}{\epsilon}\right)$ and $O\left(\left(n+\sqrt{\frac{n}{\gamma\mu}}\right)\log\frac{1}{\epsilon}\log^2\frac{1}{n\gamma\mu}\right)$, respectively. To make the accelerated algorithms faster than the non-accelerated counterparts, \cite{zhang-2015-MP} required $\frac{1}{n\gamma\mu}\gg 1$ (i.e., $\kappa\ll 1$). Thus, our complexity has a better dependence on $\left(\log\frac{1}{n\gamma\mu}\right)$ than \cite{zhang-2015-MP}. Note that APCG \cite{xiao-2015-siam} needs an extra proximal full gradient step to establish the linear convergence rate in the primal space. Our analysis does not need such an additional operation. On the other hand, \cite{xiao-2015-siam} and \cite{zhang-2015-MP} did not study the case when $\kappa$ is unknown.

\subsection{Strongly Convex and Smooth $f$, Convex and Nonsmooth $\phi_i$}\label{sec_erm3}

As discussed in Section \ref{linar_sec}, Assumption \ref{assumption2} is weaker than the strong convexity of $ D(\uu)$ and some special cases of problem (\ref{problem}) with nonsmooth $\phi_i$ also satisfy Assumption \ref{assumption2}. We take SVM and the least absolute deviation as examples. The primal problem and dual problem of SVM are
\begin{eqnarray}
\begin{aligned}\label{svm_problem}
&\min_{\x\in\R^t} F(\x)=\frac{\mu}{2}\|\x\|^2+\frac{1}{n}\sum_{i=1}^n\max\{0,1-l_i\A_i^T\x\},\\
&\min_{\uu\in\R^n}  D(\uu)=\frac{1}{2\mu}\left\|\frac{\widetilde\A\uu}{n}\right\|^2-\frac{1}{n}\sum_{i=1}^n\uu_i+I_{[0,1]}(\uu),
\end{aligned}
\end{eqnarray}
where $\widetilde\A_i=l_i\A_i$ and $l_i$ is the label for the $i$-th data $\A_i$, $I_{[0,1]}(\uu)=\left\{
  \begin{array}{ll}
    0 & \mbox{if }0\leq\uu\leq 1,\\
    \infty &\mbox{otherwise.}
  \end{array}
\right.$ \cite{lin2014} proved that $ D(\uu)$ in (\ref{svm_problem}) satisfies the global error bound condition. From \cite{lewis2018}, we know that Assumption \ref{assumption2} holds. From the discussion at the end of Section \ref{linar_sec}, we can see that Algorithm \ref{alg_arpcg_linear} needs $O\left(\frac{n}{\sqrt{\kappa}}\log\frac{1}{\epsilon}\right)$ iterations in total to achieve an $\epsilon$-optimal primal solution and dual solution in the best case scenario. As a comparison, \cite{ma2016eb} studied the randomized coordinate descent and established the $O\left(\frac{n}{\kappa}\log\frac{1}{\epsilon}\right)$ iteration complexity. The better dependence on $\kappa$ in our iteration complexity is significant when $\kappa$ is small and this is often the case in practice. We refer the reader to Section 5 of \cite{ma2016eb} for the discussion on the size of $\kappa$. When $\kappa$ is unknown, Algorithm \ref{alg_arpcg_linear} is still a better choice for a wide range of inner iteration number than the randomized coordinate descent.

For the least absolute deviation, its primal problem and dual problem are
\begin{eqnarray}
\begin{aligned}\label{lad_problem}
&\min_{\x\in\R^t} F(\x)=\frac{\mu}{2}\|\x\|^2+\|\A^T\x-\b\|_1,\\
&\min_{\uu\in\R^n}  D(\uu)=\frac{1}{2\mu}\left\|\A\uu\right\|^2+\<\uu,\b\>+I_{[-1,1]}(\uu).
\end{aligned}
\end{eqnarray}

Similar to SVM, $ D(\uu)$ in (\ref{lad_problem}) also satisfies Assumption \ref{assumption2} and Algorithm \ref{alg_arpcg_linear} needs $O\left(\frac{n}{\sqrt{\kappa}}\log\frac{1}{\epsilon}\right)$ iterations to achieve an $\epsilon$-optimal primal solution and dual solution in the best case scenario.

\section{Numerical Experiments}

In this section, we test the performance of Algorithms \ref{alg_arpcg}, \ref{alg_arpcg_linear} and \ref{alg_arpcg2} on the sparse recovery problem. Consider the sparse linear regression problem of $\b=\A^T\x+\w$, where $\x\in\R^{t}$ is the unknown sparse vector to estimate, $\b\in\R^n$ is the observation and $\w$ is some additive noise. A particular instance of this problem is compressed sensing \cite{cs-2006}. In order to recovery $\x$, a popular regularization is the $l_1$-norm, in which case people often solve the following problems:
\begin{eqnarray}
\min_{\x\in\R^t} f(\x),\quad s.t.\quad \|\A^T\x-\b\|_{\alpha}\leq \tau\qquad\mbox{or}\qquad \min_{\x\in\R^t} \lambda f(\x)+\|\A^T\x-\b\|_{\alpha},\notag
\end{eqnarray}
where $f(\x)=\|\x\|_1+\frac{\mu}{2}\|\x\|^2$. We add the term $\frac{\mu}{2}\|\x\|^2$ to make the objective function strongly convex and thus we can use some fast convergent algorithms. When the noise is generated from the Gaussian distribution, people often use the $l_2$ loss function, i.e., $\alpha=2$. When the noise is sparse and the data contains some outliers, the $l_1$ loss is often used, i.e., $\alpha=1$. When the noise is generated from a uniform distribution, we often use the $l_{\infty}$ loss instead. In this section, we solve the following three problems
\begin{eqnarray}
&&\min_{\x\in\R^t} F(\x)\equiv\lambda f(\x)+\frac{1}{2n}\|\A^T\x-\b\|_2^2,\label{exp_problem1}\\
&&\min_{\x\in\R^t} F(\x)\equiv\lambda f(\x)+\frac{1}{n}\|\A^T\x-\b\|_1,\label{exp_problem2}\\
&&\min_{\x\in\R^t} f(\x),\quad s.t.\quad -\tau\1\leq\A^T\x-\b\leq\tau\1.\label{exp_problem3}
\end{eqnarray}
Problems (\ref{exp_problem1}) and (\ref{exp_problem2}) are special cases of problem (\ref{problem2}) satisfying the assumptions in Sections \ref{sec_erm2} and \ref{sec:erm1} and problem (\ref{exp_problem3}) is a special case of problem (\ref{problem1}), respectively. In our numerical experiment, we set $t=1000$, $n=200$ and $\mu=0.1$. We generate the entries of $\A$ from the uniform distribution in $[0,1]$ and normalize each column of $\A$ such that $\|\A_{i}\|=1$. We set $t/10$ entries of $\x$ to be nonzeros. $\b$ is generated by $\A^T\x+\w$, where we generate each entry of noise $\w$ from the Gaussian distribution $N(0,\tau)$ for problem (\ref{exp_problem1}), generate $n/10$ entries of $\w$ from $N(0,\tau)$ and set the others to be 0 for problem (\ref{exp_problem2}), and generate each entry of $\w$ from the uniform distribution in $[-\tau,\tau]$ for problem (\ref{exp_problem3}). We vary $\lambda$ in the range $\{10^{-3},10^{-4},10^{-5}\}$ in problems (\ref{exp_problem1}) and (\ref{exp_problem2}) and $\tau$ in the range $\{10^{-3},10^{-4},10^{-5}\}$ in problem (\ref{exp_problem3}).

For problem (\ref{exp_problem1}), we compare ARDCA-restart (Algorithm \ref{alg_arpcg_linear}) with ASDCA \cite{zhang-2015-MP}, APCG \cite{xiao-2015-siam}, SDCA \cite{zhang-2013-jmlr} and ADFGA \cite{Beck-2014}. Figure \ref{fig1} plots the primal gap as functions of the number of passes over the data, where each $n$ (inner) iterations are equivalent to a single pass over the data for APCG and SDCA (ARDCA and ASDCA). We use the maximal dual objective value produced by the compared methods to approximate the optimal primal objective value $F(\x^*)$. We can see that ARDCA-restart outperforms the non-accelerated SDCA and non-randomized ADFGA for a wide range of $\lambda$ and ARDCA-restart is superior to APCG and ASDCA for some values of $\lambda$.

For problem (\ref{exp_problem2}), we compare ARDCA (Algorithm \ref{alg_arpcg2}) with ASDCA \cite{zhang-2015-MP}, SDCA \cite{zhang-2013-jmlr} and ADFGA \cite{Beck-2014}, where ASDCA solves a regularized dual problem of (\ref{exp_problem2}) by adding term $\frac{\epsilon}{2}\|\uu\|^2$ to the dual objective with $\epsilon=10^{-6}$. We set $\upsilon=1.1$ in Algorithm \ref{alg_arpcg}. Figure \ref{fig2} plots the results, where ARDCA-a means that we test the averaged primal solution and ARDCA-na means the non-averaged primal solution. We can see that ARDCA-a yields the best result by orders of magnitude. Specially, ASDCA with regularization does not perform well although it converges linearly when $\phi_i$ is smooth. Thus, although the regularization/smoothing based ASDCA has the near optimal theoretical result (the sub-optimality comes from  the $\left(\log\frac{1}{\epsilon}\right)$ factor), its practical performance is not satisfactory.

For problem (\ref{exp_problem3}), we compare ARDCA with SDCA and ADFGA. As demonstrated in Figure \ref{fig3}, we can see that ARDCA-a performs the best in both reducing the primal gap and constraint function value.

\begin{figure}
\centering
\begin{tabular}{@{\extracolsep{0.001em}}c@{\extracolsep{0.001em}}c@{\extracolsep{0.001em}}c@{\extracolsep{0.001em}}c}
\includegraphics[width=0.35\textwidth,keepaspectratio]{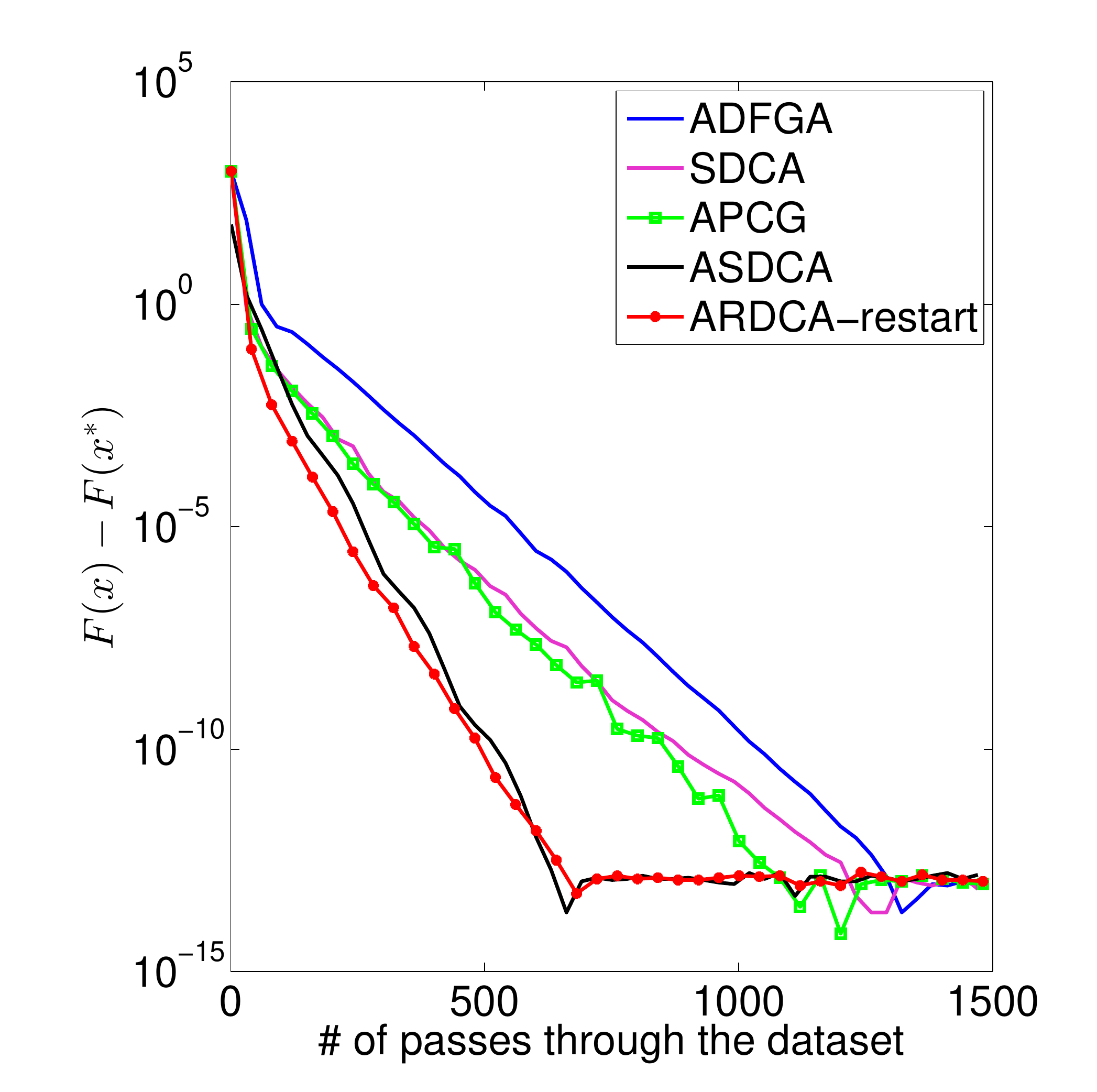}
&\hspace*{-0.2cm}\includegraphics[width=0.35\textwidth,keepaspectratio]{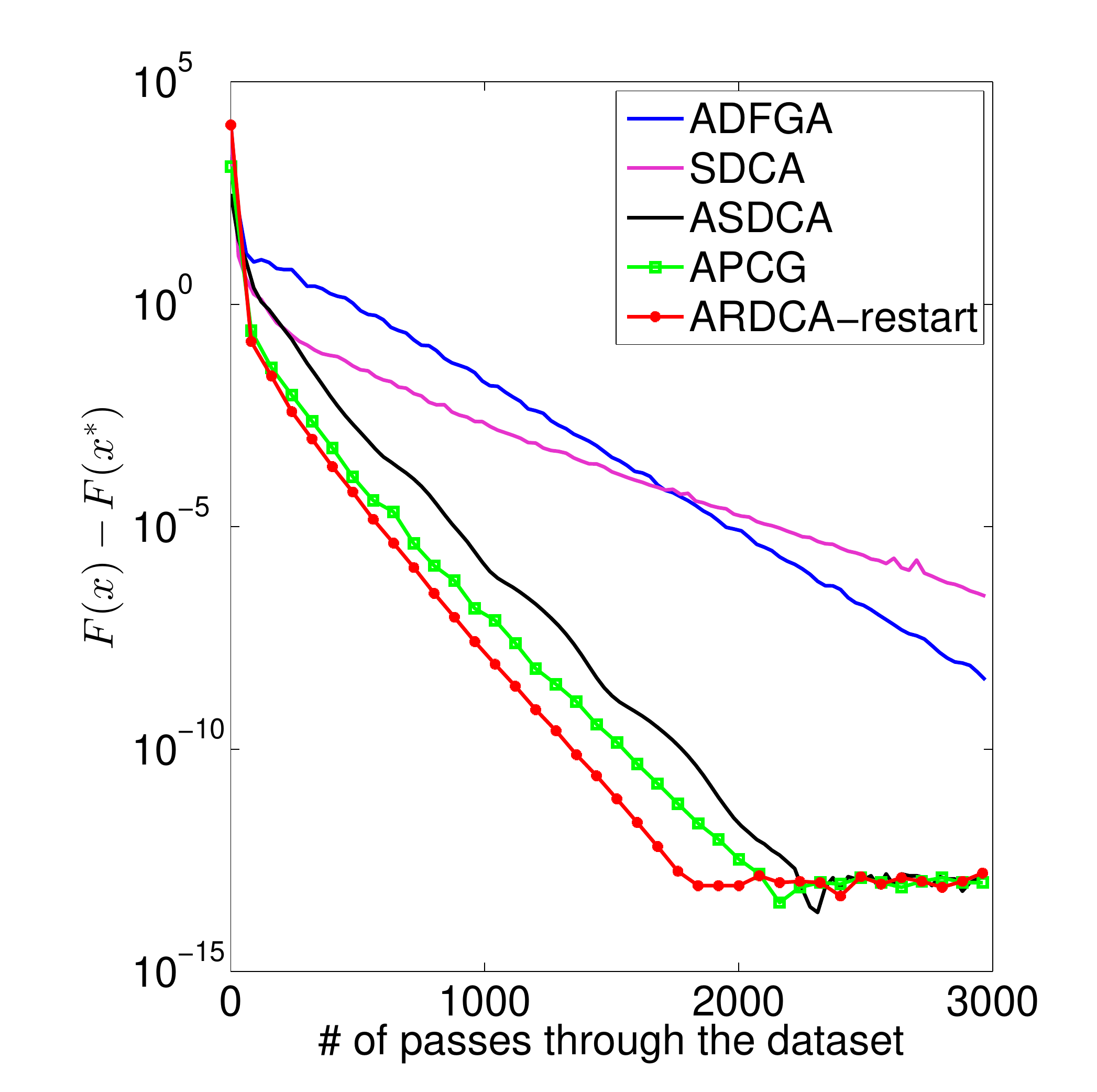}
&\hspace*{-0.2cm}\includegraphics[width=0.35\textwidth,keepaspectratio]{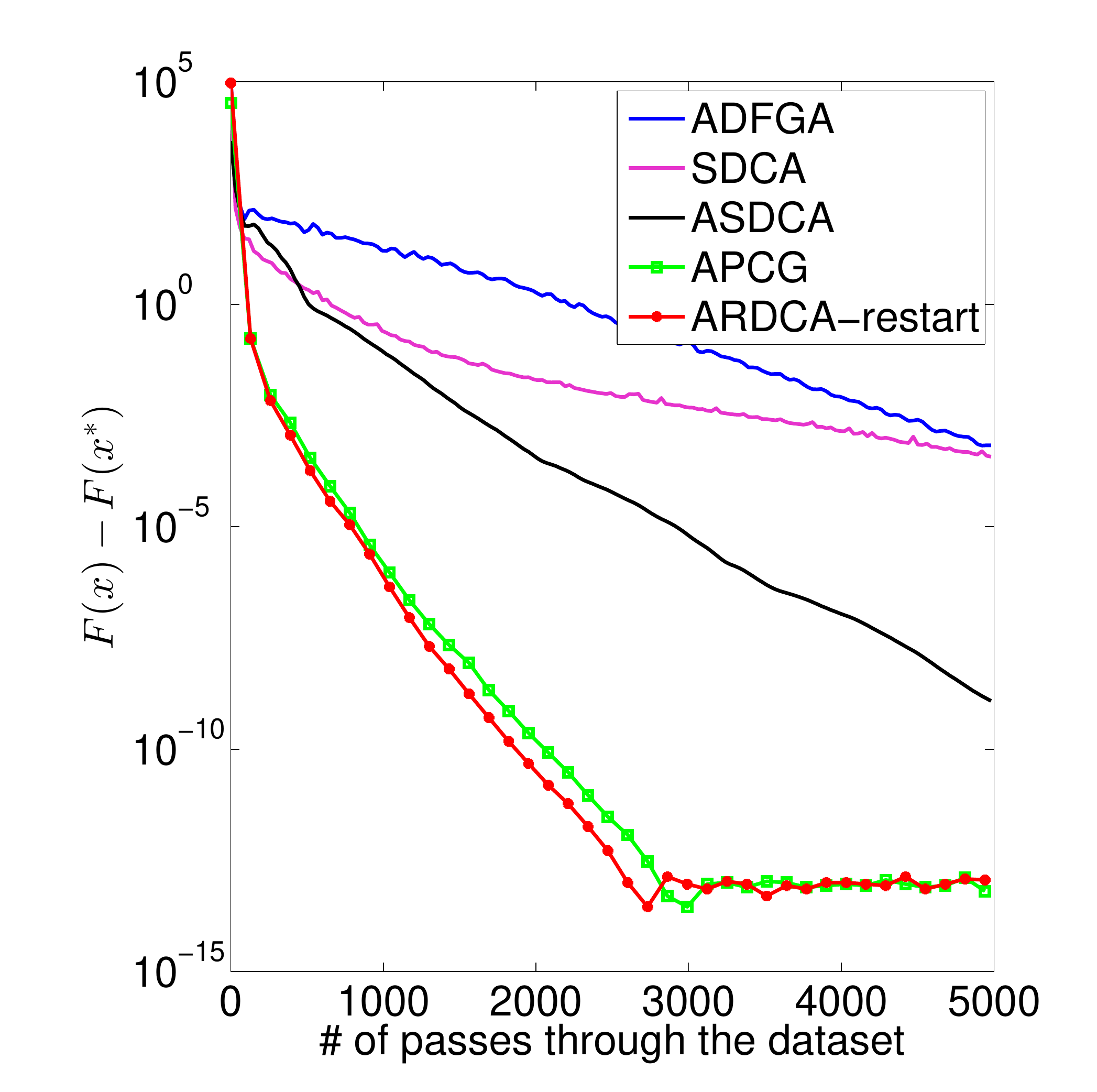}\\
$\lambda=10^{-3}$&$\lambda=10^{-4}$&$\lambda=10^{-5}$\\
\end{tabular}
\vspace*{-0.3cm}
\caption{Comparing ARDCA-restart with SDCA, APCG, ASDCA and ADFGA on problem (\ref{exp_problem1}).}\label{fig1}
\end{figure}

\begin{figure}
\centering
\begin{tabular}{@{\extracolsep{0.001em}}c@{\extracolsep{0.001em}}c@{\extracolsep{0.001em}}c@{\extracolsep{0.001em}}c}
\includegraphics[width=0.35\textwidth,keepaspectratio]{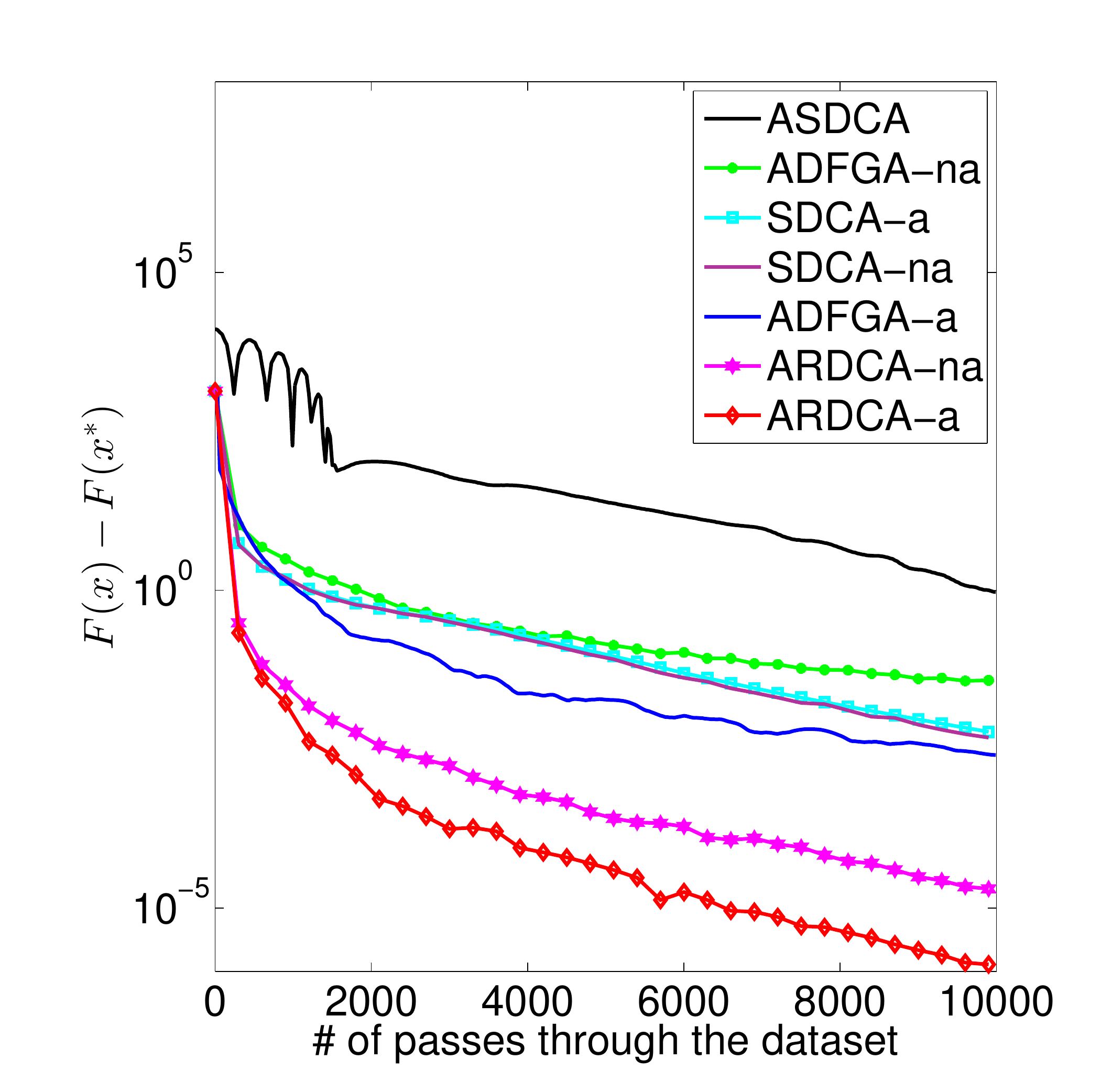}
&\hspace*{-0.2cm}\includegraphics[width=0.35\textwidth,keepaspectratio]{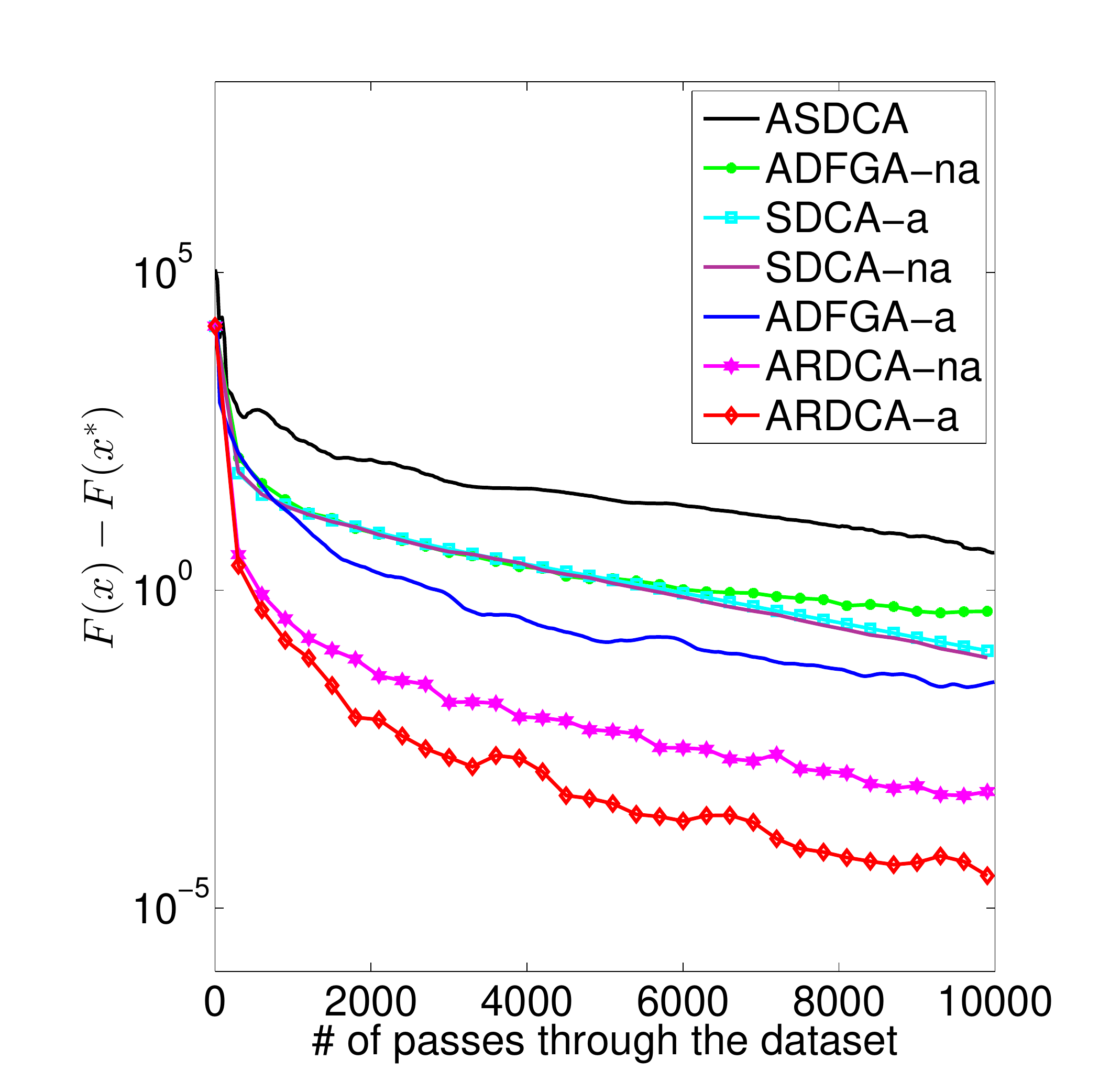}
&\hspace*{-0.2cm}\includegraphics[width=0.35\textwidth,keepaspectratio]{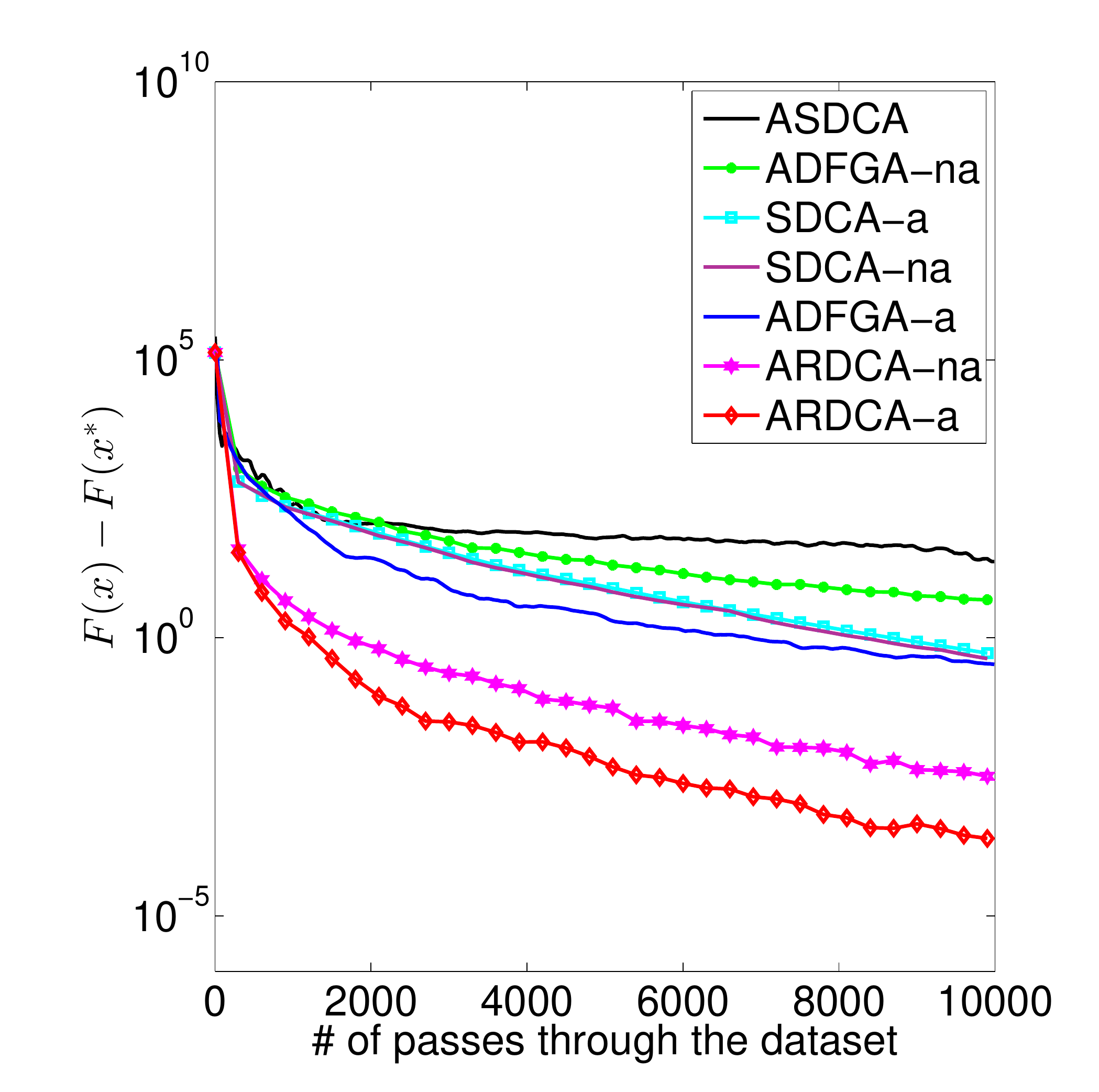}\\
$\lambda=10^{-3}$&$\lambda=10^{-4}$&$\lambda=10^{-5}$\\
\end{tabular}
\vspace*{-0.3cm}
\caption{Comparing ARDCA with SDCA, ASDCA and ADFGA on problem (\ref{exp_problem2}).}\label{fig2}
\end{figure}

\begin{figure}
\centering
\begin{tabular}{@{\extracolsep{0.001em}}c@{\extracolsep{0.001em}}c@{\extracolsep{0.001em}}c@{\extracolsep{0.001em}}c}
\includegraphics[width=0.35\textwidth,keepaspectratio]{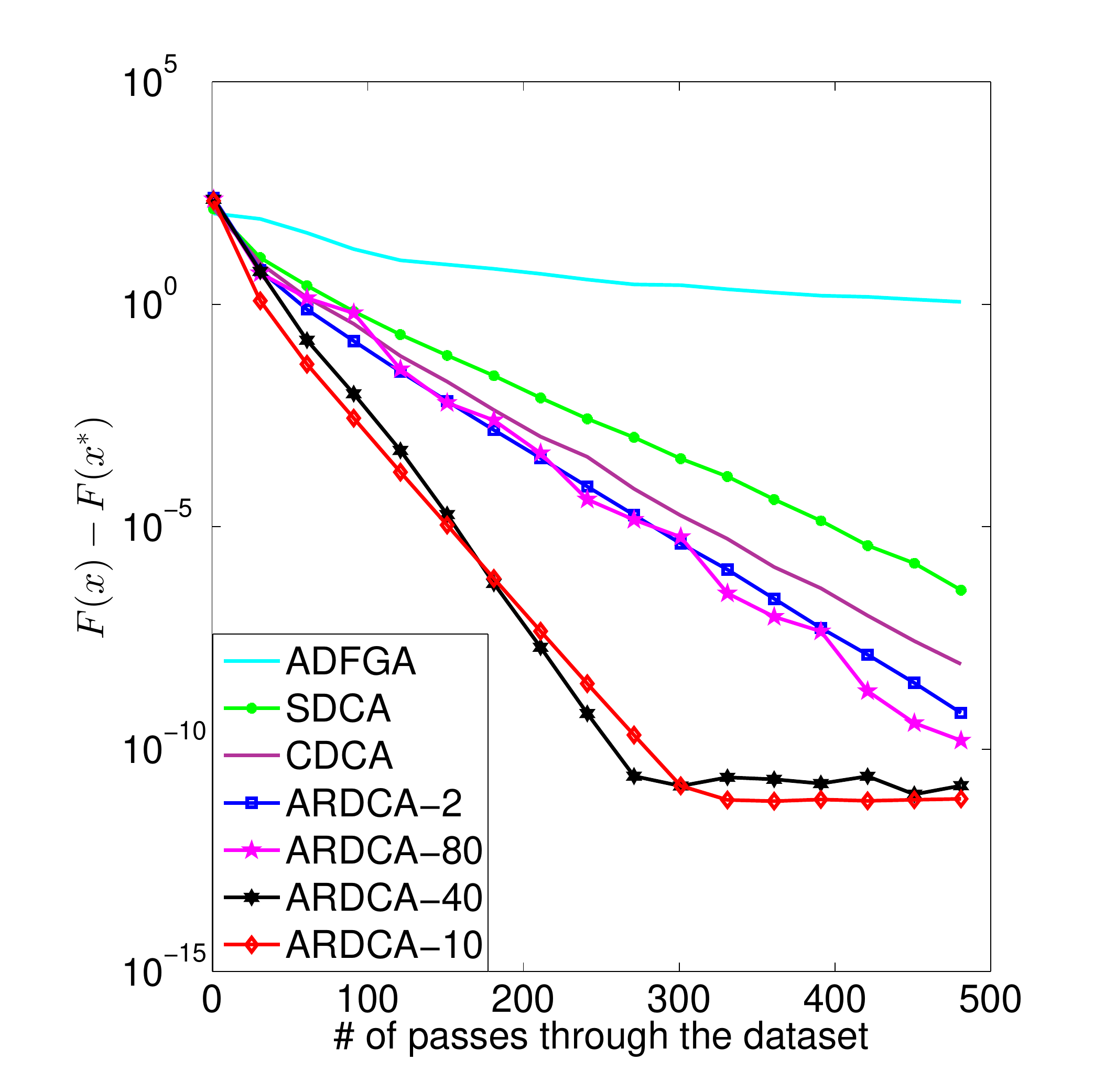}
&\hspace*{-0.2cm}\includegraphics[width=0.35\textwidth,keepaspectratio]{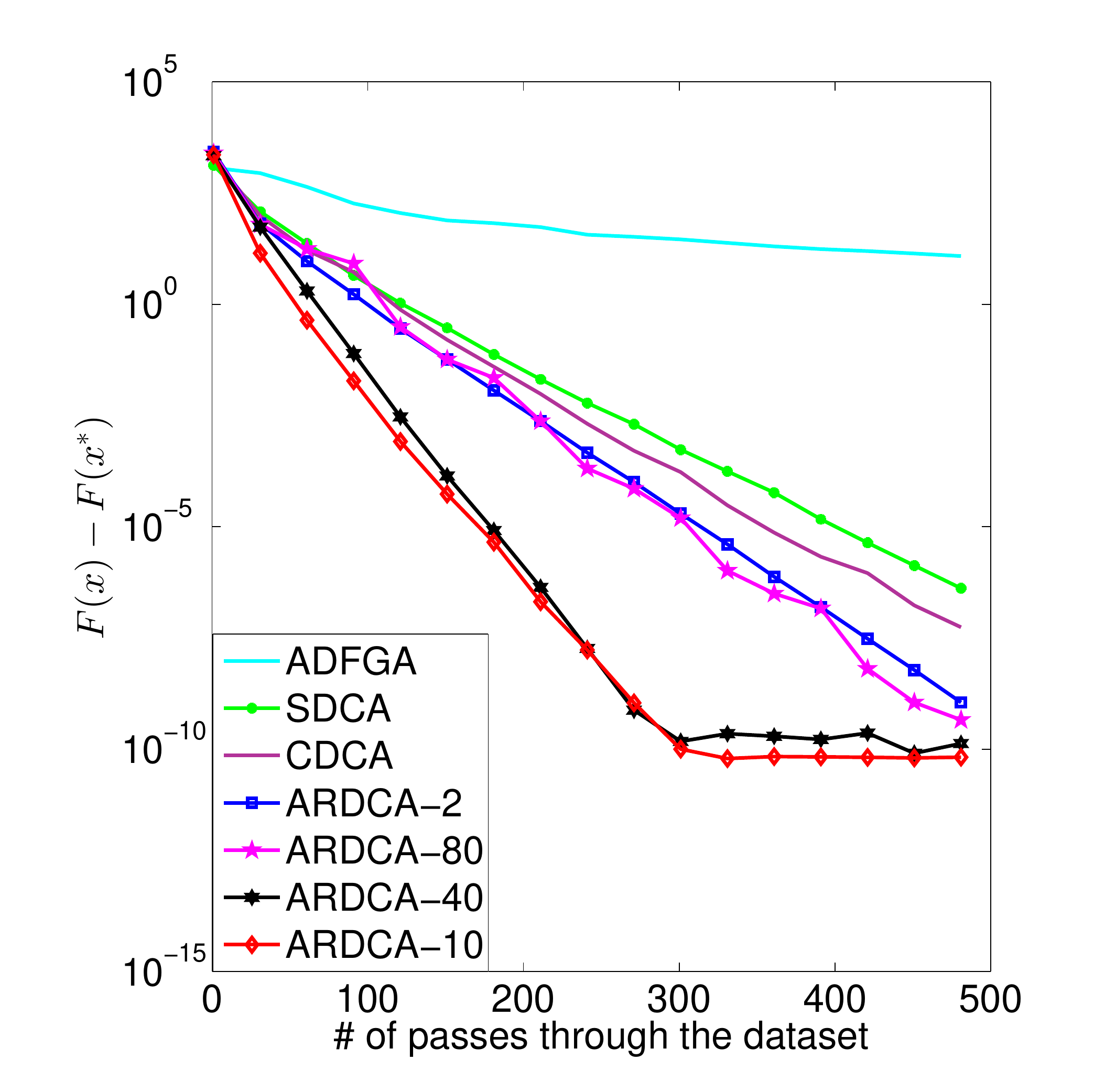}
&\hspace*{-0.2cm}\includegraphics[width=0.35\textwidth,keepaspectratio]{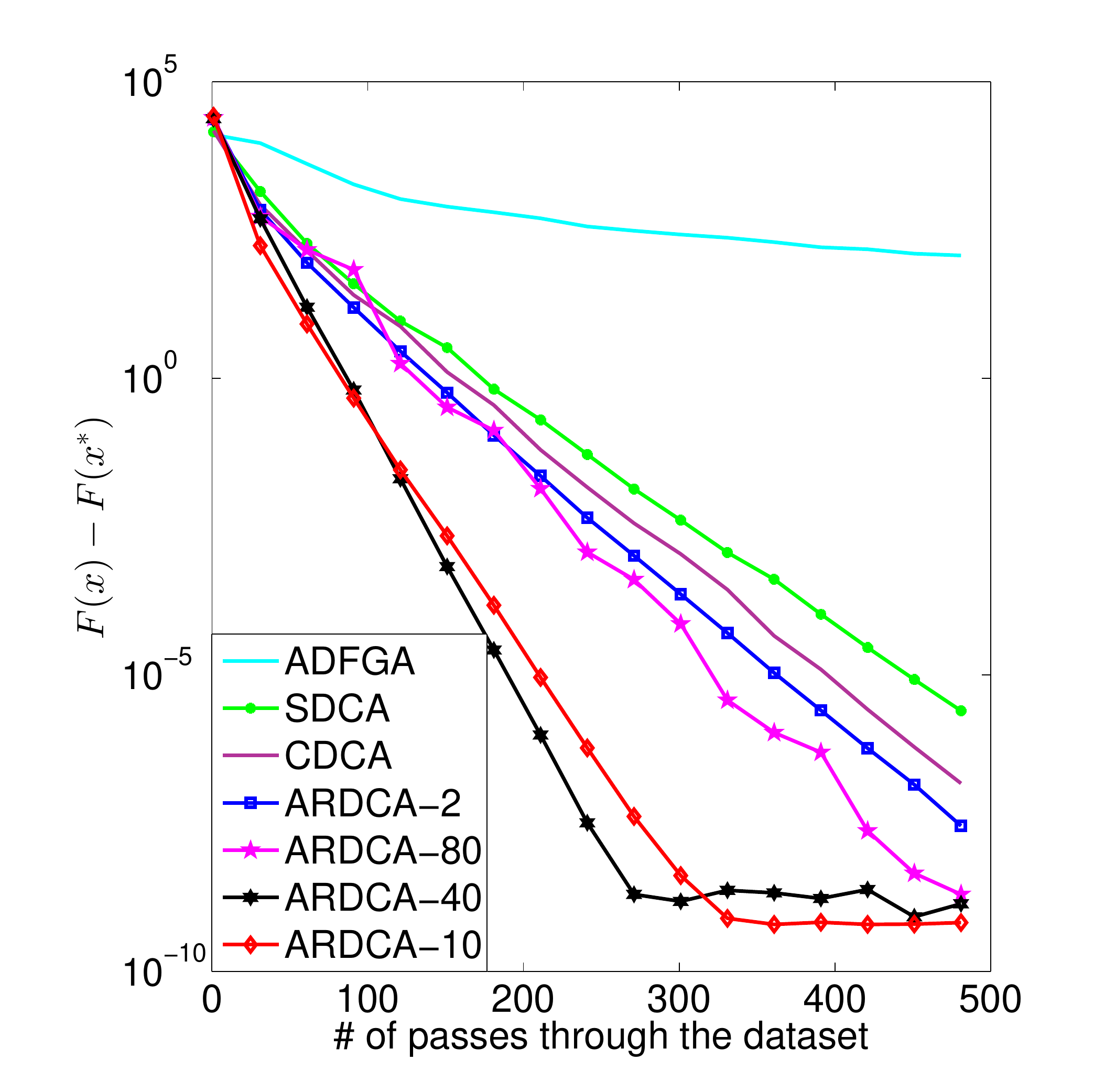}\\
$\lambda=10^{-3}$&$\lambda=10^{-4}$&$\lambda=10^{-5}$\\
\end{tabular}
\vspace*{-0.3cm}
\caption{Comparing ARDCA-restart with SDCA, CDCA and ADFGA on problem (\ref{exp_problem2}) with smooth $f(\x)$.}\label{fig4}
\end{figure}

\begin{figure}
\centering
\begin{tabular}{@{\extracolsep{0.001em}}c@{\extracolsep{0.001em}}c@{\extracolsep{0.001em}}c@{\extracolsep{0.001em}}c}
\includegraphics[width=0.35\textwidth,keepaspectratio]{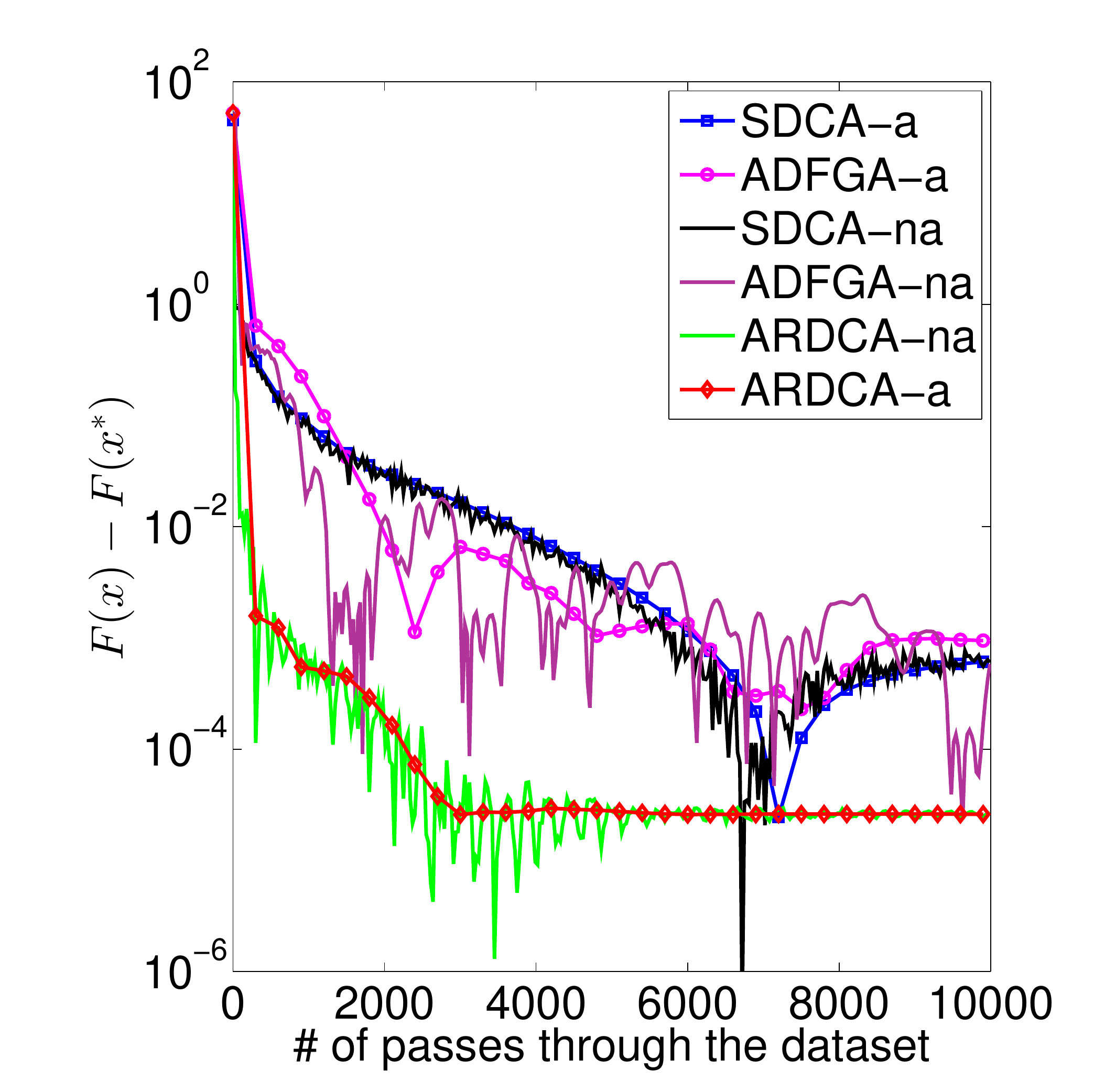}
&\hspace*{-0.2cm}\includegraphics[width=0.35\textwidth,keepaspectratio]{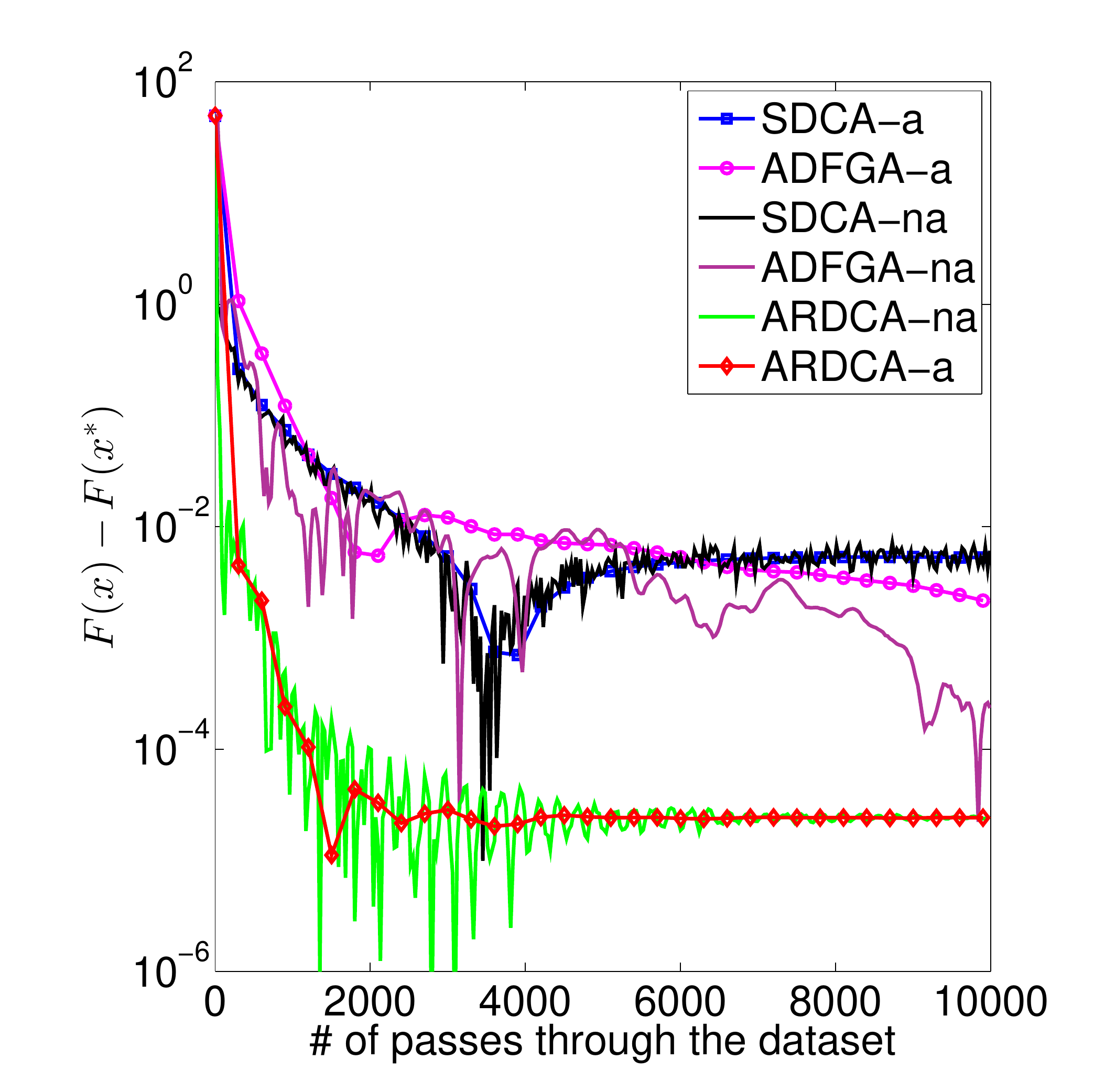}
&\hspace*{-0.2cm}\includegraphics[width=0.35\textwidth,keepaspectratio]{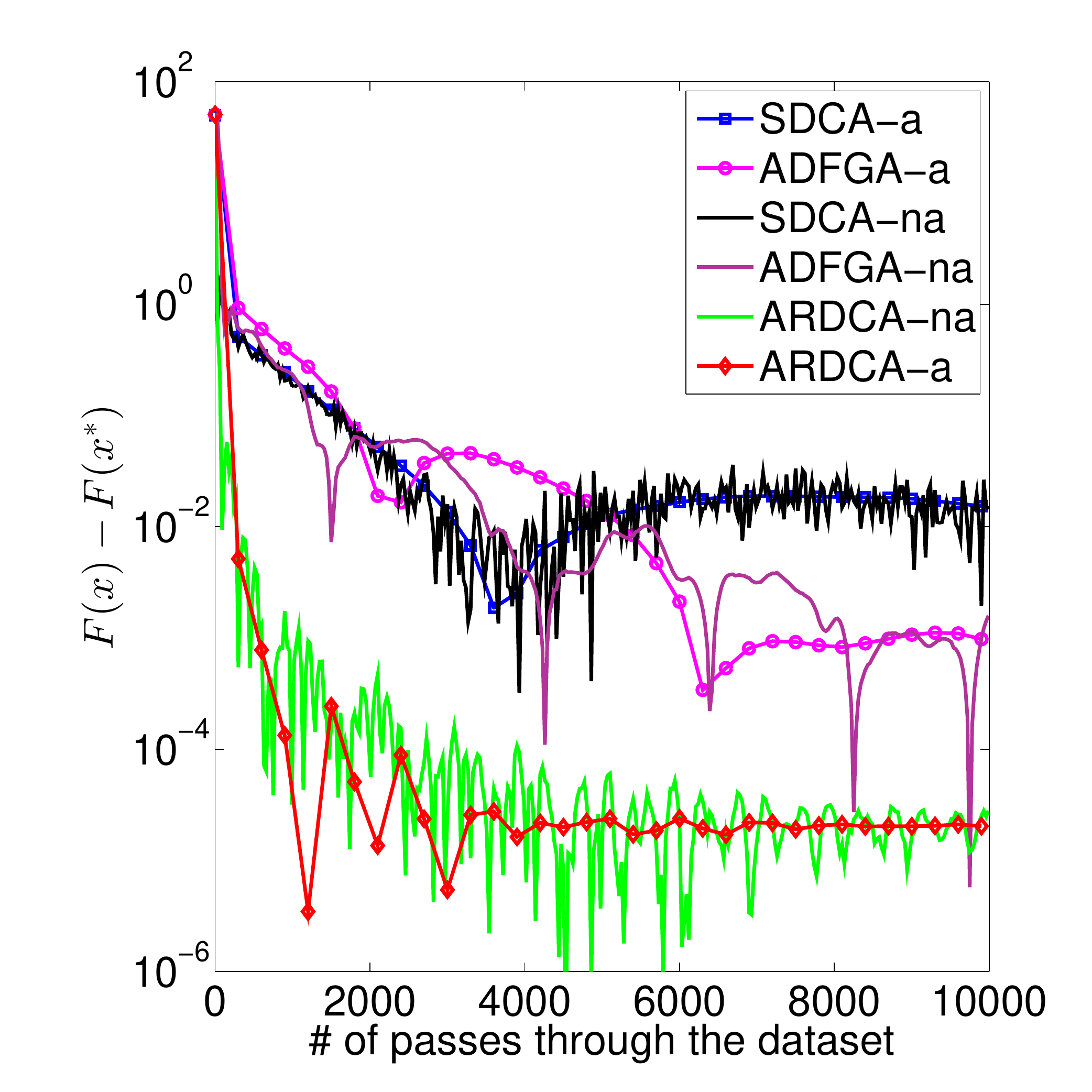}\\
\hspace*{0.4cm}\includegraphics[width=0.32\textwidth,keepaspectratio]{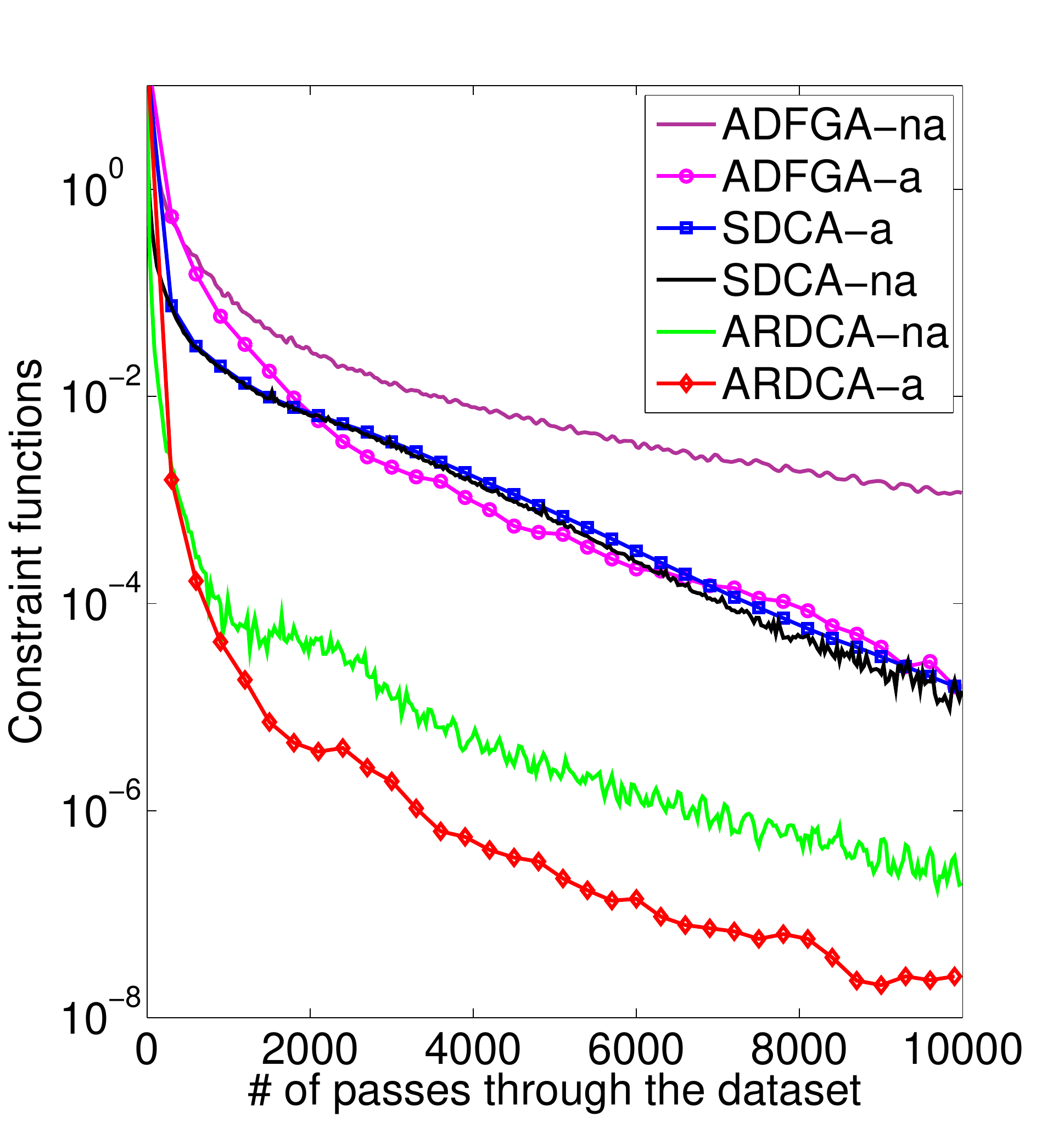}
&\hspace*{0.1cm}\includegraphics[width=0.32\textwidth,keepaspectratio]{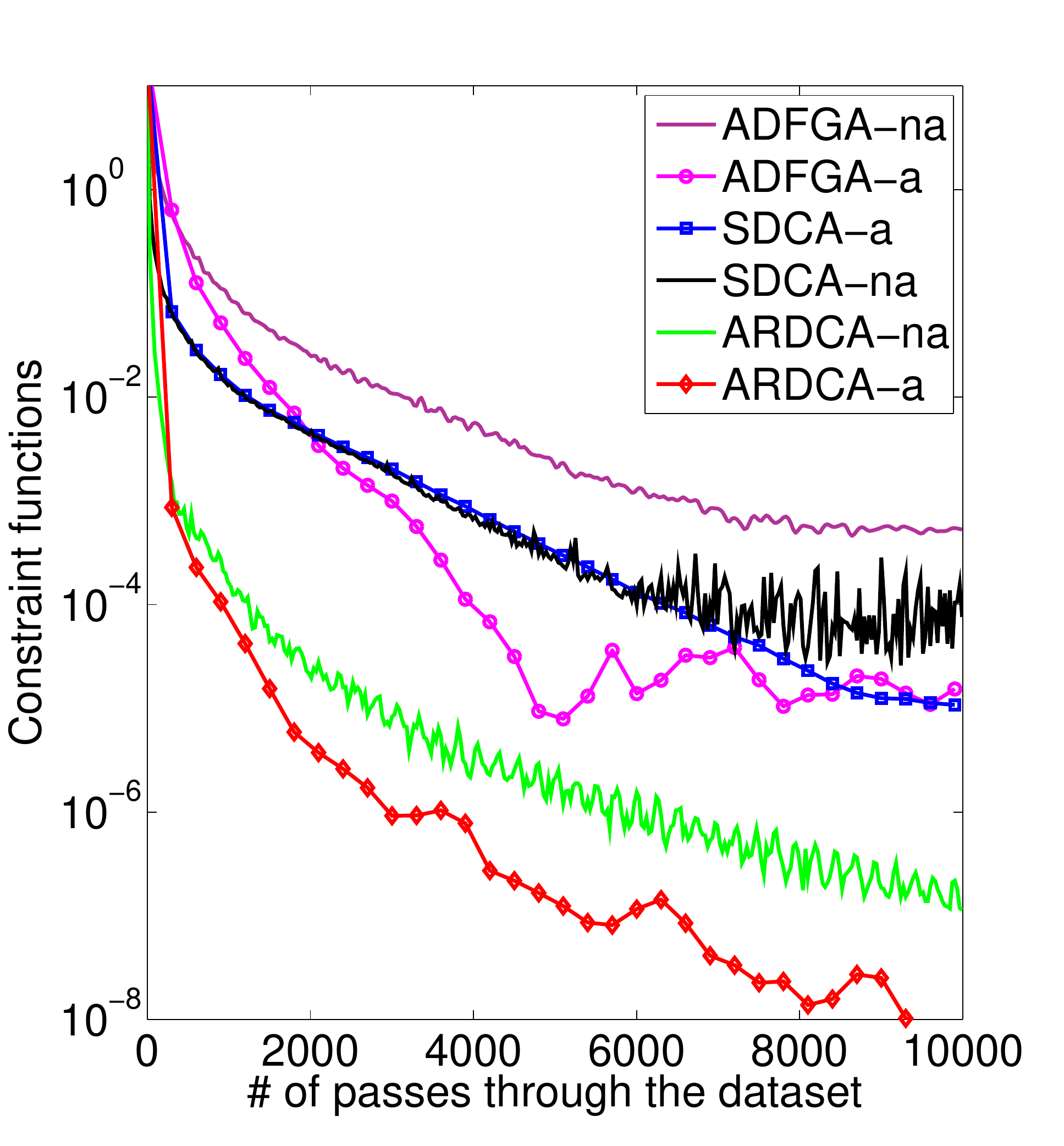}
&\hspace*{0.1cm}\includegraphics[width=0.32\textwidth,keepaspectratio]{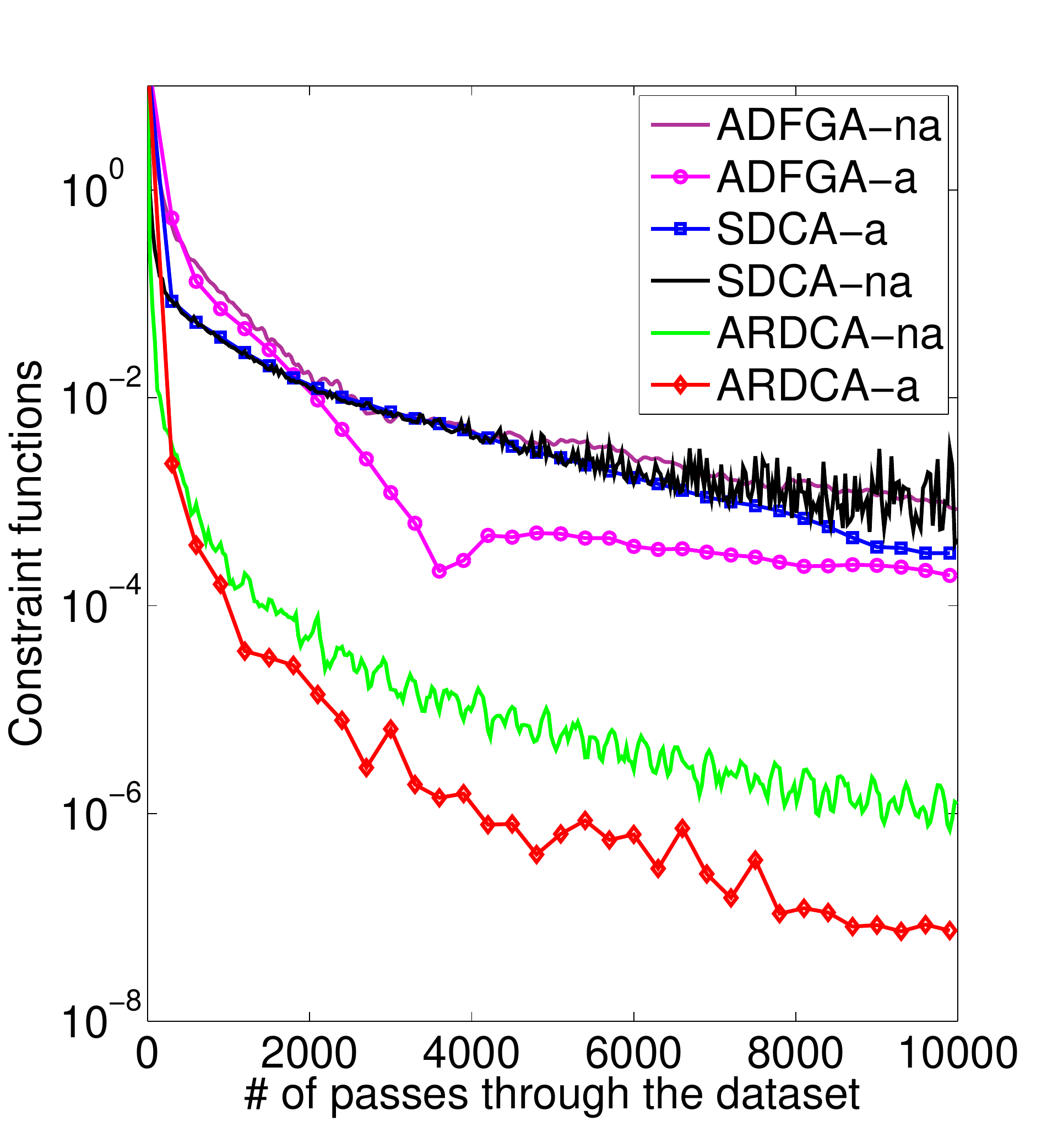}\\
$\tau=10^{-3}$&$\tau=10^{-4}$&$\tau=10^{-5}$\\
\end{tabular}
\vspace*{-0.3cm}
\caption{Comparing ARDCA with SDCA and ADFGA on problem (\ref{exp_problem3}). Top: objective function. Bottom: constraint functions.}\label{fig3}
\end{figure}

At last, we consider problem (\ref{exp_problem2}) with $f(\x)=\frac{\mu}{2}\|\x\|^2$ to verify the conclusions in Section \ref{sec_erm3}. In this scenario, we generate $\x$ to be a dense vector. We compare ARDCA with restart (Algorithm \ref{alg_arpcg_linear}) with SDCA \cite{zhang-2013-jmlr}, the Cyclic Dual Coordinate Ascent (CDCA) \cite{lin2014} and ADFGA \cite{Beck-2014}. Since the quadratic functional growth parameter $\kappa$ is unknown, we test Algorithm \ref{alg_arpcg_linear} with different inner iteration number $K_t\in\{2n,10n,40n,80n\}$. From Figure \ref{fig4} we can see that ARDCA, SDCA and CDCA all converge linearly and ARDCA with suitable $K$ performs the best. This verifies our theories in Section \ref{sec_erm3}.

\section{Conclusion}

In this paper, we prove that the iteration complexities of the primal solutions and dual solutions have the same order of magnitude for the accelerated randomized dual coordinate ascent. Specifically, when $f(\x)$ is $\mu$-strongly convex and the objectives are nonsmooth, we establish the $O\left(\frac{1}{\sqrt{\epsilon}}\right)$ iteration complexity. When the dual function further satisfies the quadratic functional growth condition, we prove the linear iteration complexity even if the condition number is unknown. When applied to the regularized empirical risk minimization problem, we prove the iteration complexity of $O\left(n\log n+\sqrt{\frac{n}{\epsilon}}\right)$, which outperforms the existing results by a $\left(\log\frac{1}{\epsilon}\right)$ factor. We also prove the accelerated linear convergence rate for some special problems with nonsmooth loss, e.g., the least absolute deviation and SVM. All the above results are established for both the primal solutions and dual solutions. The topic on the complexity analysis of the primal solutions is significant not only in stochastic optimization but also in distributed optimization. We hope that the analysis in this paper could facilitate more studies on this topic.

\section*{Appendix A: Efficient Computation of the Average}

We discuss the efficient computation of $\hat\x^K=\frac{\sum_{k=K_0}^K\frac{\x^*(\vv^k)}{\theta_k}}{\sum_{k=K_0}^K\frac{1}{\theta_k}}$. We define two variables $\mbox{sum}(\x,k)$ and $\mbox{sum}(\theta,k)$, update $\mbox{sum}(\x,k)=\mbox{sum}(\x,k-1)+\frac{\x^*(\vv^k)}{\theta_k}$ and $\mbox{sum}(\theta,k)=\mbox{sum}(\theta,k-1)+\frac{1}{\theta_k}$ at each iteration of Algorithm \ref{alg_arpcg}. We only store $\mbox{sum}(\x,k)$ and $\mbox{sum}(\theta,k)$ when $k=1,\lceil \nu(1+1/\hat n) \rceil,\lceil \nu(1+1/\hat n) \rceil^2,\lceil \nu(1+1/\hat n) \rceil^3,\cdots$. When the algorithm terminate at the $K$-th iteration, we let $K_0=\lceil \nu(1+1/\hat n) \rceil^p$ such that $\lceil \nu(1+1/\hat n) \rceil^{p+1}\leq K <\lceil \nu(1+1/\hat n) \rceil^{p+2}$ and compute $\hat\x^K=\frac{\mbox{sum}(\x,K)-\mbox{sum}(\x,K_0)}{\mbox{sum}(\theta,K)-\mbox{sum}(\theta,K_0)}$. Thus, we only need $O(t)$ computation time at each iteration and $O(t\log K)$ storage space in total, where $t$ is the dimension of $\x$.

\section*{Appendix B: Proof of Lemma \ref{coordinate_LS}}
\begin{proof}
Let $\hat\A=[\A/n,\B^T]$. From the proof of Theorem 3.1 in \cite{Lu-2016-siam}, we have
\begin{eqnarray}
\|\x^*(\uu)-\x^*(\vv)\|\leq \frac{1}{\mu}\|\hat\A\uu_{1:n+p}-\hat\A\vv_{1:n+p}\|+ \frac{1}{\mu}\sum_{i=1}^mL_{g_i}|\uu_{n+p+i}-\vv_{n+p+i}|.\notag
\end{eqnarray}
If $j\leq n+p$, then we have
\begin{eqnarray}
\begin{aligned}
&\|\x^*(\uu)-\x^*(\vv)\|\leq \frac{1}{\mu}\|\hat\A_j\uu_j-\hat\A_j\vv_j\|\leq\frac{\|\hat\A_j\|}{\mu}|\uu_j-\vv_j|=\frac{\|\hat\A_j\|}{\mu}\|\uu-\vv\|,\notag\\
&\|\nabla_j d(\uu)\hspace*{-0.05cm}-\hspace*{-0.05cm}\nabla_j d(\vv)\|^2\hspace*{-0.05cm}=\hspace*{-0.05cm}\|\hat\A_j^T\x^*(\uu)\hspace*{-0.05cm}-\hspace*{-0.05cm}\hat\A_j^T\x^*(\vv)\|^2\hspace*{-0.05cm}\leq\hspace*{-0.05cm}\|\hat\A_j\|^2\|\x^*(\uu)\hspace*{-0.05cm}-\hspace*{-0.05cm}\x^*(\vv)\|^2\hspace*{-0.05cm}\leq\hspace*{-0.05cm}\frac{\|\hat\A_j\|^4}{\mu^2}\|\uu\hspace*{-0.05cm}-\hspace*{-0.05cm}\vv\|^2.\notag
\end{aligned}
\end{eqnarray}
If $j>n+p$, then we have
\begin{eqnarray}
\begin{aligned}
&\|\x^*(\uu)-\x^*(\vv)\|\leq \frac{L_{g_{j-n-p}}}{\mu}|\uu_{j}-\vv_{j}|=\frac{L_{g_{j-n-p}}}{\mu}\|\uu-\vv\|,\notag\\
&\|\hspace*{-0.03cm}\nabla\hspace*{-0.03cm} d_j\hspace*{-0.03cm}(\hspace*{-0.03cm}\uu\hspace*{-0.03cm})\hspace*{-0.08cm}-\hspace*{-0.08cm}\nabla\hspace*{-0.03cm} d_j\hspace*{-0.03cm}(\hspace*{-0.03cm}\vv\hspace*{-0.03cm})\hspace*{-0.03cm}\|^2\hspace*{-0.08cm}=\hspace*{-0.08cm}|g_{j\hspace*{-0.03cm}-\hspace*{-0.03cm}n\hspace*{-0.03cm}-\hspace*{-0.03cm}p}(\hspace*{-0.03cm}\x^*\hspace*{-0.03cm}(\hspace*{-0.03cm}\uu\hspace*{-0.03cm})\hspace*{-0.03cm})\hspace*{-0.08cm}-\hspace*{-0.08cm}g_{j\hspace*{-0.03cm}-\hspace*{-0.03cm}n\hspace*{-0.03cm}-\hspace*{-0.03cm}p}(\hspace*{-0.03cm}\x^*\hspace*{-0.03cm}(\hspace*{-0.03cm}\vv\hspace*{-0.03cm})\hspace*{-0.03cm})\hspace*{-0.03cm}|^2\hspace*{-0.08cm}\leq\hspace*{-0.08cm} L_{g_{j\hspace*{-0.03cm}-\hspace*{-0.03cm}n\hspace*{-0.03cm}-\hspace*{-0.03cm}p}}^2\hspace*{-0.03cm}\|\hspace*{-0.03cm}\x^*\hspace*{-0.03cm}(\hspace*{-0.03cm}\uu\hspace*{-0.03cm})\hspace*{-0.08cm}-\hspace*{-0.08cm}\x^*\hspace*{-0.03cm}(\hspace*{-0.03cm}\vv\hspace*{-0.03cm})\hspace*{-0.03cm}\|^2\hspace*{-0.08cm}\leq\hspace*{-0.08cm}\frac{L_{g_{j\hspace*{-0.03cm}-\hspace*{-0.03cm}n\hspace*{-0.03cm}-\hspace*{-0.03cm}p}}^4}{\mu^2}\|\hspace*{-0.03cm}\uu\hspace*{-0.08cm}-\hspace*{-0.08cm}\vv\hspace*{-0.03cm}\|^2,\notag
\end{aligned}
\end{eqnarray}
which completes the proof.
\end{proof}

\section*{Appendix C: Proof of Lemma \ref{theta_lemma}}
\begin{proof}
From $\frac{1-\theta_{k}}{\theta_{k}^2}=\frac{1}{\theta_{k-1}^2}$, we can immediately prove the first two properties. We also have $\left(\frac{1}{\theta_k}-\frac{1}{2}-\frac{1}{2\hat n}\right)^2\leq\frac{1}{\theta_{k-1}^2}\leq\left(\frac{1}{\theta_k}-\frac{1}{2}\right)^2$, which leads to $\frac{k}{2}+\frac{k}{2\hat n}+\hat n\geq\frac{1}{\theta_k}\geq\frac{k}{2}+\hat n$. So we get $\frac{1}{\theta_K^2}-\frac{1}{\theta_{K_0-1}^2}\geq \left(\frac{K^2}{4}+\hat nK\right)\left(1-\frac{1}{\upsilon}\right)$ by letting $K_0\leq\left\lfloor\frac{K}{\upsilon(1+1/\hat n)}+1\right\rfloor$ for any $\upsilon>1$.
\end{proof}

\section*{Appendix D: Proof of Lemma \ref{lemma3}}
Lemma \ref{lemma3} can be proved by the techniques in \cite{fercoq-2015-siam} with only a little changes. We give the proof for the sake of completeness.
\begin{proof}
From the optimality condition of $\widetilde\z_i^k$, we have
\begin{eqnarray}
0\in 2\hat n\theta_k L_i(\widetilde\z_i^k-\z_{i}^k)+\nabla_{i}d(\vv^k)+\partial h_{i}(\widetilde\z_i^k),\quad\forall i=1,2,\cdots,\hat n.\notag
\end{eqnarray}
Thus, for any $\uu\in\DS$ and any $i=1,\cdots,n+p+m$, we have
\begin{eqnarray}\label{cont9}
h_{i}(\uu_{i})-h_{i}(\widetilde\z_i^k)\geq 2\hat n\theta_k L_i\<\widetilde\z_i^k-\z_{i}^k,\widetilde\z_i^k-\uu_{i}\>+\<\nabla_{i}d(\vv^k),\widetilde\z_i^k-\uu_{i}\>-\sigma_1(\uu_i,\widetilde\z_i^k),
\end{eqnarray}
where we use the convexity of $h_i(u)$ for $i>n$ and the definition of $\sigma_1(\uu_i,\widetilde\z_i^k)$ for $i\leq n$. Since $\uu^{k+1}\in \DS$, $\vv^k\in \DS$ and $\uu^{k+1}_j=\vv^k_j,\forall j\neq i_k$, then from (\ref{C_Lipschitz_smooth}), we have
\begin{eqnarray}
\begin{aligned}
&d(\uu^{k+1})\leq d(\vv^k)+\<\nabla_{i_k}d(\vv^k),\uu_{i_k}^{k+1}-\vv_{i_k}^k\>+\frac{L_{i_k}}{2}\|\uu_{i_k}^{k+1}-\vv_{i_k}^k\|^2.\notag
\end{aligned}
\end{eqnarray}
Using the relations of $\uu^{k+1}-\vv^k=\hat n\theta_k(\z^{k+1}-\z^k)$ and $\theta_k\z^k=\vv^k-(1-\theta_k)\uu^k$ in Algorithm \ref{alg_arpcg_equ}, we have
\begin{eqnarray}
\begin{aligned}
d(\uu^{k+1})\leq& d(\vv^k)+\<\nabla_{i_k}d(\vv^k),\hat n\theta_k(\widetilde\z_{i_k}^k-\z_{i_k}^k)\>+\frac{\hat n^2\theta_k^2L_{i_k}}{2}\|\widetilde\z_{i_k}^k-\z_{i_k}^k\|^2\notag\\
=& d(\vv^k)+\<\nabla_{i_k}d(\vv^k),\hat n\left[\theta_{k}\widetilde\z_{i_k}^k+(1-\theta_k)\uu_{i_k}^k-\vv_{i_k}^k\right]\>+\frac{\hat n^2\theta_k^2L_{i_k}}{2}\|\widetilde\z_{i_k}^k-\z_{i_k}^k\|^2\notag\\
=& d(\vv^k)+\<\nabla_{i_k}d(\vv^k),\hat n\left[\theta_{k}(\widetilde\z_{i_k}^k\hspace*{-0.05cm}-\hspace*{-0.05cm}\vv_{i_k}^k)\hspace*{-0.05cm}+\hspace*{-0.05cm}(1\hspace*{-0.05cm}-\hspace*{-0.05cm}\theta_k)(\uu_{i_k}^k\hspace*{-0.05cm}-\hspace*{-0.05cm}\vv_{i_k}^k)\right]\>\hspace*{-0.05cm}+\hspace*{-0.05cm}\frac{\hat n^2\theta_k^2L_{i_k}}{2}\|\widetilde\z_{i_k}^k\hspace*{-0.05cm}-\hspace*{-0.05cm}\z_{i_k}^k\|^2.\notag
\end{aligned}
\end{eqnarray}
Taking expectation with respect to $i_k$ conditioned on $\xi_{k-1}$, we have
\begin{eqnarray}
\begin{aligned}
&\E_{i_k|\xi_{k-1}}\left[d(\uu^{k+1})\right]\notag\\
\leq&\frac{1}{\hat n}\sum_{i=1}^{\hat n}\left[ d(\vv^k)+\hat n\<\nabla_{i}d(\vv^k),\theta_{k}(\widetilde\z_{i}^k-\vv_{i}^k)+(1-\theta_k)(\uu_{i}^k-\vv_{i}^k)\>+\frac{\hat n^2\theta_k^2L_{i}}{2}\|\widetilde\z_{i}^k-\z_{i}^k\|^2\right]\notag\\
=&d(\vv^k)+(1-\theta_k)\<\nabla d(\vv^k),\uu^k-\vv^k\>\hspace*{-0.07cm}+\hspace*{-0.12cm}\sum_{i=1}^{\hat n}\hspace*{-0.07cm}\left[\theta_{k}\<\nabla_{i}d(\vv^k),\widetilde\z_{i}^k-\vv_{i}^k\>+\frac{\hat n\theta_k^2L_{i}}{2}\|\widetilde\z_{i}^k-\z_{i}^k\|^2\right]\notag\\
\overset{a}\leq& (1-\theta_k)d(\uu^k)+\theta_k d(\vv^k)\notag\\
&+\sum_{i=1}^{\hat n}\left[\theta_{k}\<\nabla_{i}d(\vv^k),\uu_{i}-\vv_{i}^k\>+\theta_{k}\<\nabla_{i}d(\vv^k),\widetilde\z_{i}^k-\uu_{i}\>+\frac{\hat n\theta_k^2L_{i}}{2}\|\widetilde\z_{i}^k-\z_{i}^k\|^2\right]\notag\\
\overset{b}=& (1-\theta_k)d(\uu^k)+\theta_kd(\uu)+\theta_k \sigma_2(\uu,\vv^k)+\sum_{i=1}^{\hat n}\left[\theta_{k}\<\nabla_{i}d(\vv^k),\widetilde\z_{i}^k-\uu_{i}\>+\frac{\hat n\theta_k^2L_{i}}{2}\|\widetilde\z_{i}^k-\z_{i}^k\|^2\right]\notag\\
\overset{c}\leq&(1-\theta_k)d(\uu^k)+\theta_k d(\uu)+\theta_k \sigma_2(\uu,\vv^k)+\sum_{i=1}^{\hat n}\left[\theta_k\left( h_{i}(\uu_{i})-h_{i}(\widetilde\z_{i}^k)\right)+\theta_k\sigma_1(\uu_{i},\widetilde\z_{i}^k)\right.\notag\\
&\left.-2\hat n\theta_k^2 L_{i}\<\widetilde\z_{i}^k-\z_{i}^k,\widetilde\z_{i}^k-\uu_{i}\>+\frac{\hat n\theta_k^2L_{i}}{2}\|\widetilde\z_{i}^k-\z_{i}^k\|^2\right]\notag\\
\overset{d}=&(1-\theta_k)d(\uu^k)+\theta_k d(\uu)+\theta_k \sigma_2(\uu,\vv^k)+\sum_{i=1}^{\hat n}\left[\theta_k\left( h_{i}(\uu_{i})-h_{i}(\widetilde\z_{i}^k)\right)+\theta_k\sigma_1(\uu_{i},\widetilde\z_{i}^k)\right.\notag\\
&\left.+\hat n\theta_k^2 L_{i}\left[ \|\z_{i}^k-\uu_{i}\|^2-\|\widetilde\z_{i}^k-\uu_{i}\|^2 \right]-\frac{\hat n\theta_k^2 L_{i}}{2}\|\widetilde\z_{i}^k-\z_{i}^k\|^2\right]\notag\\
\overset{e}=&(1-\theta_k)d(\uu^k)+\theta_k  D(\uu)+\theta_k \sigma_2(\uu,\vv^k)+\theta_k\sum_{i=1}^{n}\sigma_1(\uu_{i},\widetilde\z_{i}^k)+(1-\theta_k)H^k\notag\\
&-\hspace*{-0.07cm}\E_{i_k|\xi_{k-1}}[H^{k+1}]\hspace*{-0.07cm}+\hspace*{-0.07cm}\hat n^2\theta_k^2\hspace*{-0.07cm}\left[ \|\z^k\hspace*{-0.07cm}-\hspace*{-0.07cm}\uu\|_L^2\hspace*{-0.07cm}-\hspace*{-0.07cm}\E_{i_k|\xi_{k-1}}[\|\z^{k+1}\hspace*{-0.07cm}-\hspace*{-0.07cm}\uu\|_L^2] \right]\hspace*{-0.07cm}-\hspace*{-0.07cm}\frac{\hat n^2\theta_k^2}{2}\E_{i_k|\xi_{k-1}}[\|\z^{k+1}\hspace*{-0.07cm}-\hspace*{-0.07cm}\z^k\|_L^2],\notag
\end{aligned}
\end{eqnarray}
where we use the convexity of $d$ in $\overset{a}\leq$, the definition of $\sigma_2(\uu,\vv^k)$ in $\overset{b}=$, (\ref{cont9}) in $\overset{c}\leq$, $2\<a-b,a-c\>=\|a-b\|^2+\|a-c\|^2-\|b-c\|^2$ in $\overset{d}=$ and the following three equations in $\overset{e}=$, which can be obtained from Lemma 4 and Equation (45) in \cite{fercoq-2015-siam},
\begin{eqnarray}
&&\E_{i_k|\xi_{k-1}}[\|\z^{k+1}-\uu\|_L^2]=\frac{1}{\hat n}\sum_{i=1}^{\hat n}\left[L_{i}\|\widetilde\z_{i}^k-\uu_{i}\|^2+\sum_{j\neq i}L_j\|\z_j^{k}-\uu_j\|^2 \right]\notag\\
&&\hspace*{3.62cm}=\frac{1}{\hat n}\sum_{i=1}^{\hat n}L_{i}\|\widetilde\z_{i}^k-\uu_{i}\|^2+\frac{\hat n-1}{\hat n}\sum_{j=1}^{\hat n}L_j\|\z_j^{k}-\uu_j\|^2,\notag\\
&&\E_{i_k|\xi_{k-1}}[H^{k+1}]=\sum_{t=0}^{k} \alpha_{k+1,t} h(\z^{t})+\E_{i_k|\xi_{k-1}}[\hat n\theta_k h(\z^{k+1})]\notag\\
&&\hspace*{2.46cm}=\sum_{t=0}^{k} \alpha_{k+1,t} h(\z^t)+\frac{1}{\hat n}\sum_{i=1}^{\hat n}\hat n\theta_k\left( h_{i}(\widetilde\z_{i}^k)+\sum_{j\neq i}h_j(\z_j^k)\right)\notag\\
&&\hspace*{2.46cm}=\sum_{t=0}^{k} \alpha_{k+1,t} h(\z^t)+\theta_k\sum_{i=1}^{\hat n} h_{i}(\widetilde\z_{i}^k)+(\hat n-1)\theta_k\sum_{i=1}^{\hat n}h_i(\z_i^k)\notag\\
&&\hspace*{2.46cm}=\sum_{t=0}^{k-1} \alpha_{k+1,t} h(\z^t)+\alpha_{k+1,k} h(\z^k)+(\hat n-1)\theta_k h(\z^k)+\theta_k\sum_{i=1}^{\hat n} h_{i}(\widetilde\z_{i}^k)\notag\\
&&\hspace*{2.46cm}=(1-\theta_k)\sum_{t=0}^{k-1} \alpha_{k,t} h(\z^t)+(1-\theta_k)\alpha_{k,k}h(\z^k)+\theta_k\sum_{i=1}^{\hat n} h_{i}(\widetilde\z_{i}^k)\notag\\
&&\hspace*{2.46cm}=(1-\theta_k)H^k+\theta_k\sum_{i=1}^{\hat n} h_{i}(\widetilde\z_{i}^k),\notag\\
&&\E_{i_k|\xi_{k-1}}[\|\z^{k+1}-\z^k\|_L^2]=\frac{1}{\hat n}\sum_{i=1}^{\hat n}\left( L_{i}\|\widetilde\z_{i}^k-\z_{i}^k\|^2+\sum_{j\neq i}L_j\|\z_j^k-\z_j^k\|^2 \right)\notag\\
&&\hspace*{3.82cm}=\frac{1}{\hat n}\sum_{i=1}^{\hat n}L_{i}\|\widetilde\z_{i}^k-\z_{i}^k\|^2.\label{cont61}
\end{eqnarray}
Rearranging the terms, we have
\begin{eqnarray}
\begin{aligned}\notag
&\E_{i_k|\xi_{k-1}}\left[d(\uu^{k+1})+H^{k+1}- D(\uu)+\hat n^2\theta_k^2\|\z^{k+1}-\uu\|_L^2+\frac{\hat n^2\theta_k^2}{2}\|\z^{k+1}-\z^k\|_L^2\right]\\
\leq&(1-\theta_k)\left[d(\uu^k)+H^k- D(\uu)\right]+\hat n^2\theta_k^2\|\z^k-\uu\|_L^2+\theta_k\left(\sum_{i=1}^n\sigma_1(\uu_i,\widetilde\z_i^k)+\sigma_2(\uu,\vv^k)\right).
\end{aligned}
\end{eqnarray}
Taking expectation with respect to $\xi_{k-1}$ on both sides, we have
\begin{eqnarray}
\begin{aligned}
&\E_{\xi_k}\left[d(\uu^{k+1})+H^{k+1}- D(\uu)+\hat n^2\theta_k^2\|\z^{k+1}-\uu\|_L^2+\frac{\hat n^2\theta_k^2}{2}\|\z^{k+1}-\z^k\|_L^2\right]\notag\\
\leq&(1\hspace*{-0.09cm}-\hspace*{-0.09cm}\theta_k)\E_{\xi_{k-1}}\hspace*{-0.11cm}\left[d(\uu^k)\hspace*{-0.09cm}+\hspace*{-0.09cm}H^k\hspace*{-0.09cm}-\hspace*{-0.09cm} D(\uu)\right]\hspace*{-0.09cm}+\hspace*{-0.09cm}\hat n^2\theta_k^2\E_{\xi_{k-1}}\hspace*{-0.09cm}[\|\z^k-\uu\|_L^2]\hspace*{-0.09cm}+\hspace*{-0.09cm}\theta_k\E_{\xi_{k-1}}\hspace*{-0.11cm}\left[\sum_{i=1}^n\hspace*{-0.09cm}\sigma_1(\uu_i,\widetilde\z_i^k)\hspace*{-0.09cm}+\hspace*{-0.09cm}\sigma_2(\uu,\vv^k)\hspace*{-0.09cm}\right]\hspace*{-0.09cm}.\notag
\end{aligned}
\end{eqnarray}
Since $\{\uu^{k+1}$, $\z^{k+1}$, $H^{k+1}\}$ are independent on $\{i_{k+1},\cdots,i_K\}$ and $\{\uu^k$, $\z^k$, $\widetilde\z^k$, $H^k$, $\sigma_1(\uu,\widetilde\z^k)$, $\sigma_2(\uu,\vv^k)\}$ are independent on $\{i_k,\cdots,i_K\}$, we have (\ref{cont4}).
\end{proof}

\section*{Appendix E: Proof of Lemma \ref{erm_lemma}}.
\begin{proof}
We first prove (\ref{cont57}). From (\ref{cont4}) and $\|\z^k-\uu^*\|_L^2\leq \frac{4M^2}{n\mu}$, we have
\begin{eqnarray}
\begin{aligned}
&\E_{\xi_{K'}}[D(\uu^{{K'+1}})]-D(\uu^*)\leq \E_{\xi_{K'}}[d(\uu^{{K'+1}})+H^{{K'+1}}]-D(\uu^*)\notag\\
\leq&\left(1-\frac{1}{n}\right)\left(\E_{\xi_{K'}}[d(\uu^{K'})+H^{K'}]-D(\uu^*)\right)+\frac{1}{n}\E_{\xi_{K'}}\left[\sum_{i=1}^n\sigma_1(\uu_i^*,\widetilde\z_i^{K'})+\sigma_2(\uu^*,\vv^{K'})\right]\notag\\
&+\E_{\xi_{K'}}[\|\z^{K'}-\uu^*\|_L^2]-\E_{\xi_{K'}}[\|\z^{K'+1}-\uu^*\|_L^2]\notag\\
\leq&\left(1-\frac{1}{n}\right)^{K'+1}\left(D(\uu^0)-D(\uu^*)\right)+\frac{1}{n}\E_{\xi_{K'}}\left[\sum_{i=1}^n\sigma_1(\uu_i^*,\widetilde\z_i^{K'})+\sigma_2(\uu^*,\vv^{K'})\right]\notag\\
&+\sum_{k=1}^{K'}\left(1-\frac{1}{n}\right)^{K'-k}\frac{1}{n}\E_{\xi_{K'}}[\|\z^{k}-\uu^*\|_L^2]+\left(1-\frac{1}{n}\right)^{K'}\|\z^{0}-\uu^*\|_L^2\notag\\
\leq&\exp\left(-\frac{K'+1}{n}\right)\left(D(\uu^0)-D(\uu^*)\right)+\frac{1}{n}\E_{\xi_{K'}}\left[\sum_{i=1}^n\sigma_1(\uu_i^*,\widetilde\z_i^{K'})+\sigma_2(\uu^*,\vv^{K'})\right]+\frac{8M^2}{n\mu}\notag\\
\leq&9\max\left\{\epsilon,\frac{M^2}{n\mu}\right\}+\frac{1}{n}\E_{\xi_{K'}}\left[\sum_{i=1}^n\sigma_1(\uu_i^*,\widetilde\z_i^{K'})+\sigma_2(\uu^*,\vv^{K'})\right]\leq 9\max\left\{\epsilon,\frac{M^2}{n\mu}\right\},\notag
\end{aligned}
\end{eqnarray}
which leads to (\ref{cont57}) and
\begin{eqnarray}
\begin{aligned}
-\frac{1}{n}\E_{\xi_{K'}}\left[\sum_{i=1}^n\sigma_1(\uu_i^*,\widetilde\z_i^{K'})+\sigma_2(\uu^*,\vv^{K'})\right]\leq 9\max\left\{\epsilon,\frac{M^2}{n\mu}\right\}.\label{cont59}
\end{aligned}
\end{eqnarray}
Then we prove (\ref{cont58}). From (\ref{pd_relation1}) and (\ref{cont59}), we have
\begin{eqnarray}
\begin{aligned}\notag
&\E_{\xi_{K'}}[f(\x^*(\vv^{K'})]+\frac{1}{n}\E_{\xi_{K'}}[\phi(n\y^{K'})]+D(\uu^*)-\|\uu^*\|_L\E_{\xi_{K'}}[\|\A^T\x^*(\vv^{K'})/n-\y^{K'}\|_L^*]\\
\leq&-\frac{1}{n}\E_{\xi_{K'}}\left[\sum_{i=1}^n\sigma_1(\uu_i^*,\widetilde\z_i^{K'})+\sigma_2(\uu^*,\vv^{K'})\right]\leq 9\max\left\{\epsilon,\frac{M^2}{n\mu}\right\}.
\end{aligned}
\end{eqnarray}
From (\ref{cont41}) with $\theta_k=\frac{1}{\hat n}$ and $\|\widetilde\z^{K'}-\z^{K'}\|_L^2\leq \frac{4M^2}{n\mu}$, we have
\begin{eqnarray}
\begin{aligned}
(\|\A^T\x^*(\vv^{K'})/n-\y^{K'}\|_L^*)^2=4\|\widetilde\z^{K'}-\z^{K'}\|_L^2\leq 16\max\left\{\epsilon,\frac{M^2}{n\mu}\right\}.\notag
\end{aligned}
\end{eqnarray}
So from a similar induction to (\ref{cont7}), we have
\begin{eqnarray}
\begin{aligned}\notag
&\E_{\xi_{K'}}[f(\x^*(\vv^{K'})]+\frac{1}{n}\E_{\xi_{K'}}[\phi(\A^T\x^*(\vv^{K'}))]-f(\x^*)-\frac{1}{n}\phi(\A^T\x^*)\\
\leq&9\max\left\{\epsilon,\frac{M^2}{n\mu}\right\}+\left(\|\uu^*\|_L+M\sqrt{\sum_{i=1}^nL_i}\right)\E_{\xi_{K'}}[\|\A^T\x^*(\vv^{K'})/n-\y^{K'}\|_L^*]\\
\leq& 17\max\left\{\epsilon,\frac{M^2}{n\mu}\right\},
\end{aligned}
\end{eqnarray}
where we use $\|\uu^*\|_L^2\leq\frac{M^2}{n\mu}$ and $M\sqrt{\sum_{i=1}^nL_i}\leq\frac{M}{\sqrt{n\mu}}$.
\end{proof}

\section*{Appendix F: Analysis for the Complexity Comparisons in Section \ref{linar_sec}}
The complexity of Algorithm \ref{alg_arpcg_linear} is
\begin{equation}
O\left(K\frac{\log\frac{1}{\epsilon}}{\log \frac{1+\frac{\kappa}{2}\left(\frac{K}{2\hat n}+1\right)^2}{1+(1-\theta_0)\kappa} }\right).\label{cont60}
\end{equation}
Case 1: $\kappa<1$.

Letting $\frac{1+\frac{\kappa}{2}\left(\frac{K}{2\hat n}+1\right)^2}{1+(1-\theta_0)\kappa}$ be a constant, e.g., 2, we have $K=2\hat n\left(\sqrt{\frac{2}{\kappa}+4(1-\theta_0)}-1\right)$ and $2\hat n\left( \sqrt{\frac{1}{\kappa}}+\sqrt{2(1-\theta_0)}-1 \right)\leq K\leq 2\hat n\left( \sqrt{\frac{2}{\kappa}}+2\sqrt{1-\theta_0}-1 \right)$. So $K=O\left(\hat n+\frac{\hat n}{\sqrt{\kappa}}\right)$ and (\ref{cont60}) has the same order of magnitude as $\left(\hat n+\frac{\hat n}{\sqrt{\kappa}}\right)\log\frac{1}{\epsilon}$.

When $\hat n<K<\hat n+\frac{\hat n}{\sqrt{\kappa}}$, (\ref{cont60}) has the same order of magnitude as $K\frac{\log\frac{1}{\epsilon}}{\log \left(1+\frac{\kappa K^2}{\hat n^2}\right)}=O\left(\frac{\hat n^2}{\kappa K}\log\frac{1}{\epsilon}\right)$ and it is smaller that $\frac{\hat n}{\kappa}\log\frac{1}{\epsilon}$.

When $\hat n+\frac{\hat n}{\sqrt{\kappa}}<K<\hat n+\frac{\hat n}{\kappa}$, (\ref{cont60}) has the same order of magnitude as $K\frac{\log\frac{1}{\epsilon}}{\log \left(1+\frac{\kappa K^2}{\hat n^2}\right)}$ and it is also smaller than $\left(\hat n+\frac{\hat n}{\kappa}\right)\log\frac{1}{\epsilon}$.

\noindent Case 2: $\kappa>1$.

(\ref{cont60}) has the same order of magnitude as $K\log\frac{1}{\epsilon}=O\left(\left(\hat n+\frac{\hat n}{\kappa}\right)\log\frac{1}{\epsilon}\right)$ when $\hat n<K<\hat n+\frac{\hat n}{\kappa}$.

\small
\bibliographystyle{unsrt}
\bibliography{SDCA}

\begin{thebibliography}{10}

\bibitem{zhang-2013-jmlr}
Shai Shalev-Shwartz and Tong Zhang.
\newblock Stochastic dual coordinate ascent methods for regularized loss
  minimization.
\newblock {\em Journal of Machine Learning Research}, 14(1):567--599, 2013.

\bibitem{zhang-2015-MP}
Shai Shalev-Shwartz and Tong Zhang.
\newblock Accelerated proximal stochastic dual coordinate ascent for
  regularized loss minimization.
\newblock {\em Mathematical Programming}, 155(1):105--145, 2016.

\bibitem{Tseng-1990}
Paul Tseng.
\newblock Dual ascent methods for problems with strictly convex costs and
  linear constraints: A unified approach.
\newblock {\em SIAM J. on Optimization}, 28(1):214--242, 1990.

\bibitem{Beck-2014}
Amir Beck and Marc Teboulle.
\newblock Fast dual proximal gradient algorithm for convex minimization and
  applications.
\newblock {\em Operations Research Letters}, 42:1--6, 2014.

\bibitem{ma-2013}
Bohuang Huang, Shiqian Ma, and Donald Goldfarb.
\newblock Accelerated linearized {B}regman method.
\newblock {\em Journal of Scientific Computing}, 54(2-3):428--453, 2013.

\bibitem{Nesterov-2012}
Yurii Nesterov.
\newblock Efficiency of coordinate descent methods on huge-scale optimization
  problems.
\newblock {\em SIAM J. on Optimization}, 22(2):341--362, 2012.

\bibitem{lu-2015-MP}
Zhaosong Lu and Lin Xiao.
\newblock On the complexity analysis of randomized block-coordinate descent
  methods.
\newblock {\em Mathematical Programming}, 152:615--642, 2015.

\bibitem{richtarik-2014}
Peter Richt\'{a}rik and Martin Tak\'{a}\v{c}.
\newblock Iteration complexity of randomized block-coordinate descent methods
  for minimizing a composite function.
\newblock {\em Mathematical Programming}, 144(1-2):1--38, 2014.

\bibitem{fercoq-2015-siam}
Olivier Fercoq and Peter Richt\'{a}rik.
\newblock Accelerated, parallel, and proximal coordinate descent.
\newblock {\em SIAM J. on Optimization}, 25:1997--2023, 2015.

\bibitem{xiao-2015-siam}
Qihang Lin, Zhaosong Lu, and Lin Xiao.
\newblock An accelerated randomized \textcolor{red}{proximal} coordinate
  gradient method and its application to \textcolor{red}{regularized} empirical
  risk minimization.
\newblock {\em SIAM J. on Optimization}, 25(4):2244--2273, 2015.

\bibitem{Lu-2016-siam}
Jie Lu and Mikael Johansson.
\newblock Convergence analysis of aproximate primal solutions in dual first
  order methods.
\newblock {\em SIAM J. on Optimization}, 26(4):2430--2467, 2016.

\bibitem{Dunner-2016-icml}
Celestine D\"{u}nner, Simone Forte, Martin Tak\'{a}c, and Martin Jaggi.
\newblock Primal-dual rates and certificates.
\newblock In {\em ICML}, 2016.

\bibitem{Kim-2016}
Donghwan Kim and Jeffrey~A. Fessler.
\newblock Fast dual proximal gradient algorithms with rate $o(1/k^{1.5})$ for
  convex minimization.
\newblock In {\em arxiv:1609.09441}, 2016.

\bibitem{Devolder-2012-siam}
Olivier Devolder, Francois Glineur, and Yurii Nesterov.
\newblock Double smoothing technique for large-scale linearly constrained
  convex optimization.
\newblock {\em SIAM J. on Optimization}, 22:702--727, 2012.

\bibitem{Necoara-2016}
Ion Necoara and Andrei Patrascu.
\newblock Iteration complexity analysis of dual first-order methods for conic
  convex programming.
\newblock {\em Optimization Methods and Software}, 31:645--678, 2016.

\bibitem{Tseng-2008}
Paul Tseng.
\newblock On accelerated proximal gradient methods for convex-concave
  optimization.
\newblock Technical report, University of Washington, Seattle, 2008.

\bibitem{Necoara-2014-ieee}
Ion Necoara and Valentin Nedelcu.
\newblock Rate analysis of inexact dual first-order methods application to dual
  decomposition.
\newblock {\em IEEE Trans. on Automatic Control}, 59:1232--1243, 2014.

\bibitem{Patrinos-2013-ieee}
Panagiotis Patrinos and Alberto Bemporad.
\newblock An accelerated dual gradient projection algorithm for embedded linear
  model predictive control.
\newblock {\em IEEE Trans. on Automatic Control}, 59:18--33, 2013.

\bibitem{lin-2015-nips}
Hongzhou Lin, Julien Mairal, and Zaid Harchaoui.
\newblock A universal {C}atalyst for first-order optimization.
\newblock In {\em NIPS}, 2015.

\bibitem{zhang-2017-jmlr}
Yuchen Zhang and Lin Xiao.
\newblock Stochastic primal-dual coordinate method for regularized empirical
  risk minimization.
\newblock {\em Journal of Machine Learning Research}, 18(1):2939--2980, 2017.

\bibitem{lan-2017-MP}
Guanghui Lan and Yi~Zhou.
\newblock An optimal randomized incremental gradient method.
\newblock {\em Mathematical Programming}, 2017.

\bibitem{zhengqu-2018}
Olivier Fercoq and Zheng Qu.
\newblock Restarting the accelerated coordinate descent method with a rough
  strong convexity estimate.
\newblock In {\em arxiv:1803.05771}, 2018.

\bibitem{Woodworth-2016}
Blake Woodworth and Nathan Srebro.
\newblock Tight complexity bounds for optimizing composite objectives.
\newblock In {\em NIPS}, 2016.

\bibitem{zhu-2017-stoc}
Zeyuan Allen-Zhu.
\newblock Katyusha: The first direct acceleration of stochastic gradient
  methods.
\newblock In {\em STOC}, 2017.

\bibitem{Bertsekas-book}
Dimitri Bertsekas.
\newblock {\em Nonlinear Programming}.
\newblock Athena Scientific, Belmont, Ma, 1999.

\bibitem{Necoara-2019}
I.~Necoara, Yu. Nesterov, and F.~Glineur.
\newblock Linear convergence of first order methods for non-strongly convex
  optimization.
\newblock {\em Mathematical Programming}, 175(1-2):69--107, 2019.

\bibitem{ma2016eb}
Chenxin Ma, Rachael Tappenden, and Martin Tak\'{a}\u{c}.
\newblock Linear convergence of randomized feasible descent methods under the
  weak strong convexity assumption.
\newblock {\em Journal of Machine Learning Research}, 230(17):1--24, 2016.

\bibitem{errorbound2}
Zhiquan Luo and Paul Tseng.
\newblock On the linear convergence of descent methods for convex essentially
  smooth minimization.
\newblock {\em SIAM J. on Control and Optimization}, 30(2):408--425, 1992.

\bibitem{lewis2018}
Dmitriy Drusvyatskiy and Adrian~S. Lewis.
\newblock Error bounds, quadratic growth, and linear convergence of proximal
  methods.
\newblock {\em Mathematics of Operations Research}, 2018.

\bibitem{lin2014}
Powei Wang and Chih-Jen Lin.
\newblock Iteration complexity of feasible descent methods for convex
  optimization.
\newblock {\em Journal of Machine Learning Research}, 15:1523--1548, 2014.

\bibitem{Nesterov-2004}
Yurii Nesterov.
\newblock {\em Introductory {L}ectures on {C}onvex {O}ptimization}.
\newblock Springer Science $\&$ Business Media, 2004.

\bibitem{lee2013}
Yin~Tat Lee and Aaron Sidford.
\newblock Efficient accelerated coordinate descent methods and faster
  algorithms for solving linear systems.
\newblock In {\em FOCS}, 2013.

\bibitem{Lewis-2013-newton}
Adrian~S. Lewis and Michael~L. Overton.
\newblock Nonsmooth optimization via quasi-newton methods.
\newblock {\em Mathematical Programming}, 141(1-2):135--163, 2013.

\bibitem{yin-DR2017}
Damek Davis and Wotao Yin.
\newblock Convergence rate analysis of several splitting schemes.
\newblock {\em Splitting Methods in Communication, Imaging, Science, and
  Engineering}, pages 115--163, 2017.

\bibitem{nesterov-2012-gradient}
Yurii Nesterov.
\newblock How to make the gradients small.
\newblock {\em Optima}, 88, 2012.

\bibitem{Bolte2017}
J\'{e}r\^{o}me Bolte, Trong~Phong Nguyen, Juan Peypouquet, and Bruce~W. Suter.
\newblock From error bounds to the complexity of first-order descent methods
  for convex functions.
\newblock {\em Mathematical Programming}, 165(2):471--507, 2017.

\bibitem{guoyin2013}
Guoyin Li.
\newblock Global error bounds for piecewise convex polynomials.
\newblock {\em Mathematical Programming}, 137(1-2):37--64, 2013.

\bibitem{yang2009}
Weihong Yang.
\newblock Error bounds for convex polynomials.
\newblock {\em SIAM J. on Optimization}, 19(4):1633--1647, 2009.

\bibitem{tianbao-2017}
Mingrui Liu and Tianbao Yang.
\newblock Adaptive accelerated gradient converging method under h\"{o}lderian
  error bound condition.
\newblock In {\em NIPS}, 2017.

\bibitem{Donoghue-2015-NesRestart}
B.~O'Donoghue and E.~Cand\`{e}s.
\newblock Adaptive restart for accelerated gradient schemes.
\newblock {\em Foundations of Computational Mathematics}, 15(3):715--732, 2015.

\bibitem{zhengqu-2017}
Olivier Fercoq and Zheng Qu.
\newblock Adaptive restart of accelerated gradient methods under local
  quadratic growth condition.
\newblock In {\em arxiv:1709.02300}, 2017.

\bibitem{Bousquet2002}
Olivier Bousquet and Andr\'{e} Elisseeff.
\newblock Stability and generalization.
\newblock {\em Journal of Machine Learning Research}, (2):499--526, 2002.

\bibitem{cs-2006}
Emmanuel Candes, Justin Romberg, and Terence Tau.
\newblock Stable signal recovery from incomplete and inaccurate measurements.
\newblock {\em Communications on Pure and Applied Mathematics},
  59(8):1207--1223, 2006.

\end{thebibliography}

\end{document}